\pdfoutput=1
\RequirePackage{silence}
\documentclass[a4paper,11pt]{amsart}
\usepackage[hmarginratio={1:1},vmarginratio={1:1},lmargin=60.0pt,tmargin=60.0pt]{geometry}

\allowdisplaybreaks

\usepackage[numbers]{natbib}
\usepackage[utf8]{inputenc}

\usepackage{latexsym,exscale,mathtools,textcomp}
\usepackage{amssymb,amsmath,amsthm,amsfonts,mathrsfs,bbm,enumitem,stmaryrd}
\usepackage{blkarray}
\usepackage[table]{xcolor}
\usepackage{graphicx}
\usepackage{mathtools}
\usepackage{ytableau}

\usepackage{xparse}

\usepackage{dynkin-diagrams}

\setlist[enumerate]{itemsep=0.15cm,label=\emph{\upshape(\alph*)}}
\setlist[enumerate,2]{itemsep=0.15cm,label=\emph{\upshape(\roman*)}}

\usepackage{array}
\newcolumntype{C}{>{$}c<{$}}

\definecolor{mygray}{gray}{0.6}
\definecolor{mygraydark}{gray}{0.4}
\definecolor{mygraylight}{gray}{0.85}
\definecolor{spinach}{RGB}{46,139,87}
\definecolor{tomato}{RGB}{255,99,71}
\definecolor{orchid}{RGB}{143,40,194}
\definecolor{neon}{RGB}{77,77,255}
\definecolor{pumpkin}{RGB}{224,180,80}
\definecolor{citron}{RGB}{190,180,90}

\definecolor{lava}{RGB}{207,16,32}
\definecolor{cream}{RGB}{255,253,208}
\definecolor{verdigris}{RGB}{67,179,174}
\definecolor{Black}{RGB}{0,0,0}
\definecolor{mydarkblue}{RGB}{10,10,170}
\definecolor{darkspinach}{RGB}{20,70,20}
\definecolor{darktomato}{RGB}{155,40,30}
\definecolor{darkorchid}{RGB}{50,10,100}
\definecolor{darklava}{RGB}{150,8,16}

\usepackage{todonotes}

\let\emph\relax
\DeclareTextFontCommand{\emph}{\bfseries\em}

\newcommand{\placeholder}{{}_{-}}

\renewcommand{\dots}{\text{...}}
\newcommand{\mystrut}{\rule[-0.2\baselineskip]{0pt}{0.9\baselineskip}}

\let\<=\langle
\let\>=\rangle

\renewcommand{\dots}{\text{...}}
\renewcommand{\vdots}{\rotatebox{90}{\text{...}}}
\renewcommand{\ddots}{\raisebox{0.25cm}{\rotatebox{-45}{\text{...}}}}

\DeclarePairedDelimiterX{\set}[1]{\{}{\}}{\setargs{#1}}
\NewDocumentCommand{\setargs}{>{\SplitArgument{1}{|}}m}{\setargsaux#1}
\NewDocumentCommand{\setargsaux}{mm}
{\IfNoValueTF{#2}{#1} {#1\,\delimsize|\,\mathopen{}#2}}

\newcommand{\ie}{\text{i.e.}}
\newcommand{\eg}{\text{e.g.}}

\newcommand{\etc}{\text{etc.}}

\newcommand{\C}{\mathbb{C}}

\newcommand{\N}{\mathbb{Z}_{\geq 0}}

\newcommand{\Z}{\mathbb{Z}}
\newcommand{\K}{\mathbb{K}}

\newcommand{\chark}{\mathrm{char}(\mathbb{K})}

\newcommand{\mrk}[1][\K]{\mathrm{rank}_{#1}}
\renewcommand{\dim}[1][\K]{\mathrm{dim}_{#1}}

\newcommand{\End}{\mathrm{End}}
\newcommand{\Hom}{\mathrm{Hom}}

\newcommand{\algebra}[1][A]{\mathscr{#1}}

\newcommand{\Pcal}{\mathcal{P}}
\newcommand{\apex}{\mathcal{P}^{ap}}
\newcommand{\Tcal}{\mathcal{T}}
\newcommand{\Bcal}{\mathcal{B}}
\newcommand{\cellbasis}[1][\mathscr{A}]{B_{\mathscr{A}}}
\newcommand{\calg}[1][\lambda]{\mathscr{A}_{#1}}
\newcommand{\sandorder}[1][\Pcal]{<_{#1}}
\newcommand{\rsandorder}[1][\Pcal]{>_{#1}}
\newcommand{\gsandorder}[1][\Pcal]{\geq_{#1}}
\newcommand{\lsandorder}[1][\Pcal]{\leq_{#1}}
\newcommand{\sand}[1][\lambda]{\mathscr{H}_{#1}}
\newcommand{\sandbasis}[1][\lambda]{B_{#1}}

\newcommand{\gmatrix}[1][\lambda]{\mathrm{G}^{#1}}

\newcommand{\dmod}[1][{\lambda,K}]{\Delta(#1)}
\newcommand{\modd}[1][{\lambda,K}]{\nabla(#1)}
\newcommand{\lmod}[1][{\lambda,K}]{L(#1)}

\newcommand{\usummand}{\inplus}

\newcommand{\rad}[1][\jcell]{\mathrm{Rad}}

\newcommand{\lcell}{\mathcal{L}}
\newcommand{\rcell}{\mathcal{R}}
\newcommand{\jcell}{\mathcal{J}}
\newcommand{\hcell}{\mathcal{H}}

\newcommand{\jideal}[1][\lambda]{\algebra^{>_{lr}\lambda}}

\newcommand{\module}[1][M]{{#1}}

\newcommand{\monoid}[1][S]{\mathrm{#1}}

\newcommand{\onemon}{\mathrm{1}}

\newcommand{\xmon}[2][n]{\mathrm{D}_{#2}(#1)}
\newcommand{\mxmon}[2][n]{\mathrm{MD}_{#2}(#1)}

\newcommand{\sym}[1][n]{\mathrm{S}_{#1}}
\newcommand{\psym}[1][n]{\mathrm{S}_{#1}^{p}}
\newcommand{\msym}[1][n]{\mathrm{MS}_{#1}}
\newcommand{\mpsym}[1][n]{\mathrm{MS}_{#1}^{p}}

\newcommand{\pamon}[1][n]{\mathrm{Pa}_{1}(#1)}
\newcommand{\robrmon}[1][n]{\mathrm{RoBr}_{1}(#1)}
\newcommand{\romon}[1][n]{\mathrm{Ro}_{1}(#1)}
\newcommand{\ppamon}[1][n]{\mathrm{Pa}_{1}^{p}(#1)}
\newcommand{\probrmon}[1][n]{\mathrm{RoBr}_{1}^{p}(#1)}
\newcommand{\brmon}[1][n]{\mathrm{Br}_{1}(#1)}
\newcommand{\tlmon}[1][n]{\mathrm{TL}_{1}(#1)}
\newcommand{\momon}[1][n]{\mathrm{Mo}_{1}(#1)}
\newcommand{\promon}[1][n]{\mathrm{Ro}_{1}^{p}(#1)}
\newcommand{\pbrmon}[1][n]{\mathrm{Br}_{1}^{p}(#1)}

\newcommand{\para}{\delta}
\newcommand{\paraa}{\alpha}
\newcommand{\parab}{\beta}
\newcommand{\parac}{\gamma}
\newcommand{\parap}{\delta_{+}}
\newcommand{\param}{\delta_{-}}

\newcommand{\mpamon}[2][n]{\mathrm{MPa}_{#2}(#1)}
\newcommand{\mrobrmon}[2][n]{\mathrm{MRoBr}_{#2}(#1)}
\newcommand{\mromon}[2][n]{\mathrm{MRo}_{#2}(#1)}
\newcommand{\mppamon}[2][n]{\mathrm{MPa}_{#2}^{p}(#1)}

\newcommand{\mbrmon}[2][n]{\mathrm{MBr}_{#2}(#1)}
\newcommand{\mtlmon}[2][n]{\mathrm{MTL}_{#2}(#1)}
\newcommand{\mmomon}[2][n]{\mathrm{MMo}_{#2}(#1)}
\newcommand{\mpromon}[2][n]{\mathrm{MRo}_{#2}^{p}(#1)}

\newcommand{\In}{\mathcal{I}}

\newcommand{\pacat}[1][1]{\mathbf{Pa}_{#1}}
\newcommand{\robrcat}[1][1]{\mathbf{RoBr}_{#1}}

\newcommand{\ppacat}[1][1]{\mathbf{Pa}_{#1}^{p}}

\newcommand{\brcat}[1][1]{\mathbf{Br}_{#1}}
\newcommand{\tlcat}[1][1]{\mathbf{TL}_{#1}}
\newcommand{\mocat}[1][1]{\mathbf{Mo}_{#1}}

\newcommand{\xcat}[1][1]{\mathbf{D}_{#1}}
\newcommand{\mxcat}[1][1]{\mathbf{MD}_{#1}}
\newcommand{\dashmxcat}[1][1]{\mathbf{MD}_{#1}'}

\newcommand{\mpcat}[1]{\mathbf{MPa}_{#1}}
\newcommand{\mvcat}[1]{\underline{\mathbf{MPa}}_{#1}}
\newcommand{\delcat}[1]{\mathbf{Rep}(S_{#1})}
\newcommand{\mdelcat}[3]{\mathbf{MRep}(S_{#1, #2, #3})}
\newcommand{\cob}{\mathbf{2Cob}}
\newcommand{\mcob}[1]{\mathbf{MCob}_{#1}}

\newcommand{\mucob}[3]{\mathbf{MCob}_{#1, #2, #3}}

\newcommand{\tr}{\mathrm{Tr}}

\usepackage[all]{xy}
\usepackage{tikz}
\usetikzlibrary{cd}
\usetikzlibrary{decorations}
\usetikzlibrary{decorations.markings}
\usetikzlibrary{decorations.pathreplacing}
\usetikzlibrary{decorations.pathmorphing}
\usetikzlibrary{arrows.meta,shapes,positioning,matrix,calc}
\usetikzlibrary{shapes.callouts}
\usetikzlibrary{shapes.arrows}
\usetikzlibrary{tqft}
\usetikzlibrary{scopes,intersections}

\tikzset{
anchorbase/.style={baseline={([yshift=#1]current bounding box.center)}},
anchorbase/.default={-0.5ex},
tinynodes/.style={font=\tiny,text height=0.25ex,text depth=0.05ex},
smallnodes/.style={font=\scriptsize,text height=0.75ex,text depth=0.15ex},
mor/.style={line width=0.75,color=black,fill=cream},
mor2/.style={line width=0.75,color=black,fill=tomato},
mor3/.style={line width=0.75,color=black,fill=spinach},
usual/.style={line width=1.2,color=black},
crossline/.style={preaction={draw=white,line width=5.0pt,-},preaction={draw=black,line width=0.9pt,-}},
hol/.style = {
decoration={markings,
post length=0.25mm,
pre length=0.25mm,
mark=at position #1 with {\node[circle,radius=0.15cm,inner sep=-1.2pt,draw,color=black,fill=white]{};}
},
postaction={decorate}
},
mob/.style = {
decoration={markings,
post length=0.25mm,
pre length=0.25mm,
mark=at position #1 with {\node[circle,radius=0.15cm,inner sep=-1.2pt,draw, color=tomato,fill=tomato]{};}
},
postaction={decorate}
},
dot/.style = {
decoration={markings,
post length=0.25mm,
pre length=0.25mm,
mark=at position #1 with {\node[circle,radius=0.15cm,inner sep=-1.2pt,color=black,fill=black]{};}
},
postaction={decorate}
},
dot/.default=1,
}
\tikzstyle directed=[postaction={decorate,decoration={markings,
mark=at position #1 with {\arrow[line width=0.25mm, black]{>}}}}]

\let\oldlightning\lightning
\renewcommand{\lightning}{\textcolor{tomato}{\pmb{\oldlightning}}}

\usepackage{aliascnt,etoolbox}
\def\NewTheorem#1{%
\newaliascnt{#1}{equation}%
\newtheorem{#1}[#1]{#1}%
\aliascntresetthe{#1}%
\expandafter\def\csname #1autorefname\endcsname{#1}%
}
\def\equationautorefname~#1\null{(#1)\null}

\numberwithin{equation}{subsection}

\NewTheorem{Proposition}
\NewTheorem{Theorem}
\NewTheorem{Corollary}
\AtEndEnvironment{Corollary}{\null\hfill$\square$}%
\NewTheorem{Lemma}
\theoremstyle{definition}
\NewTheorem{Definition}
\AtEndEnvironment{Definition}{\null\hfill$\Diamond$}%
\NewTheorem{Notation}
\AtEndEnvironment{Notation}{\null\hfill$\Diamond$}%
\NewTheorem{Example}
\AtEndEnvironment{Example}{\null\hfill$\Diamond$}%
\theoremstyle{remark}
\NewTheorem{Remark}
\AtEndEnvironment{Remark}{\null\hfill$\Diamond$}%

\usepackage[hypertexnames=false]{hyperref}
\usepackage{bookmark}
\hypersetup{
pdftoolbar=true,
pdfmenubar=true,
pdffitwindow=false,
pdfstartview={FitH},
pdftitle={Möbius Strip Diagram Algebras},
pdfauthor={Daniel W. Collison and Daniel Tubbenhauer},
pdfsubject={},
pdfcreator={Daniel W. Collison and Daniel Tubbenhauer},
pdfproducer={Daniel W. Collison and Daniel Tubbenhauer},
pdfkeywords={},
pdfnewwindow=true,
colorlinks=true,
linkcolor=mydarkblue,
citecolor=teal,
filecolor=magenta,
urlcolor=orchid,
linkbordercolor=lava,
citebordercolor=teal,
urlbordercolor=orchid,
linktocpage=true
}

\newcommand{\nnfootnote}[1]{%
\begin{NoHyper}
\renewcommand\thefootnote{}\footnote{#1}%
\addtocounter{footnote}{-1}%
\end{NoHyper}
}

\def\makeautorefname#1#2{\csdef{#1autorefname}{#2}}

\makeautorefname{section}{Section}%
\makeautorefname{subsection}{Section}%
\makeautorefname{subsubsection}{Section}%

\begin{document}
\title[Möbius Strip Diagram Algebras]{Möbius Strip Diagram Algebras}
\author[D. W. Collison and D. Tubbenhauer]{Daniel W. Collison and Daniel Tubbenhauer}

\address{D.W.C.: The University of Sydney, School of Mathematics and Statistics F07, NSW 2006, Australia}
\email{daniel.collison@sydney.edu.au}

\address{D.T.: The University of Sydney, School of Mathematics and Statistics F07, Office Carslaw 827, NSW 2006, Australia, \href{http://www.dtubbenhauer.com}{www.dtubbenhauer.com}, https://orcid.org/0000-0001-7265-5047}
\email{daniel.tubbenhauer@sydney.edu.au}

\begin{abstract}
We introduce Möbius strip diagram algebras (and their monoid and categorical versions) as subalgebras of a partition-style diagram calculus in which strands may carry handles and Möbius strip features.
We identify the resulting diagram category with a linear quotient of a nonorientable two-dimensional cobordism category. Finally, we develop the associated cell theory and use it to classify the simple modules and compute dimensions in a range of cases.
\end{abstract}

\nnfootnote{\textit{Mathematics Subject Classification 2020.} Primary: 20M30, 57R56; Secondary: 16G10, 18M05.}
\nnfootnote{\textit{Keywords.} Möbius strip; diagram algebras; partition monoids; unoriented cobordisms; two dimensional topological quantum field theories; sandwich cellular algebras; Gram matrices; simple modules.}

\addtocontents{toc}{\protect\setcounter{tocdepth}{1}}

\maketitle

\tableofcontents


\section{Introduction}\label{S:Introduction}


Diagrammatic algebra is, at heart, the idea that complicated algebraic operations can be made local by drawing them.
Morphisms live in rectangles, composition is stacking, and the tensor product is juxtaposition.
Such a viewpoint has paid for itself many times over: it turns representation-theoretic questions into combinatorics of pictures,
and from topological constructions, it produces algebraic structures for which feasible computations are possible.
Classical examples include the Temperley--Lieb algebra, governing the SL2 story \cite{RuTeWe-sl2},
and the Brauer algebra, which plays the analogous role for the orthogonal and symplectic groups \cite{Br-brauer-algebra-original}.

A particularly robust source of diagram categories uses (1+1)-dimensional (or 2D) cobordisms, especially when passing to their spines, for example:
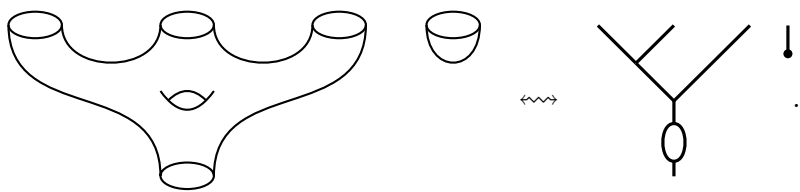
\begin{gather*}
\begin{tikzpicture}[
baseline ={([yshift=-.5ex]current bounding box.center)},anchorbase]
\pic[
tqft,
incoming boundary components=3,
outgoing boundary components=1,
every lower boundary component/.style={draw},
genus=1,
offset = 1,
draw,
thick,
name=ex1
];
\pic[tqft,
incoming boundary components=1,
outgoing boundary components=0,
draw,
thick,
name=cyldoub1,
every lower boundary component/.style={draw},
at={(5.5,0)},
];
\end{tikzpicture} \quad
\leftrightsquigarrow \quad 
\begin{tikzpicture}[anchorbase]
\draw[usual] (0,0) to (1,1);
\draw[usual] (0,0) to (-0.5,0.5);
\draw[usual] (-0.5,0.5) to (-1,1);
\draw[usual] (-0.5,0.5) to (0,1);
\draw[usual] (0,0) to (0,-0.3);
\draw[usual] (0,-1) to (0,-0.8);
\draw[usual] (0,-0.3) to[out=0,in=0] (0,-0.8);
\draw[usual] (0,-0.3) to[out=180,in=180] (0,-0.8);
\draw[usual,dot] (1.5,1) to (1.5,0.6);
\end{tikzpicture}
.
\end{gather*}
On the one hand, cobordism categories are universal for topological field theories: a 2D TQFT is a symmetric monoidal functor out of a suitable cobordism category.
On the other hand, once local relations are imposed, such as evaluations of closed components, neck-cutting-type relations, handle relations, and so on, as in
\cite{BlHaMaVo-tqft-kauffman-bracket,BaNa-tangles-cobordisms,KhSa-cobordisms}, familiar algebraic structures are commonly recovered:
partition-style categories and their associated Deligne-type interpolation categories, as well as other diagram algebras with rich cellular structure; this is well-known, for newer perspectives see \eg\, \cite{Co-jelly,KhSa-cobordisms,KhOsKo-cobordisms}.

Moreover, many common diagram calculi, while not direct offspring of 2D TQFTs, are driven by the same locality philosophy and are built from strikingly similar
primitives (merges, splits, cups, caps, and their relatives).
Examples include diagrammatic approaches to tensor categories (as summarized e.g. in \cite{TuVi-monoidal-tqft,Tu-qt}),
Soergel diagrams as e.g. in \cite{ElKh-diagrams-soergel}, and webs as e.g. in \cite{Ku-spiders-rank-2} (in particular, for SO3 \cite{Ya-invariant-graphs,MoPeSn-categories-trivalent-vertex,Tu-web-reps}),
as well as representation-theoretic structures related to cellular algebras, see e.g. \cite{AnStTu-cellular-tilting,BeTh-hwt-findim-algebras}.
There are many more examples, far too many to list here; in particular, related diagrammatics also appear in formalisms for embedded graphs/ribbon graphs on surfaces and their categorifications.
A common feature in these graph-on-surface settings is the introduction of half-twists (or when working in genuinely nontrivial surface topology) resulting in nonorientability.

In this paper, we incorporate nonorientability into the 2D TQFT setting in the easiest possible way: by adjoining M\"obius strips,
which we indicate in spine pictures by dots
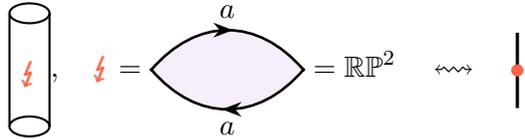
\begin{gather*}
\begin{tikzpicture}[baseline ={([yshift=-.5ex]current bounding box.center)}, scale=0.25]
\pic[
tqft,
incoming boundary components=1,
outgoing boundary components=1,
every lower boundary component/.style={draw},
genus=0,
scale=0.75,
draw,
thick,
boundary separation=40pt,
name=pop
]; 
\node at ([xshift=0pt, yshift=80pt]pop-outgoing boundary 1) {$\lightning$};
\end{tikzpicture},
\quad
\lightning
=
\raisebox{-0.1cm}{$\begin{tikzpicture}[scale=0.5, line cap=round, line join=round,anchorbase]
\coordinate (L) at (-2,0);
\coordinate (R) at ( 2,0);
\fill[orchid!8]
(L) .. controls (-0.8, 1.4) and (0.8, 1.4) .. (R)
.. controls (0.8,-1.4) and (-0.8,-1.4) .. cycle;
\draw[line width=1.0pt]
(L) .. controls (-0.8, 1.4) and (0.8, 1.4) .. (R);
\draw[line width=1.0pt]
(L) .. controls (-0.8,-1.4) and (0.8,-1.4) .. (R);
\draw[line width=1.0pt, -{Stealth[length=2.6mm]}]
(-0.15,1.05) -- (0.15,1.05);
\draw[line width=1.0pt, -{Stealth[length=2.6mm]}]
(0.15,-1.05) -- (-0.15,-1.05);
\node at (0, 1.55) {$a$};
\node at (0,-1.55) {$a$};
\fill (L) circle (1.3pt);
\fill (R) circle (1.3pt);
\node (B) at (0,-2.15) {};
\end{tikzpicture}$}
=\mathbb{RP}^2 \quad
\leftrightsquigarrow \quad
\begin{tikzpicture}[anchorbase]
\draw[usual, mob=0.5] (0,0) to (0,1);
\end{tikzpicture}
.
\end{gather*}
A new family of diagrammatic objects is then born: nonorientable diagram monoids, which we call \emph{M\"obius strip diagram monoids or algebras or categories}.
They are very much in the spirit of ``decorated'' diagram algebras, akin to blob algebras \cite{MaSa-blob} or dotted Temperley--Lieb algebras as in e.g.
\cite{Gr-dots-tl,DeTu-dicyclic}, but now in a partition-style setting.
Our aim is to show that they form a natural extension of the existing diagrammatic zoo, while still supporting a tractable internal structure and representation theory.

\subsection*{Why nonorientability appears even when you did not ask for it}

There is a recurring phenomenon in diagrammatics and adjacent fields.
To keep the discussion concrete, let us use low dimensional topology, in particular knot theory and topological aspects of graph theory, as guiding examples:
if one insists on locality (relations supported in small disks) and enlarges the ambient class of diagrams (for instance, by allowing knots or graphs on surfaces),
then unoriented and nonorientable cobordisms tend to appear whether planned or not.

A clean instance is the extension of Bar-Natan's geometric formalism \cite{BaNa-tangles-cobordisms} to link diagrams on surfaces: when the surface itself is allowed to vary up to stabilization (adding/removing handles away from the diagram),
the natural cobordisms between resolutions are not always orientable, and the corresponding TQFT input must be upgraded from ``oriented'' to ``unoriented''.
The same pattern emerges in \cite{TuTu-utqft}, where unoriented (1+1)-TQFTs are introduced and classified (in terms of a Frobenius algebra with extra structure)
in order to produce link homology theories for stable equivalence classes; see also \cite{Ta-uhqft,Tu-virtual-khovanov}.
A graph-theoretic instance of this ``unoriented upgrade'' is the TQFT approach to graph coloring as e.g. in \cite{BaMc-graph-coloring}.
More generally, half-twist operations on edges of embedded graphs provide a basic mechanism by which nonorientability enters the topology of ribbon graphs; see e.g. \cite{EMM-twisted-duality}.

From the diagrammatic viewpoint, ``nonorientability'' manifests itself simply by recording when a local piece of surface has undergone a half-twist.
In a planar calculus, half-twists are naturally encoded by considering a \emph{M\"obius strip} as an elementary decoration.
Similarly, stabilization by handles necessitates introducing a \emph{handle} as another basic local feature.
The result is a calculus in which strands may carry handles and M\"obius strips, subject to local relations that model the corresponding cobordism moves.

\begin{Remark}\label{rm:nophi}
For the reader familiar with \cite{TuTu-utqft}: we do not use the twist itself as a gluing operation, which, in our opinion, makes the diagrammatics easier though a small amount of information is lost essentially amounting to a factor $\Z/2\Z$, cf. \autoref{T:Deligne} with \cite[Theorem II]{Cz-mobius-tqft}.
\end{Remark}

\subsection*{Main results (informal)}

We now summarize the contributions of the paper; precise statements appear in the body.

\begin{enumerate}
\item \textbf{Definition and normal forms.}
We begin by defining our main objects, M\"obius strip diagram categories/monoids/algebras, as a partition-style calculus with handle and M\"obius generators and natural local relations.
We give an explicit spanning set and convenient normal forms enabling effective computations.

\item \textbf{Cobordism identification.}
We make precise the topological motivations by constructing a symmetric monoidal functor from a nonorientable 2D cobordism category to our diagram category, which induces an equivalence after imposing the defining relations.
Equivalently, our diagram category can be viewed as a linear quotient of nonorientable cobordisms.

\item \textbf{Cell theory and simple modules.}
We develop the associated cell theory and describe the corresponding cell modules.
We then use the developed theory to classify the simple modules and determine their parametrization.

\item \textbf{Dimensions and examples.}
We compute dimensions for a range of simples/cell modules and work out illustrative low-rank examples, highlighting how the new nonorientable features
modify familiar partition-type answers.
\end{enumerate}

\noindent\textbf{Acknowledgments.}
This paper is part of the first author's PhD thesis.

DC was supported by the Commonwealth through an Australian Government Research Training Program Scholarship [https://doi.org/10.82133/C42F-K220]. DT is supported by the ARC Future Fellowship FT230100489, and any remaining errors are an inherent feature of a one-sided approach.


\section{Möbius strip diagram monoids and categories}\label{M:Monoid}


Below and throughout, we assume some familiarity with the language of monoidal categories and their diagrammatics as e.g. in \cite{EtGeNiOs-tensor-categories} or in \cite{Tu-qt}.

We first consider the \emph{partition monoid} $\pamon$ of all diagrams of partitions of a $2n$-element set, and the respective planar submonoid $\ppamon$ of $\pamon$, 
which arise as endomorphism monoids in the respective monoidal categories $\ppacat$ and $\pacat$. More specifically, the \emph{partition category} $\pacat$ has as objects nonnegative integers and morphisms from $n$ to $m$ are partition diagrams corresponding to partitions of $\{1, \dots, n\}\cup\{1', \dots, m'\}$. The following example summarizes our conventions; see \eg \, \cite{Jo-potts-sym,Ma-potts-sym,HaRa-partition-algebras} for a thorough treatment.

\begin{Example}\label{E:Partitions}
The elements of $\pamon[6]$ represent partitions of $\{1,\dots,6\}\cup\{1',\dots,6'\}$, with the primed numbers corresponding to the number of strands at the top end of the respective partition diagrams. For example:
\begin{gather*}
\big\{
\{1,2'\},\{2,4,5\},\{3,3'\},\{6,1',4',6'\},\{5'\}
\big\}\leftrightsquigarrow
a=
\begin{tikzpicture}[anchorbase]
\draw[usual] (0.5,0) to[out=90,in=180] (1.25,0.45) to[out=0,in=90] (2,0);
\draw[usual] (0.5,0) to[out=90,in=180] (1,0.35) to[out=0,in=90] (1.5,0);
\draw[usual] (0,1) to[out=270,in=180] (0.75,0.55) to[out=0,in=270] (1.5,1);
\draw[usual] (1.5,1) to[out=270,in=180] (2,0.55) to[out=0,in=270] (2.5,1);
\draw[usual] (0,0) to (0.5,1);
\draw[usual] (1,0) to (1,1);
\draw[usual] (2.5,0) to (2.5,1);
\draw[usual,dot] (2,1) to (2,0.8);
\end{tikzpicture}
\in\pamon[6]
,\\
\big\{
\{1,1'\},\{2,4,5\},\{3\},\{6,2',4',6'\},\{3'\},\{5'\}
\big\}\leftrightsquigarrow
b=
\begin{tikzpicture}[anchorbase]
\draw[usual] (0.5,0) to[out=90,in=180] (1.25,0.45) to[out=0,in=90] (2,0);
\draw[usual] (0.5,0) to[out=90,in=180] (1,0.35) to[out=0,in=90] (1.5,0);
\draw[usual] (0.5,1) to[out=270,in=180] (1,0.55) to[out=0,in=270] (1.5,1);
\draw[usual] (1.5,1) to[out=270,in=180] (2,0.55) to[out=0,in=270] (2.5,1);
\draw[usual] (0,0) to (0,1);
\draw[usual] (2.5,0) to (2.5,1);
\draw[usual,dot] (1,0) to (1,0.2);
\draw[usual,dot] (1,1) to (1,0.8);
\draw[usual,dot] (2,1) to (2,0.8);
\end{tikzpicture}
\in\ppamon[6].
\end{gather*}

Composition $\circ$ in $\pacat$ is vertical concatenation of diagrams and removal of internal components. The monoidal product $\otimes$ is horizontal juxtaposition. Our conventions are:
\begin{gather*}
a\circ b
=
\begin{tikzpicture}[anchorbase]
\draw[usual] (0.5,1) to[out=90,in=180] (1.25,1.45) to[out=0,in=90] (2,1);
\draw[usual] (0.5,1) to[out=90,in=180] (1,1.35) to[out=0,in=90] (1.5,1);
\draw[usual] (0,2) to[out=270,in=180] (0.75,1.55) to[out=0,in=270] (1.5,2);
\draw[usual] (1.5,2) to[out=270,in=180] (2,1.55) to[out=0,in=270] (2.5,2);
\draw[usual] (0,1) to (0.5,2);
\draw[usual] (1,1) to (1,2);
\draw[usual] (2.5,1) to (2.5,2);
\draw[usual,dot] (2,2) to (2,1.8);
\draw[usual] (0.5,0) to[out=90,in=180] (1.25,0.45) to[out=0,in=90] (2,0);
\draw[usual] (0.5,0) to[out=90,in=180] (1,0.35) to[out=0,in=90] (1.5,0);
\draw[usual] (0.5,1) to[out=270,in=180] (1,0.55) to[out=0,in=270] (1.5,1);
\draw[usual] (1.5,1) to[out=270,in=180] (2,0.55) to[out=0,in=270] (2.5,1);
\draw[usual] (0,0) to (0,1);
\draw[usual] (2.5,0) to (2.5,1);
\draw[usual,dot] (1,0) to (1,0.2);
\draw[usual,dot] (1,1) to (1,0.8);
\draw[usual,dot] (2,1) to (2,0.8);
\end{tikzpicture}
=
\begin{tikzpicture}[anchorbase]
\draw[usual] (0.5,0) to[out=90,in=180] (1.25,0.45) to[out=0,in=90] (2,0);
\draw[usual] (0.5,0) to[out=90,in=180] (1,0.35) to[out=0,in=90] (1.5,0);
\draw[usual] (0,1) to[out=270,in=180] (0.75,0.55) to[out=0,in=270] (1.5,1);
\draw[usual] (1.5,1) to[out=270,in=180] (2,0.55) to[out=0,in=270] (2.5,1);
\draw[usual] (0,0) to (0.5,1);
\draw[usual] (2.5,0) to (2.5,1);
\draw[usual,dot] (1,0) to (1,0.2);
\draw[usual,dot] (1,1) to (1,0.8);
\draw[usual,dot] (2,1) to (2,0.8);
\end{tikzpicture}
,\quad
a\otimes b=
\begin{tikzpicture}[anchorbase]
\draw[usual] (0.5,0) to[out=90,in=180] (1.25,0.45) to[out=0,in=90] (2,0);
\draw[usual] (0.5,0) to[out=90,in=180] (1,0.35) to[out=0,in=90] (1.5,0);
\draw[usual] (0,1) to[out=270,in=180] (0.75,0.55) to[out=0,in=270] (1.5,1);
\draw[usual] (1.5,1) to[out=270,in=180] (2,0.55) to[out=0,in=270] (2.5,1);
\draw[usual] (0,0) to (0.5,1);
\draw[usual] (1,0) to (1,1);
\draw[usual] (2.5,0) to (2.5,1);
\draw[usual,dot] (2,1) to (2,0.8);
\draw[usual] (3.5,0) to[out=90,in=180] (4.25,0.45) to[out=0,in=90] (5,0);
\draw[usual] (3.5,0) to[out=90,in=180] (4,0.35) to[out=0,in=90] (4.5,0);
\draw[usual] (3.5,1) to[out=270,in=180] (4,0.55) to[out=0,in=270] (4.5,1);
\draw[usual] (4.5,1) to[out=270,in=180] (5,0.55) to[out=0,in=270] (5.5,1);
\draw[usual] (3,0) to (3,1);
\draw[usual] (5.5,0) to (5.5,1);
\draw[usual,dot] (4,0) to (4,0.2);
\draw[usual,dot] (4,1) to (4,0.8);
\draw[usual,dot] (5,1) to (5,0.8);
\end{tikzpicture}
,
\end{gather*}
where we simplified the resulting diagram on the left by noting that partition diagrams are considered equivalent if and only if they define the same partition. 

As another example, the composition of diagrams
\begin{gather*}
\Hom_{\pacat}(2, 3) \ni
\begin{tikzpicture}[anchorbase]
\draw[usual, dot] (1,0) to (1,0.2);
\draw[usual] (0,1) to (0.5,0);
\draw[usual] (0, 1) to[out=315, in=215] (0.5,1);
\draw[usual, dot] (1, 1) to (1,0.8);
\end{tikzpicture}
\quad \circ 
\quad
\begin{tikzpicture}[anchorbase]
\draw[usual] (0,0) to (0.5, 1);
\draw[usual, dot] (1, 1) to (1, 0.8);
\draw[usual] (0.5, 0) to[out=90, in=90] (1.5, 0);
\draw[usual, dot] (1, 0) to (1, 0.2);
\end{tikzpicture}
\in \Hom_{\pacat}(4, 2),
\end{gather*}
corresponding to $\big\{\{1, 1', 2'\}, \{2\}, \{3'\}\big\}$ and $\big\{\{1, 1'\}, \{2, 4\}, \{3\}, \{2'\}\big\}$ is given by 
\begin{gather*}
\Hom_{\pacat}(4, 3) \ni
\begin{tikzpicture}[anchorbase]
\draw[usual] (0,0) to (0.5, 1);
\draw[usual, dot] (1, 1) to (1, 0.8);
\draw[usual] (0.5, 0) to[out=90, in=90] (1.5, 0);
\draw[usual, dot] (1, 0) to (1, 0.2);
\draw[usual, dot] (1,1) to (1,1.2);
\draw[usual] (0,2) to (0.5,1);
\draw[usual] (0, 2) to[out=315, in=215] (0.5,2);
\draw[usual, dot] (1, 2) to (1,1.8);
\draw[very thick,densely dotted,tomato] (0,1) to (1.5, 1);
\end{tikzpicture}
=
\begin{tikzpicture}[anchorbase]
\draw[usual] (0,0) to (0.5, 1);
\draw[usual] (0.5, 0) to[out=90, in=90] (1.5, 0);
\draw[usual, dot] (1, 0) to (1, 0.2);
\draw[usual] (0.5, 1) to[out=300, in=230] (1,1);
\draw[usual, dot] (1.5, 1) to (1.5,0.8);
\end{tikzpicture}.
\end{gather*}
The reader familiar with linear versions of the partition category might want to multiply by a scalar when removing internal components; we allow a generalization of this type below.
\end{Example}

A set of $\circ$-$\otimes$-generators for $\pacat$ is given by (with names, in order: identity, multiplication, unit, comultiplication, counit, crossing):
\begin{gather}\label{Eq:MoGen}
1_1=\text{id}_1 \colon
\begin{tikzpicture}[anchorbase,scale=0.55]
\draw[usual] (0,1) to[out=270,in=90] (0,0);
\end{tikzpicture}
\,,\quad
\mu \colon
\begin{tikzpicture}[anchorbase,scale=0.55]
\draw[usual] (0,1) to[out=270,in=90] (0,0);
\draw[usual] (0,1) to[out=270,in=90] (1,0);
\end{tikzpicture}
\,,\quad
\eta \colon
\begin{tikzpicture}[anchorbase,scale=0.55]
\draw[white] (0,0) to[out=90,in=270] (0,1);
\draw[usual,dot] (0,1) to (0,0.7);
\end{tikzpicture}
\,,\quad
\Delta \colon
\begin{tikzpicture}[anchorbase,scale=0.55]
\draw[usual] (0,0) to[out=90,in=270] (0,1);
\draw[usual] (0,0) to[out=90,in=270] (1,1);
\end{tikzpicture}
,\quad
\epsilon \colon
\begin{tikzpicture}[anchorbase,scale=0.55]
\draw[white] (0,0) to[out=90,in=270] (0,1);
\draw[usual,dot] (0,0) to (0,0.3);
\end{tikzpicture}
,\quad
s \colon
\begin{tikzpicture}[anchorbase,scale=0.55]
\draw[usual] (0,0) to[out=90,in=270] (1,1);
\draw[usual] (1,0) to[out=90,in=270] (0,1);
\end{tikzpicture}
\,.
\end{gather}

\begin{Notation}
Define $1_k=\text{id}_k=\otimes^k1_1$, the identity morphism in $\End_{\pacat}(k)$.
\end{Notation}

Using
\begin{gather*}
\text{cup}\colon
\begin{tikzpicture}[anchorbase,scale=0.55]
\draw[usual] (0,0) to[out=90,in=180] (0.5,0.5) to[out=0,in=90] (1,0);
\end{tikzpicture}
=
\begin{tikzpicture}[anchorbase,scale=0.55]
\draw[usual] (0,1) to[out=270,in=90] (0,0);
\draw[usual] (0,1) to[out=270,in=90] (1,0);
\draw[usual,dot] (0,1) to (0,1.3);
\end{tikzpicture}
,\quad
\text{cap}\colon
\begin{tikzpicture}[anchorbase,scale=0.55,yscale=-1]
\draw[usual] (0,0) to[out=90,in=180] (0.5,0.5) to[out=0,in=90] (1,0);
\end{tikzpicture}
=
\begin{tikzpicture}[anchorbase,scale=0.55,yscale=-1]
\draw[usual] (0,1) to[out=270,in=90] (0,0);
\draw[usual] (0,1) to[out=270,in=90] (1,0);
\draw[usual,dot] (0,1) to (0,1.3);
\end{tikzpicture}
,
\end{gather*}
it is easy to see that $\pacat$ and $\ppacat$ are pivotal (with the cup and cap as structures). Moreover, $\pacat$ is symmetric using the crossing $s$.

\begin{Notation}
For all diagram categories throughout that allow cup and cap, we will always use the pivotal structure they define. Ditto, when they allow the crossing, then they are symmetric with the crossing structure, and we will always use this symmetry.
\end{Notation}

\begin{Example}
Here is an example of how to decompose a partition diagram into a $\circ$-$\otimes$-product of generators:
\begin{gather*}
\pamon[3] \ni \quad
\begin{tikzpicture}[anchorbase]
\draw[usual] (0, 0.8) to[out=90, in=270] (1, 2);
\draw[usual] (0, 0) to[out=90, in=270] (0, 0.8);
\draw[usual] (0, 0.8) to[out=270, in=90] (1, 0);
\draw[usual] (2, 0) to[out=90, in=270] (0, 2);
\draw[usual,dot] (2,2) to (2,1.8);
\draw[very thick,densely dotted,tomato] (0,1.6) to (2,1.6);
\draw[very thick,densely dotted,tomato] (0,0.9) to (2,0.9);
\end{tikzpicture}
\quad =\quad
(1_2 \otimes \eta) \circ s  \circ (\mu \otimes 1_1).
\end{gather*}
As we will see later, the slices correspond to parts of a \textit{sandwich structure}.
\end{Example}

The generators are subject to the following relations that we recommend drawing:
\begin{align}
s^2 &= 1_2, \label{startrel}\\
(1_1\otimes s) \circ (s \otimes 1_1) \circ (1_1 \otimes s) & = (s \otimes 1_1) \circ (1_1 \otimes s) \circ (s \otimes 1_1), \label{symm}\\
s \circ (1_1 \otimes \eta) &= \eta \otimes 1_1, \label{hol1}\\
(1_1 \otimes \mu) \circ (s \otimes 1_1)\circ (1_1 \otimes s) & = s \circ (\mu \otimes 1_1), \label{hol2} \\
(1_1 \otimes \epsilon) \circ s & = \epsilon \otimes 1_1, \label{hol3}\\
(1_1 \otimes s) \circ (s \otimes 1_1) \circ (1_1 \otimes \Delta) &= (\Delta \otimes 1_1) \circ s, \label{hol4}\\
\mu \circ (1_1 \otimes \eta) &= 1_1 = \mu \circ (\eta \otimes 1_1), \\
(1_1 \otimes \epsilon) \circ \Delta &= 1_1 = (\epsilon \otimes 1_1) \circ \Delta, \\
(\mu \otimes 1_1) \circ (1_1 \otimes \Delta) &= \Delta \circ \mu = (1_1 \otimes \mu) \circ (\Delta \otimes 1_1), \label{frob}\\
\mu \circ s &= \mu,  \label{endrel} \\
\mu \circ \Delta &= 1_1. \label{holes}
\end{align}
The generators from \autoref{Eq:MoGen} and the relations \autoref{startrel}--\autoref{holes} give a generator-relation presentation of $\pacat$, cf. \cite[Theorem 2.1]{Co-jelly}; excluding $s$ one gets a generator-relation presentation of $\ppacat$.

Together with $\ppacat$, we also consider the following important submonoids and categories; the list is taken verbatim from \cite{KhSiTu-monoidal-cryptography}.
\begin{enumerate}[label=$\bullet$]

\item The \emph{rook-Brauer monoid} $\robrmon$ consisting of all diagrams with components of size $1$ or $2$. The planar rook-Brauer monoid 
$\probrmon=\momon$ is also called \emph{Motzkin monoid}.
\begin{gather*}
\begin{tikzpicture}[anchorbase]
\draw[usual] (1,0) to[out=90,in=180] (1.25,0.25) to[out=0,in=90] (1.5,0);
\draw[usual] (1,1) to[out=270,in=180] (1.75,0.55) to[out=0,in=270] (2.5,1);
\draw[usual] (0,0) to (0.5,1);
\draw[usual] (2.5,0) to (2,1);
\draw[usual,dot] (0.5,0) to (0.5,0.2);
\draw[usual,dot] (2,0) to (2,0.2);
\draw[usual,dot] (0,1) to (0,0.8);
\draw[usual,dot] (1.5,1) to (1.5,0.8);
\end{tikzpicture}
\in\robrmon[6]
,\quad
\begin{tikzpicture}[anchorbase]
\draw[usual] (0.5,0) to[out=90,in=180] (1.25,0.5) to[out=0,in=90] (2,0);
\draw[usual] (1,0) to[out=90,in=180] (1.25,0.25) to[out=0,in=90] (1.5,0);
\draw[usual] (2,1) to[out=270,in=180] (2.25,0.75) to[out=0,in=270] (2.5,1);
\draw[usual] (0,0) to (1,1);
\draw[usual,dot] (2.5,0) to (2.5,0.2);
\draw[usual,dot] (0,1) to (0,0.8);
\draw[usual,dot] (0.5,1) to (0.5,0.8);
\draw[usual,dot] (1.5,1) to (1.5,0.8);
\end{tikzpicture}
\in\momon[6]
.
\end{gather*}

\item The \emph{Brauer monoid} $\brmon$ consisting of all diagrams with components of size $2$. The planar Brauer monoid $\pbrmon=\tlmon$ is known as the \emph{Temperley--Lieb monoid} (sometimes $\tlmon$ is called \emph{Jones monoid} or \emph{Kauffman monoid}).
\begin{gather*}
\begin{tikzpicture}[anchorbase]
\draw[usual] (0.5,0) to[out=90,in=180] (1.25,0.45) to[out=0,in=90] (2,0);
\draw[usual] (1,0) to[out=90,in=180] (1.25,0.25) to[out=0,in=90] (1.5,0);
\draw[usual] (0,1) to[out=270,in=180] (0.75,0.55) to[out=0,in=270] (1.5,1);
\draw[usual] (1,1) to[out=270,in=180] (1.75,0.55) to[out=0,in=270] (2.5,1);
\draw[usual] (0,0) to (0.5,1);
\draw[usual] (2.5,0) to (2,1);
\end{tikzpicture}
\in\brmon[6]
,\quad
\begin{tikzpicture}[anchorbase]
\draw[usual] (0.5,0) to[out=90,in=180] (1.25,0.5) to[out=0,in=90] (2,0);
\draw[usual] (1,0) to[out=90,in=180] (1.25,0.25) to[out=0,in=90] (1.5,0);
\draw[usual] (0,1) to[out=270,in=180] (0.25,0.75) to[out=0,in=270] (0.5,1);
\draw[usual] (2,1) to[out=270,in=180] (2.25,0.75) to[out=0,in=270] (2.5,1);
\draw[usual] (0,0) to (1,1);
\draw[usual] (2.5,0) to (1.5,1);
\end{tikzpicture}
\in\tlmon[6]
.
\end{gather*}

\item The \emph{rook monoid} or 
\emph{symmetric inverse semigroup} $\romon$ consisting of all diagrams with components of size $1$ or $2$, and all partitions have at most one component 
at the bottom and at most one at the top. The \emph{planar rook monoid}
$\promon$ is the corresponding planar submonoid.
\begin{gather*}
\begin{tikzpicture}[anchorbase]
\draw[usual] (0,0) to (1,1);
\draw[usual] (0.5,0) to (0,1);
\draw[usual] (2,0) to (2,1);
\draw[usual] (2.5,0) to (0.5,1);
\draw[usual,dot] (1,0) to (1,0.2);
\draw[usual,dot] (1.5,0) to (1.5,0.2);
\draw[usual,dot] (1.5,1) to (1.5,0.8);
\draw[usual,dot] (2.5,1) to (2.5,0.8);
\end{tikzpicture}
\in\romon[6]
,\quad
\begin{tikzpicture}[anchorbase]
\draw[usual] (0,0) to (0.5,1);
\draw[usual] (0.5,0) to (1,1);
\draw[usual] (2,0) to (1.5,1);
\draw[usual] (2.5,0) to (2.5,1);
\draw[usual,dot] (1,0) to (1,0.2);
\draw[usual,dot] (1.5,0) to (1.5,0.2);
\draw[usual,dot] (0,1) to (0,0.8);
\draw[usual,dot] (2,1) to (2,0.8);
\end{tikzpicture}
\in\promon[6]
.
\end{gather*}

\item The \emph{symmetric group} $\sym$ consisting of all matchings with components of size $1$. The \emph{planar symmetric group} is trivial $\psym\cong\onemon$.
\begin{gather*}
\begin{tikzpicture}[anchorbase]
\draw[usual] (0,0) to (1,1);
\draw[usual] (0.5,0) to (0,1);
\draw[usual] (1,0) to (1.5,1);
\draw[usual] (1.5,0) to (2.5,1);
\draw[usual] (2,0) to (2,1);
\draw[usual] (2.5,0) to (0.5,1);
\end{tikzpicture}
\in\sym[6]
,\quad
\begin{tikzpicture}[anchorbase]
\draw[usual] (0,0) to (0,1);
\draw[usual] (0.5,0) to (0.5,1);
\draw[usual] (1,0) to (1,1);
\draw[usual] (1.5,0) to (1.5,1);
\draw[usual] (2,0) to (2,1);
\draw[usual] (2.5,0) to (2.5,1);
\end{tikzpicture}
\in\psym[6]
.
\end{gather*}

\end{enumerate}
There are also the respective categories $\robrcat[1]$, $\mocat[1]$, 
$\brcat[1]$, $\tlcat[1]$ etc.

\begin{Notation}\label{N:DAlgebrasTMonTCellsTwo}
Below we write $\xmon{1}$ or $\xcat[1]$ for any of the monoids/categories listed above, seen as submonoids/subcategories of the partition category. In fact, one could also take any submonoid of $\pacat$ (e.g. the ones in \cite{KaMaYu-extension-tl}).
\end{Notation}

Let us now modify $\xcat$ by ignoring the final relation 
\autoref{holes} and set 
\begin{gather*}
h\colon \begin{tikzpicture}[anchorbase]
\draw[usual, hol=0.5] (0,0) to (0,1);
\end{tikzpicture}
=
\begin{tikzpicture}[anchorbase]
\draw[usual] (0,0) to[out=90, in=270] (0,0.5);
\draw[usual] (0,0) to[out=90, in = 270] (0.5,0.5);
\draw[usual] (0,0.5) to (0,1);
\draw[usual] (0.5,0.5) to[out= 90, in=270] (0,1);
\end{tikzpicture}
= \mu \circ \Delta.
\end{gather*}
Furthermore, introduce an additional generator
\begin{gather}\label{Eq:MDot}
m \colon \begin{tikzpicture}[anchorbase]
\draw[usual, mob=0.5] (0,0) to (0,1);
\end{tikzpicture}
,
\end{gather}
satisfying the relations
\begin{align}
\mu \circ (m \otimes 1_1 ) &= m \circ \mu = \mu \circ (1_1 \otimes m)  \label{mobmu} \\
s \circ (m \otimes 1_1) & =  (1_1 \otimes m) \circ s \label{mobs} \\
m^3 &=h \circ m. \label{mobrel}
\end{align}
The morphisms $h$ and $m$ are the \emph{handle dot} (sometimes just called \emph{hole}) and \emph{M{\"o}bius dot}.

\begin{Example}\label{exmp1}
The two relations \autoref{mobmu} and \autoref{mobs} in diagrammatic form are
\begin{gather*}
\begin{tikzpicture}[anchorbase]
\draw[usual,mob=0.25] (0,0) to (0.5,0.5) to (1,0);
\draw[usual] (0.5,0.5) to (0.5,1);
\end{tikzpicture}
=
\begin{tikzpicture}[anchorbase]
\draw[usual] (0,0) to (0.5,0.5) to (1,0);
\draw[usual,mob=0.5] (0.5,0.5) to (0.5,1);
\end{tikzpicture}
=
\begin{tikzpicture}[anchorbase]
\draw[usual,mob=0.75] (0,0) to (0.5,0.5) to (1,0);
\draw[usual] (0.5,0.5) to (0.5,1);
\end{tikzpicture}
\hspace{0.5cm} \text{and} \hspace{0.5cm}
\begin{tikzpicture}[anchorbase]
\draw[usual,mob=0.25] (0,0) to (1,1);
\draw[usual] (1,0) to (0,1);
\end{tikzpicture}
=
\begin{tikzpicture}[anchorbase]
\draw[usual,mob=0.75] (0,0) to (1,1);
\draw[usual] (1,0) to (0,1);
\end{tikzpicture},
\end{gather*}
while the final relation \autoref{mobrel} is
\begin{gather*}
\begin{tikzpicture}[anchorbase]
\draw[usual, mob=0.25, mob=0.5, mob=0.75] (0,0) to (0,1);
\end{tikzpicture}
=
\begin{tikzpicture}[anchorbase]
\draw[usual, hol=0.66, mob=0.33] (0,0) to (0,1);
\end{tikzpicture}
.
\end{gather*}
We will see later that is somewhat awkward looking relation is a famous topological relation, which we will call the \emph{M{\"o}bius relation}.
\end{Example}

\begin{Notation}\label{N:MDAlgebrasTMonTCellsTwo}
Let $\mxcat$ denote the modified $\xcat$ as above; ditto for monoids.
\end{Notation}

\begin{Example}
In other words, elements of $\mxmon{1}$ are diagrams in $\xmon{1}$ with parts decorated by $h$ and $m$, representing white and red circles respectively, which commute with each other, move freely along the strands, and are subject to \autoref{mobrel}. For example,
\begin{gather*}
\begin{tikzpicture}[anchorbase]
\draw[usual] (1,0) to[out=90,in=180] (1.25,0.25) to[out=0,in=90] (1.5,0);
\draw[usual, mob=0.5] (1,1) to[out=270,in=180] (1.75,0.55) to[out=0,in=270] (2.5,1);
\draw[usual, mob=0.25, mob=0.5, hol=0.75] (0,0) to (0.5,1);
\draw[usual] (2.5,0) to (2,1);
\draw[usual,dot] (0.5,0) to (0.5,0.4);
\draw[usual,dot] (2,0) to (2,0.4);
\draw[usual, mob=0.5] (2,0) to (2,0.4);
\draw[usual,dot] (0,1) to (0,0.8);
\draw[usual,dot] (1.5,1) to (1.5,0.8);
\draw[usual, hol=0.5] (0.5,0) to (0.5,0.4);
\end{tikzpicture}
\in\mrobrmon[6]{1}
,\hspace{0.5cm}
\begin{tikzpicture}[anchorbase]
\draw[usual, hol=0.25, hol=0.5, hol=0.75] (0.5,0) to[out=90,in=180] (1.25,0.5) to[out=0,in=90] (2,0);
\draw[usual] (1,0) to[out=90,in=180] (1.25,0.25) to[out=0,in=90] (1.5,0);
\draw[usual, mob=0.5] (0,1) to[out=270,in=180] (0.25,0.75) to[out=0,in=270] (0.5,1);
\draw[usual] (2,1) to[out=270,in=180] (2.25,0.75) to[out=0,in=270] (2.5,1);
\draw[usual, mob=0.33, mob=0.66] (0,0) to (1,1);
\draw[usual,  mob=0.25, hol=0.5, hol=0.75] (2.5,0) to (1.5,1);
\end{tikzpicture}
\in\mtlmon[6]{1}
.
\end{gather*}
Note that every connected component can only contain zero, one, or two M{\"o}bius dots.
\end{Example}

From now on, fix a ground field $\K$.

\begin{Definition}\label{D:DAlgebrasTheAlgebras}
Fix the following rational functions:
\begin{align*}
Z_{\paraa} = \frac{p_\paraa (T)}{q(T)} = \sum_{k \geq 0} \paraa_kT^k, \hspace{0.5cm} Z_{\parab} = \frac{p_\parab (T)}{q(T)} = \sum_{k \geq 0} \parab_kT^k, \hspace{0.5cm} Z_{\parac} = \frac{p_\parac (T)}{q(T)} = \sum_{k \geq 0} \parac_kT^k,
\end{align*}
where $p_\paraa(T), p_\parab(T), p_\parac(T), q(T) \in \K[T]$ with $q(0)=1$ satisfying
\begin{align}
\deg\big(p_\parab(T)\big), \deg\big(p_\parac(T)\big) < K = \max(N+1, M), \label{conds}
\end{align}
where 
\begin{align*}
N = \deg\big(p_\paraa(T)\big) \hspace{0.5cm} \text{and} \hspace{0.5cm} M = \deg\big(q(T)\big).
\end{align*}
We define $\dashmxcat[{\paraa, \parab, \parac}]$ to be the $\K$-linear category with hom-space basis the diagrams from $\mxcat$ and multiplication of basis elements given by the diagrammatic $\circ$ or $\otimes$ concatenation, except that all closed components are evaluated according to:
\begin{gather}
\begin{tikzpicture}[anchorbase]
\draw[usual, dot] (0,0) to (0,1);
\draw[usual, dot] (0,1) to (0,0);
\end{tikzpicture}
\quad = \quad
\paraa_0, \hspace{0.75cm} \begin{tikzpicture}[anchorbase]
\draw[usual, dot] (0,0) to (0,1);
\draw[usual, dot] (0,1) to (0,0);
\draw[usual, hol=0.5] (0,1) to (0,0);
\end{tikzpicture}
\quad = \quad
\paraa_1, \hspace{0.75cm} \begin{tikzpicture}[anchorbase]
\draw[usual, dot] (0,0) to (0,1);
\draw[usual, dot] (0,1) to (0,0);
\draw[usual, hol=0.33, hol=0.66] (0,1) to (0,0);
\end{tikzpicture}
\quad = \quad
\paraa_2 \label{series1}, \hspace{0.75cm} \ldots\,, \\
\begin{tikzpicture}[anchorbase]
\draw[usual, dot] (0,0) to (0,1);
\draw[usual, dot] (0,1) to (0,0);
\draw[usual, mob=0.5] (0,0) to (0,1);
\end{tikzpicture}
\quad = \quad
\parab_0, \hspace{0.75cm} \begin{tikzpicture}[anchorbase]
\draw[usual, dot] (0,0) to (0,1);
\draw[usual, dot] (0,1) to (0,0);
\draw[usual, mob=0.33, hol=0.66] (0,0) to (0,1);
\end{tikzpicture}
\quad = \quad
\parab_1, \hspace{0.75cm} \begin{tikzpicture}[anchorbase]
\draw[usual, dot] (0,0) to (0,1);
\draw[usual, dot] (0,1) to (0,0);
\draw[usual, mob =0.25, hol=0.5, hol=0.75] (0,0) to (0,1);
\end{tikzpicture}
\quad = \quad
\parab_2, \hspace{0.75cm} \ldots\,, \label{series2} \\
\begin{tikzpicture}[anchorbase]
\draw[usual, dot] (0,0) to (0,1);
\draw[usual, dot] (0,1) to (0,0);
\draw[usual, mob=0.33, mob=0.66] (0,0) to (0,1);
\end{tikzpicture}
\quad = \quad
\parac_0, \hspace{0.75cm} \begin{tikzpicture}[anchorbase]
\draw[usual, dot] (0,0) to (0,1);
\draw[usual, dot] (0,1) to (0,0);
\draw[usual, mob=0.25, mob=0.5, hol=0.75] (0,0) to (0,1);
\end{tikzpicture}
\quad = \quad
\parac_1, \hspace{0.75cm} \begin{tikzpicture}[anchorbase]
\draw[usual, dot] (0,0) to (0,1);
\draw[usual, dot] (0,1) to (0,0);
\draw[usual, mob =0.2, mob=0.4, hol=0.6, hol=0.8] (0,0) to (0,1);
\end{tikzpicture}
\quad = \quad
\parac_2, \hspace{0.75cm} \ldots\,. \label{series3}
\end{gather}
These are the \emph{evaluation rules}.
\end{Definition}

\begin{Remark}
By the relation \autoref{mobrel}, closed components (and indeed any individual parts) containing more than three instances of the generator $m$ can be reduced to two or less. Hence, \autoref{series1}, \autoref{series2}, and \autoref{series3} together represent all possible closed components.
\end{Remark}

Finally, using the notation introduced in \autoref{D:DAlgebrasTheAlgebras}, given 
\begin{align*}
q(T) = 1 - a_1 T +a_2T^2+ \ldots + (-1)^Ma_MT^M,
\end{align*}
we define the \emph{handle relation}
\begin{align}\label{handle}
\sigma = h^K + \sum_{i=1}^M (-1)^ia_ih^{K-i}=0,
\end{align}
which is 
\begin{gather*}
\begin{tikzpicture}[anchorbase]
\draw[usual,hol=0.5] (0,0) to (0,0.5)node[left]{$K$} to (0,1);
\end{tikzpicture}
-a_1\cdot 
\begin{tikzpicture}[anchorbase]
\draw[usual,hol=0.5] (0,0) to (0,0.5)node[left]{$(K-1)$} to (0,1);
\end{tikzpicture}
+a_2\cdot 
\begin{tikzpicture}[anchorbase]
\draw[usual,hol=0.5] (0,0) to (0,0.5)node[left]{$(K-2)$} to (0,1);
\end{tikzpicture}
-\dots+(-1)^{\deg q}a_{\deg q}\cdot
\begin{tikzpicture}[anchorbase]
\draw[usual,hol=0.5] (0,0) to (0,0.5)node[left]{$(K-\deg q)$} to (0,1);
\end{tikzpicture}
=0
\end{gather*}
in diagrammatic form (the label next to the handle dots represent their number).

\begin{Example}\label{E:RootOfUnity}
For odd $K$, the ``root of unity relation''
\begin{gather*}
\begin{tikzpicture}[anchorbase]
\draw[usual,hol=0.5] (0,0) to (0,0.5)node[left]{$K$} to (0,1);
\end{tikzpicture}
=
\begin{tikzpicture}[anchorbase]
\draw[usual] (0,0) to (0,0.5) to (0,1);
\end{tikzpicture}
,
\end{gather*}
comes from the choice $p_\paraa(T)=1$ and $q(T)=1-T^K$.
\end{Example}

\begin{Definition}\label{mobalgdef}
The \emph{Möbius strip diagram category} $\mxcat[{\paraa, \parab, \parac}]$ is defined as the quotient $\dashmxcat[{\paraa, \parab, \parac}] / \, \In(\sigma)$ where $\In(\sigma)$ is the $\circ$-$\otimes$-ideal generated by $\sigma$.

In the same way, we get \emph{Möbius strip diagram algebras} $\mxmon{\paraa, \parab, \parac}$.
Whenever all involved parameters are zero or one we can, and will, use $\mxmon{\paraa, \parab, \parac}$ as a monoid after adjoining a formal zero element.
\end{Definition}

At this stage we cannot say anything about the Möbius strip diagram category, e.g. we do not know whether the composition is associative (it is, but proving this from the definition is not immediate); we address these issues in \autoref{M:Mobius}.

\begin{Remark}
We will see more clearly in the next section that the quotient in \autoref{mobalgdef} limits the number of occurrences of $h$ on any individual part to $K-1$ \ie \, the Möbius strip diagram algebras are finite dimensional.  For monoid parameters, we get examples of families of finite monoids. Forming tensor products of their representations, one could then e.g. explore analogs of \cite{CoOsTu-growth,HeTu-monoid-growth}.
\end{Remark} 

Once established, it is easy to see that the categories $\mxcat[\paraa, \parab, \parac]$ are monoidal, symmetric whenever the crossing is present, and pivotal whenever the cup and cap are present.


\section{Möbius strip TQFT}\label{M:Mobius}


We would now like to provide some motivation for the construction of the Möbius strip diagram algebras, and in particular, the defining relation \autoref{mobrel}. As the name suggests, we will need to make contact with topology. This section is motivated by \cite{TuTu-utqft,Tu-virtual-khovanov,Cz-mobius-tqft}, and some familiarity with 2D TQFT as in e.g. \cite{Ko-tqfts} is useful.

We first consider the $\K$-linear monoidal partition category $\pacat[\para]$, whose objects and morphisms are given by those of $\pacat$, except that we are allowed to take $\K$-linear combinations of partition diagrams and when composing diagrams, closed components are evaluated to some fixed $\para\in\K$. In fact, and similar as before, $\pacat[\para]$ is symmetric and pivotal. Furthermore, forming the additive idempotent completion (additive closure followed by the Karoubi envelope) of $\pacat[\para]$ produces (the non-abelian version of) Deligne's category $\delcat{\para}$, see e.g. \cite{De-cat-st,CoOs-blocks-deligne-sym-group}.

\begin{Remark}
The monoids $\xmon{1}\subset \End_{\pacat[\para]}(n)$ and the associated diagram algebras $\xmon{\para}$ shown in \eg \,\cite{Tu-sandwich-cellular} are subalgebras of $\End_{\pacat[\para]}(n)$.
\end{Remark}

\begin{Remark}
A set of generators for $\pacat[\para]$ are given by $\mu, \eta, \Delta, \epsilon, s$ from the previous section, subject to the same relations together with $\epsilon \circ \eta = \para$.
\end{Remark}

\begin{Remark}
Below we treat cobordisms as abstract manifolds \ie \, they are not embedded anywhere. 
(Although it would be interesting to consider an embedded analogue of the constructions below, cf. \cite{BeWa-eqft}, we do not pursue this here.)
We will nevertheless use the usual pictures for illustrative purposes. We also silently strictify all categories, which is not a problem for our purposes by the usual coherence theorems.
\end{Remark}

An important related category is the category of \emph{oriented 2D cobordisms} $\cob$. The objects of $\cob$ are nonnegative integers and morphisms from $n$ to $m$ are (diffeomorphism classes of) cobordisms $\coprod_{i=1}^n \, \mathbb{S}^1 \,
\begin{tikzpicture}[baseline=-.75ex]
\node [single arrow, draw, minimum height=1.2em, 
single arrow head extend=0.5ex,inner sep=0.3ex] {};
\end{tikzpicture}\
\coprod_{i=1}^m \, \mathbb{S}^1$.

\begin{Example}\label{examporien}
Here is an example:
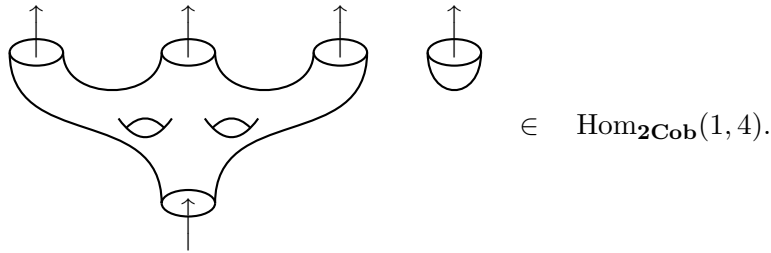
\begin{gather*}
\begin{tikzpicture}[
baseline ={([yshift=-.5ex]current bounding box.center)}]
\pic[
tqft,
incoming boundary components=3,
outgoing boundary components=1,
every lower boundary component/.style={draw},
genus=2,
offset = 1,
draw,
thick,
name=ex1
];
\pic[tqft,
incoming boundary components=1,
outgoing boundary components=0,
draw,
thick,
name=cyldoub1,
every lower boundary component/.style={draw},
at={(5.5,0)},
];
\node at ([yshift=-8pt]ex1-outgoing boundary 1) {$\Big\uparrow$};
\node at ([yshift=8pt]ex1-incoming boundary 2) {$\Big\uparrow$};
\node at ([yshift=8pt]ex1-incoming boundary 3)
{$\Big\uparrow$};
\node at ([yshift=8pt]ex1-incoming boundary 1) {$\Big\uparrow$};
\node at ([yshift=8pt]cyldoub1-incoming boundary 1) {$\Big\uparrow$};
\end{tikzpicture} \quad \in \quad \Hom_{\cob}(1, 4).
\end{gather*}
The orientations for the boundary $\mathbb{S}^1$ are indicated by arrows; this convention is used throughout, and so the arrows are omitted in future. We also consider individual cobordisms as representatives for their corresponding diffeomorphism class. 
\end{Example}

It is a well-known fact (see \eg \, \cite{Ko-tqfts}) that $\cob$ forms a monoidal category under disjoint union, which is symmetric and pivotal (crossing, and cup and cap), and has the following generators (as before and in order: identity, multiplication, unit, comultiplication, counit, crossing):
\begin{gather}\label{Eq:Gen}
1_1 =\text{id}_1 \colon \begin{tikzpicture}[baseline ={([yshift=-.5ex]current bounding box.center)}, scale=0.25]
\pic[
tqft,
incoming boundary components=1,
outgoing boundary components=1,
every lower boundary component/.style={draw},
genus=0,
offset=0,
scale=0.75,
draw,
thick,
boundary separation=40pt,
name=pop
]; 
\end{tikzpicture}, \quad
\mu\colon \begin{tikzpicture}[baseline ={([yshift=-.5ex]current bounding box.center)}, scale=0.25]
\pic[
tqft,
incoming boundary components=1,
outgoing boundary components=2,
every lower boundary component/.style={draw},
genus=0,
offset=-0.5,
scale=0.75,
draw,
thick,
boundary separation=40pt,
name=pop
]; 
\end{tikzpicture}, \quad
\eta\colon \begin{tikzpicture}[baseline ={([yshift=-.5ex]current bounding box.center)}, scale=0.25]
\pic[
tqft,
incoming boundary components=1,
outgoing boundary components=0,
every lower boundary component/.style={draw},
genus=0,
scale=0.75,
draw,
thick,
name=cup
]; 
\end{tikzpicture}, \quad
\Delta\colon   \begin{tikzpicture}[baseline ={([yshift=-.5ex]current bounding box.center)}, scale=0.25]
\pic[
tqft,
incoming boundary components=2,
outgoing boundary components=1,
every lower boundary component/.style={draw},
genus=0,
offset=0.5,
scale = 0.75,
draw,
thick,
boundary separation=40pt,
name=pop
]; 
\end{tikzpicture}, \quad
\epsilon\colon \begin{tikzpicture}[baseline ={([yshift=-.5ex]current bounding box.center)}, scale=0.25]
\pic[
tqft,
incoming boundary components=0,
outgoing boundary components=1,
every lower boundary component/.style={draw},
genus=0,
scale = 0.75,
draw,
thick,
name=cup
]; 
\end{tikzpicture}, \quad
s\colon  \begin{tikzpicture}[baseline ={([yshift=-.5ex]current bounding box.center)}, scale=0.25,
tqft/.cd,
cobordism/.style={draw},
every upper boundary component/.style={draw},
every lower boundary component/.style={draw},
]
\pic [tqft/cylinder to next,anchor=incoming boundary 1,name=c, boundary separation=75pt, scale=0.75, draw, thick];
\pic [tqft/cylinder to prior,anchor=incoming boundary 1,
at=(c-outgoing boundary |- c-incoming boundary), boundary separation=75pt, scale=0.75 , draw, thick];
\end{tikzpicture},
\end{gather}
subject to the relations \autoref{startrel}--\autoref{endrel}. If $\K\cob$ denotes the $\K$-linearisation of $\cob$, then it follows straightforwardly that $\pacat[\para]$ is equivalent, as a $\K$-linear monoidal category, to the category obtained from $\K\cob$ by factoring out by the additional relations \autoref{holes} and $\epsilon \circ \eta = \para$, cf. \cite{Co-jelly}.

Let us now introduce the following additional generator to $\cob$:
\begin{gather}\label{Eq:MGen}
m=\begin{tikzpicture}[baseline ={([yshift=-.5ex]current bounding box.center)}, scale=0.25]
\pic[
tqft,
incoming boundary components=1,
outgoing boundary components=1,
every lower boundary component/.style={draw},
genus=0,
scale=0.75,
draw,
thick,
boundary separation=40pt,
name=pop
]; 
\node at ([xshift=0pt, yshift=80pt]pop-outgoing boundary 1) {$\lightning$};
\end{tikzpicture},
\quad
\lightning
=
\raisebox{-0.1cm}{$\begin{tikzpicture}[scale=0.5, line cap=round, line join=round,anchorbase]
\coordinate (L) at (-2,0);
\coordinate (R) at ( 2,0);
\fill[orchid!8]
(L) .. controls (-0.8, 1.4) and (0.8, 1.4) .. (R)
.. controls (0.8,-1.4) and (-0.8,-1.4) .. cycle;
\draw[line width=1.0pt]
(L) .. controls (-0.8, 1.4) and (0.8, 1.4) .. (R);
\draw[line width=1.0pt]
(L) .. controls (-0.8,-1.4) and (0.8,-1.4) .. (R);
\draw[line width=1.0pt, -{Stealth[length=2.6mm]}]
(-0.15,1.05) -- (0.15,1.05);
\draw[line width=1.0pt, -{Stealth[length=2.6mm]}]
(0.15,-1.05) -- (-0.15,-1.05);
\node at (0, 1.55) {$a$};
\node at (0,-1.55) {$a$};
\fill (L) circle (1.3pt);
\fill (R) circle (1.3pt);
\node (B) at (0,-2.15) {};
\end{tikzpicture}$}
=\mathbb{RP}^2
,
\end{gather}
which represents the connected sum of a cylinder and the real projective plane $\mathbb{RP}^2$ \ie \, $m$ can be thought of as a Möbius strip glued to the identity cobordism. Let $\mcob{}$ denote the resulting category.

To simplify our illustrations, we will frequently identify cobordisms with their spines: instead of \autoref{Eq:Gen}, we draw \autoref{Eq:MoGen} (matching them following the nomenclature), and instead of \autoref{Eq:MGen}, we draw a M{\"o}bius dot \autoref{Eq:MDot}. We also use handle dots instead of
\begin{gather*}
h\colon \begin{tikzpicture}[baseline ={([yshift=-.5ex]current bounding box.center)}, scale=0.25]
\pic[
tqft,
incoming boundary components=1,
outgoing boundary components=1,
every lower boundary component/.style={draw},
genus=1,
scale=0.75,
draw,
thick,
boundary separation=40pt,
name=pop
]; 
\end{tikzpicture} \quad = \quad \begin{tikzpicture}[tqft, baseline ={([yshift=-.5ex]current bounding box.center)}, scale=0.25, cobordism/.style={draw},
every upper boundary component/.style={draw},
every lower boundary component/.style={draw}]
\pic[tqft/pair of pants,draw, thick, name=a, scale=0.5];
\pic[tqft/reverse pair of pants,draw, thick, scale=0.5, name=b, anchor = incoming boundary 1, at =(a-outgoing boundary 1)];
\end{tikzpicture} \quad = \quad \mu \circ \Delta.
\end{gather*}

\begin{Lemma}\label{L:Mobius}
The morphism $m$ satisfies the relations \autoref{mobmu}, \autoref{mobs}, and \autoref{mobrel}.
\end{Lemma}

\begin{proof}
Using topological arguments, it is straightforward to see that $m$ satisfies the relations \autoref{mobmu} and \autoref{mobs}. For \autoref{mobrel}, a visualization exercise shows
\begin{gather*}
\begin{tikzpicture}[baseline ={([yshift=-.5ex]current bounding box.center)}, draw, thick, scale=0.25]
\draw (0,0) circle (2cm);
\draw (-2,0) arc (180:360:2 and 0.6);
\draw[dashed, thick] (2,0) arc (0:180:2 and 0.6);
\node at (0,-1){$\lightning$};
\node at (-0.66,0.8){$\lightning$};
\node at (0.66,0.8){$\lightning$};
\end{tikzpicture}
=
\begin{tikzpicture}[scale=0.25, draw, thick, baseline ={([yshift=-.5ex]current bounding box.center)}]
\draw (0,0) ellipse (3cm and 2cm);
\draw (-1.5, 0.2) arc (180:360:1.5 and 0.6);
\draw (1,-0.2) arc (0:180:1 and 0.6);
\node at (-2,0){$\lightning$};
\end{tikzpicture}
,
\end{gather*}
from which the relations follows; this is not easy to see but well-known (there are animations on YouTube demonstrating the diffeomorphism between them). The surface in question is called the \emph{Dyck surface}.
\end{proof}

\begin{Remark}
The Dyck surface is an example of how Möbius strips behave like monoid elements (they are not cancelable in general) since
\begin{gather*}
\begin{tikzpicture}[baseline ={([yshift=-.5ex]current bounding box.center)}, draw, thick, scale=0.25]
\draw (0,0) circle (2cm);
\draw (-2,0) arc (180:360:2 and 0.6);
\draw[dashed, thick] (2,0) arc (0:180:2 and 0.6);
\node at (0,-1){$\lightning$};
\node at (-0.66,0.8){$\lightning$};
\node at (0.66,0.8){$\lightning$};
\end{tikzpicture}
=
\begin{tikzpicture}[scale=0.25, draw, thick, baseline ={([yshift=-.5ex]current bounding box.center)}]
\draw (0,0) ellipse (3cm and 2cm);
\draw (-1.5, 0.2) arc (180:360:1.5 and 0.6);
\draw (1,-0.2) arc (0:180:1 and 0.6);
\node at (-2,0){$\lightning$};
\end{tikzpicture}
\quad\text{but}\quad
\begin{tikzpicture}[scale=0.25, line cap=round, line join=round, anchorbase]
\coordinate (BL) at (0,0);
\coordinate (BR) at (4,0);
\coordinate (TR) at (4,4);
\coordinate (TL) at (0,4);
\fill[orchid!8] (BL) -- (BR) -- (TR) -- (TL) -- cycle;
\draw[line width=1.0pt] (BL) -- (BR) -- (TR) -- (TL) -- cycle;
\draw[line width=1.0pt,-{Stealth[length=2.6mm]}] (0.8,4) -- (3.2,4);
\draw[line width=1.0pt,-{Stealth[length=2.6mm]}] (0.8,0) -- (3.2,0);
\draw[line width=1.0pt,-{Stealth[length=2.6mm]}] (0,0.8) -- (0,3.2);
\draw[line width=1.0pt,-{Stealth[length=2.6mm]}] (4,3.2) -- (4,0.8);
\node at (2,4.7) {$b$};
\node at (2,-0.7) {$b$};
\node at (-0.7,2) {$a$};
\node at (4.7,2) {$a$};
\end{tikzpicture}
=
\begin{tikzpicture}[baseline ={([yshift=-.5ex]current bounding box.center)}, draw, thick, scale=0.25]
\draw (0,0) circle (2cm);
\draw (-2,0) arc (180:360:2 and 0.6);
\draw[dashed, thick] (2,0) arc (0:180:2 and 0.6);
\node at (0,-1){$\lightning$};
\node at (0,0.8){$\lightning$};
\end{tikzpicture}
\neq
\begin{tikzpicture}[scale=0.25, draw, thick, baseline ={([yshift=-.5ex]current bounding box.center)}]
\draw (0,0) ellipse (3cm and 2cm);
\draw (-1.5, 0.2) arc (180:360:1.5 and 0.6);
\draw (1,-0.2) arc (0:180:1 and 0.6);
\end{tikzpicture}
.
\end{gather*}
The latter is the famous observation that the Klein bottle is not the torus.
\end{Remark}

We now consider the notion of \emph{nonorientable 2D cobordisms}, which are based on the usual \emph{unoriented 2D cobordisms} as e.g. in \cite{TuTu-utqft} (excluding the gluing data), except that we preserve our conventions for orienting the boundary $\mathbb{S}^1$ as shown in \autoref{examporien}; in particular, we do not consider orientation-reversing diffeomorphisms of $\mathbb{S}^1$, which are explicitly included in \eg \, \cite{Cz-mobius-tqft}, as mentioned in \autoref{rm:nophi}.

\begin{Theorem}\label{T:Mcob}
The category $\mcob{}$ is the category of nonorientable 2D cobordisms.
The category is symmetric, pivotal and has a generator-relation presentation with generators \autoref{Eq:Gen} and \autoref{Eq:MGen}, and relations \autoref{startrel}--\autoref{holes} and \autoref{mobmu}--\autoref{mobrel}.
\end{Theorem}

\begin{proof}
Firstly, it is easy to see that the stated relations imply all other ``sensible variations'' of them. For example,
\begin{gather*}
\left(
\begin{tikzpicture}[anchorbase]
\draw[usual,mob=0.25] (0,0) to (0.5,0.5) to (1,0);
\draw[usual] (0.5,0.5) to (0.5,1);
\end{tikzpicture}
=
\begin{tikzpicture}[anchorbase]
\draw[usual] (0,0) to (0.5,0.5) to (1,0);
\draw[usual,mob=0.5] (0.5,0.5) to (0.5,1);
\end{tikzpicture}
=
\begin{tikzpicture}[anchorbase]
\draw[usual,mob=0.75] (0,0) to (0.5,0.5) to (1,0);
\draw[usual] (0.5,0.5) to (0.5,1);
\end{tikzpicture}
\right)
\Rightarrow
\left(
\begin{tikzpicture}[anchorbase,yscale=-1]
\draw[usual,mob=0.25] (0,0) to (0.5,0.5) to (1,0);
\draw[usual] (0.5,0.5) to (0.5,1);
\end{tikzpicture}
=
\begin{tikzpicture}[anchorbase,yscale=-1]
\draw[usual] (0,0) to (0.5,0.5) to (1,0);
\draw[usual,mob=0.5] (0.5,0.5) to (0.5,1);
\end{tikzpicture}
=
\begin{tikzpicture}[anchorbase,yscale=-1]
\draw[usual,mob=0.75] (0,0) to (0.5,0.5) to (1,0);
\draw[usual] (0.5,0.5) to (0.5,1);
\end{tikzpicture}
\right)
.
\end{gather*}
The remaining part of the proof is then a variation of \cite{Ko-tqfts}, so we will be brief. As a first step, the relations can be used to show that every cobordism has a normal form, that we call a \emph{presandwich structure}, as follows. 

\begin{Notation} 
The \emph{diagrammatic anti-involution} ${}^{\ast}$ flips a cobordism upside-down.
A diagram is a \emph{merge diagram} if it contains only multiplications, counits, a minimal number of crossings, and dots (handles or M{\"o}bius) on components restricted to the bottom of the diagram. 
A \emph{split diagram} is a ${}^{\ast}$-flipped merge diagram. A \emph{dotted permutation diagram} contains only dots (handles or M{\"o}bius) and crossings. For example (flipping the left illustration gives a split diagram):
\begin{gather*}
\text{Merge diagram: }
\begin{tikzpicture}[anchorbase]
\draw[usual] (0,0) to (0.5,0.5) to (1,0);
\draw[usual] (0.5,0.5) to (0.5,1);
\draw[usual] (2,0) to (0.5,0.75);
\draw[usual] (0.5,0) to (1,1);
\draw[usual,hol=0.2,mob=0.6,dot] (-0.5,0) to (-0.5,0.5);
\end{tikzpicture}
,\quad
\text{dotted permutation diagram: }
\begin{tikzpicture}[anchorbase]
\draw[usual,hol=0.85,mob=0.95] (0,0) to (3,1);
\draw[usual,mob=0.85] (1,0) to (1,1);
\draw[usual] (2,0) to (0,1);
\draw[usual] (3,0) to (2,1);
\end{tikzpicture}
.
\end{gather*}
\end{Notation}

The normal form is then given by:
\begin{gather*}
\begin{tikzpicture}[anchorbase,scale=1]
\draw[line width=0.75,color=black,fill=cream] (0,1) to (0.25,0.5) to (0.75,0.5) to (1,1) to (0,1);
\node at (0.5,0.75){$T$};
\draw[line width=0.75,color=black,fill=cream] (0.25,0) to (0.25,0.5) to (0.75,0.5) to (0.75,0) to (0.25,0);
\node at (0.51,0.25){$m$};
\draw[line width=0.75,color=black,fill=cream] (0,-0.5) to (0.25,0) to (0.75,0) to (1,-0.5) to (0,-0.5);
\node at (0.5,-0.25){$B$};
\end{tikzpicture}
\quad\text{where}\quad
\begin{aligned}
\begin{tikzpicture}[anchorbase,scale=1]
\draw[line width=0.75,color=black,fill=cream] (0,1) to (0.25,0.5) to (0.75,0.5) to (1,1) to (0,1);
\node at (0.5,0.75){$T$};
\end{tikzpicture}
&\text{ a split diagram,}
\\
\begin{tikzpicture}[anchorbase,scale=1]
\draw[white,ultra thin] (0,0) to (1,0);
\draw[line width=0.75,color=black,fill=cream] (0.25,0) to (0.25,0.5) to (0.75,0.5) to (0.75,0) to (0.25,0);
\node at (0.5,0.25){$m$};
\end{tikzpicture}
&\text{ a dotted permutation,}
\\
\begin{tikzpicture}[anchorbase,scale=1]
\draw[line width=0.75,color=black,fill=cream] (0,-0.5) to (0.25,0) to (0.75,0) to (1,-0.5) to (0,-0.5);
\node at (0.5,-0.25){$B$};
\end{tikzpicture}
&\text{ a merge diagram.}
\end{aligned}
\quad
\text{${}^{\ast}$: }
\left(\,
\begin{tikzpicture}[anchorbase,scale=1]
\draw[line width=0.75,color=black,fill=cream] (0,1) to (0.25,0.5) to (0.75,0.5) to (1,1) to (0,1);
\node at (0.5,0.75){$T$};
\draw[line width=0.75,color=black,fill=cream] (0.25,0) to (0.25,0.5) to (0.75,0.5) to (0.75,0) to (0.25,0);
\node at (0.51,0.25){$m$};
\draw[line width=0.75,color=black,fill=cream] (0,-0.5) to (0.25,0) to (0.75,0) to (1,-0.5) to (0,-0.5);
\node at (0.5,-0.25){$B$};
\end{tikzpicture}\,
\right)^{\ast}
=
\begin{tikzpicture}[anchorbase,scale=1]
\draw[line width=0.75,color=black,fill=cream] (0,1) to (0.25,0.5) to (0.75,0.5) to (1,1) to (0,1);
\node at (0.5,0.75){\reflectbox{\rotatebox{180}{$B$}}};
\draw[line width=0.75,color=black,fill=cream] (0.25,0) to (0.25,0.5) to (0.75,0.5) to (0.75,0) to (0.25,0);
\node at (0.51,0.25){\reflectbox{\rotatebox{180}{$m$}}};
\draw[line width=0.75,color=black,fill=cream] (0,-0.5) to (0.25,0) to (0.75,0) to (1,-0.5) to (0,-0.5);
\node at (0.5,-0.25){\reflectbox{\rotatebox{180}{$T$}}};
\end{tikzpicture},
\end{gather*}
which is almost the same as in \cite{Ko-tqfts} (see also \cite{GrTu-diagram-growth,FrStTu-twists}), since the relations involving the M{\"o}bius generator imply that the M{\"o}bius dot moves freely through its connected component. Supposing that every connected component contains at most two M{\"o}bius dots, the normal form is also unique.

Also recall the (well-known) classification of surfaces: given $\mathbb{S}^2$ the two sphere, $\mathbb{D}$ a disc, and $\mathbb{T}$ a torus, any surface $S$ is completely determined by $S\cong\mathbb{S}^2\#(\#^d\mathbb{D})\#(\#^t\mathbb{T})\#(\#^p\mathbb{RP}^2)$ for unique $d,p,t\in\N$ with $tp=0$. The M{\"o}bius relation \autoref{mobrel} holds by \autoref{L:Mobius}, and reading right to left implies that we have either handle or M{\"o}bius dots on a connected component, but not both at the same time; this is the condition $tp=0$. Of course, we could also read left to right, in which case we end with the condition that every connected component contains at most two M{\"o}bius strips. In other words, the above normal form fits perfectly to the classification of surfaces. 

From here, the result follows from \cite{Ko-tqfts}.
\end{proof}

Using the same setup as \autoref{mobalgdef}, we then consider the $\K$-linearization of $\mcob{}$ with closed surfaces evaluated according to \autoref{series1}--\autoref{series3} (the empty cobordism $\emptyset$ is evaluated to $1\in\K$).
Following \cite{BlHaMaVo-tqft-kauffman-bracket} (see also \cite{KhSa-cobordisms}), we use the \emph{universal construction} with the given closed surface evaluations.
Let $\mucob{\paraa}{\parab}{\parac}$ denote the resulting category.
Using $h$, we can define $\sigma$ as in \autoref{handle}, and $\sigma = 0$ in this context is also the \emph{handle relation}.

\begin{Proposition}
The category $\mucob{\paraa}{\parab}{\parac}$ is symmetric and pivotal, and has a generator-relation presentation with generators \autoref{Eq:Gen} and \autoref{Eq:MGen}, and relations \autoref{startrel}--\autoref{holes} and \autoref{mobmu}--\autoref{mobrel}, and the handle relation \autoref{handle}.
\end{Proposition}

\begin{proof}
(The reader might benefit from checking \cite{EhStTu-blanchet-khovanov} for a similar argument spelled out in more detail.)

The universal construction implies that $\mucob{\paraa}{\parab}{\parac}$ is a well-defined monoidal category (e.g. it is associative), and that this category is symmetric and pivotal is easy to see.

For the remaining claim, first observe that the handle relation holds: this is essentially just the first row in \autoref{series1}. Indeed, we have
\begin{gather*}
\begin{tikzpicture}[anchorbase]
\draw[usual,hol=0.5,dot=0,dot=1] (0,0) to (0,0.5)node[left]{$K$} to (0,1);
\end{tikzpicture}
-a_1\cdot 
\begin{tikzpicture}[anchorbase]
\draw[usual,hol=0.5,dot=0,dot=1] (0,0) to (0,0.5)node[left]{$(K-1)$} to (0,1);
\end{tikzpicture}
+a_2\cdot 
\begin{tikzpicture}[anchorbase]
\draw[usual,hol=0.5,dot=0,dot=1] (0,0) to (0,0.5)node[left]{$(K-2)$} to (0,1);
\end{tikzpicture}
-\dots+(-1)^{\deg q}a_{\deg q}\cdot
\begin{tikzpicture}[anchorbase]
\draw[usual,hol=0.5,dot=0,dot=1] (0,0) to (0,0.5)node[left]{$(K-\deg q)$} to (0,1);
\end{tikzpicture}
=0,
\end{gather*}
the diagrammatic form of
\begin{gather*}
\paraa_K-a_1\cdot\paraa_{K-1}+a_2\cdot\paraa_{K-2}-\dots
-\dots+(-1)^{\deg q}a_{\deg q}\cdot\paraa_{K-\deg q}=0,
\end{gather*}
which holds by construction (recall \autoref{D:DAlgebrasTheAlgebras}).
Moreover, it is easy to see that any other way to close with relation gives a true equation as well.
Following \cite{EhStTu-blanchet-khovanov}, it is then easy to see that the unique normal form from the proof of \autoref{T:Mcob} gives a basis for the hom-spaces.
\end{proof}

Of course, the same modifications can be made to $\pacat[\para]$, firstly by ignoring the relations \autoref{holes} and $\epsilon \circ \eta =\para$, and introducing an additional generator $m$ satisfying \autoref{mobmu}--\autoref{mobrel}. It is clear that the resulting category is equivalent, as $\K$-linear symmetric pivotal category, to $\mpcat{\paraa, \parab, \parac}$ from the previous section, which in turn satisfies:

\begin{Theorem}
We have $\mpcat{\paraa, \parab, \parac} \cong \mucob{\paraa}{\parab}{\parac}$ as $\K$-linear symmetric pivotal categories.
\end{Theorem}

\begin{proof}
By the above.
\end{proof}

As a result, we will solely refer to the category $\mpcat{\paraa, \parab, \parac}$ going forward.

\begin{Remark}
The Möbius strip diagram algebras $\mxmon{\paraa, \parab, \parac}$ are subalgebras of $\End_{\mpcat{\paraa, \parab, \parac}}(n)$ and in particular, they are finite dimensional.
\end{Remark}

Let $\mvcat{\paraa, \parab, \parac}$ be the semisimplification (as e.g. in \cite{EtOs-semisimple}) of $\mpcat{\paraa, \parab, \parac}$ using the traces defined by \autoref{series1}--\autoref{series3}, with trace map (see also \cite{KhSa-cobordisms}):
\begin{gather*}
\tr\colon \quad
\begin{tikzpicture}[scale=0.8, draw, thick, baseline ={([yshift=-.5ex]current bounding box.center)}]
\begin{scope}[scale=0.8,shift={(0,0)}] 
\draw[usual] (0.5,2)--(0.5, -1); 
\draw[usual] (1.5,2)--(1.5, -1);
\draw[usual] (2.5,2)--(2.5, -1);
\node at (2.5,0.5) {$f$};
\draw[usual, fill=white] (0,1)--(3,1)--(3,0)--(0,0)--(0,1);
\node at (1.5,0.5) {$f$};
\node at (5.5,0.5) {$\longmapsto$ };
\end{scope}      
 \begin{scope}[scale=0.5,shift={(13.3,0.25)}]
  \draw[usual] (5.2,.5) ellipse (1.2cm and 1.5cm); 
  \draw[usual](5.2,.5) ellipse (2.7cm and 3cm);
  \draw[usual] (5.2,.5) ellipse (4.2cm and 4.5cm);
\draw[usual, fill=white] (0,1.5)--(5,1.5)--(5,-0.5)--(0,-0.5)--(0,1.5);
\node at (2.5,0.5) {$f$}; 
            \end{scope}              
\end{tikzpicture}.
\end{gather*}

We conclude this section with a theorem relating $\mvcat{\paraa, \parab, \parac}$ to the Deligne category $\mathbf{Re}\mathbf{p}(S_{\para})$. First, let $\mdelcat{\paraa}{\parab}{\parac}$ denote the additive idempotent completion of $\mvcat{\paraa, \parab, \parac}$. In addition, let $\underline{\mathbf{MRe}}\mathbf{p}(S_{\paraa,\parab,\parac})$ and 
$\underline{\mathbf{Re}}\mathbf{p}(S_{\para})$ denote the quotient of $\mdelcat{\paraa}{\parab}{\parac}$ and $\delcat{\para}$ respectively by the $\circ$-$\otimes$-ideal of negligible morphisms. We have:

\begin{Theorem}\label{T:Deligne}
In case that
\begin{align}\label{basiccond}
Z_{\paraa}(T) = \frac{\paraa_0}{1-\lambda T}, \hspace{1cm} Z_{\parab}(T) = \frac{\parab_0}{1-\lambda T}, \hspace{1cm} Z_{\parac}(T) = \frac{\parac_0}{1-\lambda T},
\end{align}
with $\paraa_0, \parab_0, \parac_0,\sqrt{\lambda} \in \K^{\times}$, there is an equivalence of $\K$-linear symmetric pivotal categories 
\begin{align}\label{mrepeq}
\underline{\mathbf{MRe}}\mathbf{p}(S_{\paraa,\parab,\parac})
\cong
\underline{\mathbf{Re}}\mathbf{p}(S_{\para})
\boxtimes
\underline{\mathbf{Re}}\mathbf{p}(S_{\para_+})
\boxtimes
\underline{\mathbf{Re}}\mathbf{p}(S_{\para_-}),
\end{align}
where $\delta = \lambda\paraa_0 - \parac_0$, $\parap=\frac{1}{2}(\parac_0 + \sqrt{\lambda}\parab_0)$, and $\param=\frac{1}{2}(\parac_0 - \sqrt{\lambda}\parab_0)$. All of these are semisimple.
\end{Theorem}

\begin{proof}
We begin by noticing that the category $\mcob{}$ is a subcategory of $\mathbf{U}\cob$, the category of two-dimensional unoriented cobordisms, as presented \eg \, in \cite{Cz-mobius-tqft}; in fact, to obtain $\mcob{}$ from $\mathbf{U}\cob$, we simply ignore the generator $\phi$ representing an orientation-reversing diffeomorphism of the circle. The construction of $\mdelcat{\paraa}{\parab}{\parac}$ then mirrors that of $\mathbf{U}\cob_{\paraa, \parab, \parac}$, and as a result, the proof of the theorem follows from that of \cite[Theorem II]{Cz-mobius-tqft} with some minor alterations; namely, the absence of $\phi$ eliminates a $\mathbb{Z}_2$ factor, which manifests as giving $\underline{\mathbf{Re}}\mathbf{p}(S_\para)$ instead of $\underline{\mathbf{Re}}\mathbf{p}(S_\para \wr \mathbb{Z}_2)$ in the first factor of the Deligne tensor product in \autoref{mrepeq} and multiplying $\para$ by a factor $2$. 

Finally, semisimplicity follows from a more general fact: given a $\K$-linear Karoubian rigid monoidal category $\mathcal{C}$ whose hom spaces are finite dimensional and that has a pivotal structure satisfying certain properties, the quotient of $\mathcal{C}$ by the tensor ideal generated by negligible morphisms produces a semisimple category (see \cite[Theorem 2.6]{EtOs-semisimple}).
\end{proof}

\begin{Example}
Take $\K=\C$ and $\lambda=1$, and say $\paraa_0=19$, $\parab_0=4$ and $\parac_0=10$. By \autoref{T:Deligne}, we have the equivalence
\begin{gather*}
\underline{\mathbf{MRe}}\mathbf{p}(S_{\paraa,\parab,\parac})
\cong\mathbf{Rep}(S_9)\boxtimes\mathbf{Rep}(S_7)\boxtimes
\mathbf{Rep}(S_3),
\end{gather*}
where $\mathbf{Rep}(S_k)$ denotes complex finite-dimensional representations of the symmetric group $S_k$.
\end{Example}


\section{Sandwich cellularity}\label{S:Sandwich}


We now study the representation theory of the Möbius strip diagram algebras using the framework of sandwich cellular algebras. For most of the this section, we follow the theory developed in \cite{Br-gen-matrix-algebras}, and some familiarity with \cite{Tu-sandwich-cellular} is assumed.

Let $\algebra$ be an associative unital $\K$-algebra. Recall the following definition:

\begin{Definition}\label{D:SandwichCellularAlgebra}
A \emph{sandwich cell datum} for $\algebra$ is a quadruple 
$\big(\Pcal,(\Tcal,\Bcal),(\sand,\sandbasis),C\big)$, where:
\begin{itemize}

\item $\Pcal=(\Pcal,\sandorder[\Pcal])$ is a poset (the \emph{middle poset} 
with \emph{sandwich order} $\sandorder[\Pcal]$),

\item $\Tcal=\bigcup_{\lambda\in\Pcal}\Tcal(\lambda)$ and $\Bcal=\bigcup_{\lambda\in\Pcal}\Bcal(\lambda)$ are collections of finite sets (the \emph{top/bottom sets}),

\item For $\lambda\in\Pcal$ we have
algebras $\sand[\lambda]$
(the \emph{sandwiched algebras}) and bases $\sandbasis[\lambda]$ of $\sand[\lambda]$,

\item  $C\colon\coprod_{\lambda\in\Pcal}\Tcal(\lambda)\times\sandbasis[\lambda]\times \Bcal(\lambda)\to\algebra;(T,m,B)\mapsto c_{T,m,B}^{\lambda}$ is an injective map,

\end{itemize}
such that:
\begin{enumerate}[label=\upshape(AC${}_{\arabic*}$\upshape)]

\item The set $\cellbasis=\set[\big]{c_{T,m,B}^{\lambda}|\lambda\in\Pcal,T\in\Tcal(\lambda),B\in\Bcal(\lambda),m\in\sandbasis[\lambda]}$
is a basis of $\algebra$.
(We call $\cellbasis$ a \emph{sandwich cellular basis}.)

\item For all $x\in\algebra$ 
there exist scalars $r_{TU}^{x}\in\K$ that do not depend
on $B$ or on $m$, such that
\begin{gather}\label{Eq:SandwichCellCondition}
xc_{T,m,B}^{\lambda}\equiv
\sum_{U\in\Tcal(\lambda),n\in\sandbasis}r_{TU}^{x}c_{U,n,B}^{\lambda}\pmod{\algebra^{\rsandorder[\Pcal]\lambda}}.
\end{gather}
Similarly for right multiplication by $x$.

\item There exists a free 
$\algebra$-$\sand[\lambda]$-bimodule $\dmod[\lambda]$,
a free  $\sand[\lambda]$-$\algebra$-bimodule
$\modd[\lambda]$, and an $\algebra$-bimodule isomorphism
\begin{gather}\label{Eq:SandwichCellAlgebra}
\calg=\algebra^{\gsandorder[\Pcal]\lambda}/\algebra^{\rsandorder[\Pcal]\lambda}\cong\dmod[\lambda]\otimes_{\sand[\lambda]}\modd[\lambda],
\end{gather}
where $\algebra^{\gsandorder[\Pcal]\lambda}$ ($\algebra^{\rsandorder[\Pcal]\lambda}$) stands for 
the $\K$-submodule of $\algebra$ spanned by
\begin{gather*}
\set{c_{U,n,V}^{\mu}|\mu\in\Pcal, \mu \gsandorder[\Pcal] \lambda \,(\mu\rsandorder[\Pcal]\lambda),U\in\Tcal(\mu),V\in\Bcal(\mu),n\in\sandbasis[\mu]}.
\end{gather*}
We call $\calg$ the \emph{cell algebra}, and $\dmod[\lambda]$ and 
$\modd[\lambda]$ left and right \emph{cell modules}.

\end{enumerate}
The algebra $\algebra$ is a \emph{sandwich cellular algebra} if it
has a sandwich cell datum.
\end{Definition}

\begin{Definition}\label{D:SandwichInvolutive}
In the setup of \autoref{D:SandwichCellularAlgebra}, suppose that $\Tcal(\lambda)=\Bcal(\lambda)$ for all $\lambda\in\Pcal$, and that there is an antiinvolution $(\placeholder)^{\star}\colon\algebra\to\algebra$ 
and order two bijections $(\placeholder)^{\star}\colon\sandbasis\to\sandbasis$  such that:
\begin{enumerate}[label=\upshape(AC${}_{4}$\upshape)]

\item We have $(c_{T,m,B}^{\lambda})^{\star}\equiv c_{B,m^{\star},T}^{\lambda}\pmod{\algebra^{\rsandorder[\Pcal]\lambda}}$.

\end{enumerate}
In this case we call the datum 
\emph{involutive} and write $\big(\Pcal,\Tcal,(\sand,\sandbasis),C,(\placeholder)^{\star}\big)$ 
for it.
\end{Definition}

\begin{Remark}\label{E:SandwichCellularAlgebra}
Sandwich cellular algebras have been rediscovered many times over the years, see e.g. \cite{Br-gen-matrix-algebras} for an early reference.
Two special cases of involutive sandwich cellular algebras are:
\begin{enumerate}

\item If $\sand\cong\K$, then 
$\big(\Pcal,\Tcal,(\sand\cong\K,\sandbasis=\{1\}),
C,(\placeholder)^{\star}\big)$ is a cell 
datum and $\algebra$ is a \emph{cellular 
algebra} in the sense of \cite{GrLe-cellular}.

\item If all sandwiched algebras are commutative, then 
$\big(\Pcal,\Tcal,(\sand,\sandbasis),C,(\placeholder)^{\star}\big)$ is an affine cell datum and $\algebra$ is an \emph{affine cellular algebra} in the sense of \cite{KoXi-affine-cellular}.

\end{enumerate}
A useful illustration is the following:
\begin{gather*}
\text{cellular: }
c_{T,1,B}^{\lambda}
\leftrightsquigarrow
\scalebox{0.8}{$\begin{tikzpicture}[anchorbase,scale=1]
\draw[mor] (0,-0.5) to (0.25,0) to (0.75,0) to (1,-0.5) to (0,-0.5);
\node at (0.5,-0.25){$B$};
\draw[mor] (0,0.5) to (0.25,0) to (0.75,0) to (1,0.5) to (0,0.5);
\node at (0.5,0.25){$T$};
\draw[thick,->] (1.5,0)node[right]{$\sand\cong\K$} to (1,0);
\end{tikzpicture}$}
,\quad
\text{affine cellular: }
c_{T,m,B}^{\lambda}
\leftrightsquigarrow
\scalebox{0.8}{$\begin{tikzpicture}[anchorbase,scale=1]
\draw[mor] (0,-0.5) to (0.25,0) to (0.75,0) to (1,-0.5) to (0,-0.5);
\node at (0.5,-0.25){$B$};
\draw[mor] (0,1) to (0.25,0.5) to (0.75,0.5) to (1,1) to (0,1);
\node at (0.5,0.75){$T$};
\draw[mor] (0.25,0) to (0.25,0.5) to (0.75,0.5) to (0.75,0) to (0.25,0);
\node at (0.5,0.25){$m$};
\draw[thick,->] (1.5,0.25)node[right]{commutative $\sand$} to (1,0.25);
\end{tikzpicture}$}
,
\end{gather*}
\begin{gather}\label{Eq:SandwichMnemonic}
\text{sandwich cellular: }
c_{T,m,B}^{\lambda}
\leftrightsquigarrow
\scalebox{0.8}{$\begin{tikzpicture}[anchorbase,scale=1]
\draw[mor] (0,-0.5) to (0.25,0) to (0.75,0) to (1,-0.5) to (0,-0.5);
\node at (0.5,-0.25){$B$};
\draw[mor] (0,1) to (0.25,0.5) to (0.75,0.5) to (1,1) to (0,1);
\node at (0.5,0.75){$T$};
\draw[mor] (0.25,0) to (0.25,0.5) to (0.75,0.5) to (0.75,0) to (0.25,0);
\node at (0.5,0.25){$m$};
\draw[thick,->] (1.5,0.25)node[right]{general $\sand$} to (1,0.25);
\end{tikzpicture}$}
,
\end{gather}
where we assume the existence of an antiinvolution for the 
top two pictures.
\end{Remark}

\begin{Example}
There are several important reasons to work with sandwich cellularity: first, many algebras admit ``nice'' sandwich bases, whereas their cellular bases are not well behaved; this occurs for \eg\, the Brauer algebra, as shown in \cite{FiGr-canonical-cases-brauer,SaSn-triangular}, KLR algebras of types A and C, cf. \cite{MaTu-klrw-algebras}, or various variations of diagram algebras, see e.g. \cite{LaReMoDu-affine-tl-monoid,Li-cyclic-monoid}. 

Moreover, some algebras are not cellular at all but nevertheless possess a well-structured sandwich datum, such as KLR algebras of more exotic types, see e.g. \cite{MaTu-klrw-algebras-bad,MaTu-klrw-crystal}, algebras where the antiinvolution is ``messy'', see e.g. \cite{GuWi-almost-cellular}, or affine diagram algebras \cite{KoXi-affine-cellular,HeTu-affine-monoid}. 

Finally, the concept of sandwich cellularity is directly motivated by, and partially generalizes, the corresponding theory for semigroups; compare \eg\, \cite{GaMaSt-irreps-semigroups} with \autoref{T:SandwichCMP} below. In particular, ``most'' semigroup rings are sandwich cellular, see e.g. \cite{Ea-cellular-semigroups}.
\end{Example}

\begin{Example}
A crucial example are the diagram algebras from above but without the M{\"obius} strip. They have a natural sandwich 
cellular structure, cf. \cite{FrStTu-twists}.
\end{Example}

The most important theorem concerning sandwich cellular algebras is $\textit{$H$-reduction}$, which relates simple $\algebra$-modules to simple modules over the sandwiched algebras. We begin with the following definition and lemma. Below, 
all modules are left modules unless otherwise stated.

\begin{Definition}
An \emph{apex} of an $\algebra$-module $\module$, if it exists, is a maximal $\lambda\in\Pcal$ such that
\begin{gather*}
\K\set[\big]{c_{T,m,B}^{\lambda}|T\in\Tcal(\lambda),B\in\Bcal(\lambda),m\in\sandbasis[\lambda]}
\end{gather*}
does not annihilate $\module$.
The same notion is used for $\calg$ instead of $\algebra$.
\end{Definition}

\begin{Lemma}\label{apexlem}
Every simple $\algebra$-module has an apex $\lambda\in\Pcal$ (uniquely associated to it), 
and similarly for 
simple $\calg$-modules.
\end{Lemma}

\begin{proof}
Standard, see e.g. \cite[Lemma 2.15]{TuVa-handlebody} .
\end{proof}

\begin{Notation}
For a class $X$ of modules we write $X\,/\cong$ for the set 
of isomorphism classes obtained from $X$ by 
identifying isomorphic modules.
\end{Notation}

\begin{Lemma}(\emph{$H$-reduction})\label{T:SandwichCMP}
Let $\K$ be a field.
\begin{enumerate}

\item If $\lambda\in\Pcal$ is an apex and $\sand$ is Artinian, then we have bijections
\begin{gather*}
\begin{aligned}
\text{Start}\colon&
\\
&\mystrut
\\
\text{$J$-reduction}\colon&
\\
&\mystrut
\\
\text{$H$-reduction}\colon&
\end{aligned}
\begin{gathered}
\left\{
\text{simple $\algebra$-modules with apex $\lambda$}
\right\}/\cong
\\
\xleftrightarrow{1:1}
\\
\left\{
\text{simple $\calg$-modules with apex $\lambda$}
\right\}/\cong
\\
\xleftrightarrow{1:1}
\\
\left\{
\text{simple $\sand$-modules}
\right\}/\cong
.
\end{gathered}
\end{gather*}
Moreover, every simple $\algebra$-module arises this way.

\item All $\algebra$-modules with apex $\lambda$ have composition 
factors of apex $\mu$ with $\mu\lsandorder[\Pcal]\lambda$.

\end{enumerate}

\end{Lemma}

\begin{proof}
Standard, see e.g. \cite[Theorem 2.16]{TuVa-handlebody}.
\end{proof}

\begin{Notation}
In case \autoref{T:SandwichCMP} applies, we write $\apex \subset \Pcal$ for the set of apexes for the simple $\algebra$-modules.
\end{Notation}

We would now like to describe a sandwich cell datum for $\mxmon{\paraa, \parab, \parac}$. First, fix a nonnegative odd integer $r$ and consider the rational functions 
\begin{align}\label{gencond}
Z_{\paraa}(T) = \frac{p_\paraa(T)}{1-T^r}, \hspace{1cm} Z_{\parab}(T) = \frac{p_\parab(T)}{1- T^r}, \hspace{1cm} Z_{\parac}(T) = \frac{p_\parac(T)}{1-T^r},
\end{align}
with polynomial degrees satisfying the conditions of \autoref{D:DAlgebrasTheAlgebras}. To motivate the sandwich cellular structure, we have the following lemma, cf. \cite[Lemma 4B.9]{Tu-sandwich-cellular}:

\begin{Lemma}\label{L:DAlgebrasDMonFactor}
Let $\mxmon{1}$ be symmetric.
For $a\in\mxmon{1}$ there is a unique factorization of the form $a=\tau\circ\sigma_{\lambda}^{\mathcal{M}}\circ\beta$ such that:
\begin{enumerate}[label=(\roman*)]
\item $\beta$ 
and $\tau$ have a minimal 
number of crossings, handle dots, and Möbius dots;
\item $\beta$ contains no cups, $\Delta$, or $\eta$;
\item $\tau$ contains no caps, $\mu$, or $\epsilon$;
\item $\sigma_{\lambda}^\mathcal{M}\in  \mathcal{M} \wr \sym[\lambda]$ for minimal $\lambda$, where $\mathcal{M}$ is a monoid with following presentation:
\begin{align*}
\mathcal{M} = \langle a, b \, \vert \, ab=ba, \, ab = 
b^3, \, a^K=a^{K-r} \rangle.
\end{align*}
\end{enumerate}

Similarly for $a\in\mxmon{1}$ when $\mxmon{1}$ is planar, but with $\sigma_{\lambda}^\mathcal{M}\in \mathcal{M}^\lambda$.\qed
\end{Lemma}

\begin{proof}
(i)--(iii) are as in the proof of \autoref{T:Mcob}. We first note that the Möbius dots move freely along strands as a result of relations \autoref{mobmu} and \autoref{mobs} (the mirror image diagrams are implied by the other relations). The holes also move freely, since
\begin{align*}
\begin{tikzpicture}[anchorbase]
\draw[usual,hol=0.25] (0,0) to (0.5,0.5) to (1,0);
\draw[usual] (0.5,0.5) to (0.5,1);
\end{tikzpicture}
=
\begin{tikzpicture}[anchorbase]
\draw[usual] (0,0) to (0.5,0.5) to (1,0);
\draw[usual,hol=0.5] (0.5,0.5) to (0.5,1);
\end{tikzpicture}
=
\begin{tikzpicture}[anchorbase]
\draw[usual,hol=0.75] (0,0) to (0.5,0.5) to (1,0);
\draw[usual] (0.5,0.5) to (0.5,1);
\end{tikzpicture}
,\quad
\begin{tikzpicture}[anchorbase]
\draw[usual,hol=0.25] (0,0) to (1,1);
\draw[usual] (1,0) to (0,1);
\end{tikzpicture}
=
\begin{tikzpicture}[anchorbase]
\draw[usual,hol=0.75] (0,0) to (1,1);
\draw[usual] (1,0) to (0,1);
\end{tikzpicture}
\end{align*}
(and their mirror images) follows using the relations \autoref{startrel}, \autoref{hol1}--\autoref{hol4}, and \autoref{frob}. As a result, any holes or Möbius dots that appear on a strand connecting the bottom of the diagram to the top can be moved to the middle. The conditions (i)--(iii) then follow from \cite[Lemma 4B.9]{Tu-sandwich-cellular}.

For the final condition (iv), we have holes and Möbius dots present on $\lambda$ separate strands connected the bottom of the diagram to the top, which may or may not contain crossings; it is straightforward to see that the configuration of the strands is (uniquely) specified by an element of $\mathcal{M} \wr \sym[\lambda]$ in the former case, and of $\mathcal{M}^\lambda$ in the latter.
\end{proof}

\begin{Notation}
With respect to the factorization in \autoref{L:DAlgebrasDMonFactor}, 
we call $\lambda$ the number of \emph{through strands}, 
$\beta$ the \emph{bottom}, $\tau$ the \emph{top}, and 
$\sigma_{\lambda}^\mathcal{M}$ the \emph{middle} of $a$.
\end{Notation}

\begin{Proposition}\label{mobcelldatum}
The following specifies an involutive sandwich cell datum for $\mxmon{\paraa, \parab, \parac}$:
\begin{itemize}
\item The middle poset $\Pcal$ is given by the set of nonnegative integers with reverse standard ordering as the sandwich order $\sandorder[\Pcal]$;
\item For $\lambda \in \Pcal$, the top/bottom sets are given by the top/bottom of all $a \in \mxmon{1}$ with a fixed number of through strands $\lambda$;
\item For $\lambda \in \Pcal$, we have $\sand= \mathcal{M} \wr \sym[\lambda]$, or $\sand= \mathcal{M}^\lambda$ in the planar case (we really mean the $\K$-linearization of the respective monoids when referring to the sandwiched algebras);
\item The injective map $C$ is the obvious one given \autoref{L:DAlgebrasDMonFactor};
\item The involution $(\placeholder)^{\star}$ reflects individual diagrams about the horizontal.
\end{itemize}
\end{Proposition}

\begin{proof}
The elements of $\mxmon{1}$ are basis for $\mxmon{\paraa, \parab, \parac}$ by construction, so (AC$_1$) follows. 

For (AC$_2$), we first note that in the case that holes and Möbius dots are the identity (forcing $p_\paraa(T) = p_\parab(T) =p_\parac(T) =\paraa_0$ and $r=1$ in \autoref{gencond}), there is an analogous cell datum where $\mathcal{M}$ is the trivial monoid \cite[Proposition 4B.10]{Tu-sandwich-cellular}. Since holes and Möbius dots move freely, we then only need to show that multiplication in the sandwiched algebras $\mathcal{M} \wr \sym[\lambda]$, or $\mathcal{M}^\lambda$ in the planar case, does not introduce element-dependent coefficients. The only relevant relations are the Möbius relation, which does not introduce any coefficients, and the handle relation:
\begin{gather*}
\begin{tikzpicture}[anchorbase]
\draw[usual,hol=0.5] (0,0) to (0,0.5)node[left]{$K$} to (0,1);
\end{tikzpicture}
-a_1\cdot 
\begin{tikzpicture}[anchorbase]
\draw[usual,hol=0.5] (0,0) to (0,0.5)node[left]{$(K-1)$} to (0,1);
\end{tikzpicture}
+a_2\cdot 
\begin{tikzpicture}[anchorbase]
\draw[usual,hol=0.5] (0,0) to (0,0.5)node[left]{$(K-2)$} to (0,1);
\end{tikzpicture}
-\dots+(-1)^{\deg q}a_{\deg q}\cdot
\begin{tikzpicture}[anchorbase]
\draw[usual,hol=0.5] (0,0) to (0,0.5)node[left]{$(K-\deg q)$} to (0,1);
\end{tikzpicture}
=0,
\end{gather*}
which, from \autoref{gencond}, in our context reduces to 
\begin{gather*}
\begin{tikzpicture}[anchorbase]
\draw[usual,hol=0.5] (0,0) to (0,0.5)node[left]{$K$} to (0,1);
\end{tikzpicture}
= 
\begin{tikzpicture}[anchorbase]
\draw[usual,hol=0.5] (0,0) to (0,0.5)node[left]{$(K-r)$} to (0,1);
\end{tikzpicture},
\end{gather*}
and so does not introduce any coefficients either.

Furthermore, it is straightforward to verify that left  $\Delta(\lambda)$ (right $\nabla(\lambda)$) cells with elements defined by taking $\K$-linear combinations of the top and middle (middle and bottom) of all $a \in \mxmon{1}$ with a fixed number of through strands $\lambda$ satisfy the final condition (AC$_3$).

Finally, given the map $(\placeholder)^{\star}$, the conditions of \autoref{D:SandwichInvolutive} are immediate.
\end{proof}

Before we describe some properties of the sandwich cell datum for $\mxmon{\paraa, \parab, \parac}$ given above, we briefly recall the notion of cells for based algebras, and their connection with sandwich cell datum; this will be important for the remaining two sections.

Suppose that the algebra $\algebra$ has a basis $\cellbasis$, so that $(\algebra,\cellbasis)$ is a \emph{(based) pair}. For $a,b,c\in\cellbasis$ 
we write $b\usummand ca$ if, when 
$ca$ is expanded in terms of $\cellbasis$, 
$b$ appears with a nonzero coefficient in $ca$.  

\begin{Definition}
The preorders
\begin{gather*}
(a\leq_{l}b)\Leftrightarrow
\exists c:b\usummand ca
,\quad
(a\leq_{r}b)\Leftrightarrow
\exists d:b\usummand ad
,\quad
(a\leq_{lr}b)\Leftrightarrow
\exists c,d:b\usummand cad
\end{gather*}
on $\cellbasis$ are called left, right and two-sided \emph{cell orders}.
\end{Definition}

\begin{Remark}\label{R:SandwichOrder}
Our convention for the 
cell orders follows the one commonly used
in the theory of cellular algebras, but the opposite of the one typically seen in monoid theory.
\end{Remark}

\begin{Definition}\label{D:SandwichCells}
Equivalence classes defined using the relations
\begin{gather*}
(a\sim_{l}b)\Leftrightarrow
(a\leq_{l}b\text{ and }b\leq_{l}a)
,\;
(a\sim_{r}b)\Leftrightarrow
(a\leq_{r}b\text{ and }b\leq_{r}a)
,\;
(a\sim_{lr}b)\Leftrightarrow
(a\leq_{lr}b\text{ and }b\leq_{lr}a)
\end{gather*}
are called 
left, right respectively two-sided \emph{cells}. An \emph{$H$-cell} $\hcell=\hcell(\lcell,\rcell)=\lcell\cap\rcell$ is an intersection of a left cell $\lcell$ and a right cell $\rcell$.
\end{Definition}

\begin{Notation}
We also say \emph{$J$-cells} instead of two-sided cells, following 
the notation in \cite{Gr-structure-semigroups}.
\end{Notation}

Since the left, right and $J$-cell orders induce preorders on the left, right and $J$-cells \cite[Lemma 2B.5]{Tu-sandwich-cellular}, it makes sense to index the cells by their respective preorders. For example, given an indexing set $\Pcal$ for the $J$-cells, the notation $\jcell_\lambda$, $\lambda \in \Pcal$, refers to fixed $J$-cell.

We now introduce the notion of an idempotent $J$-cell:

\begin{Notation}
An equation {\etc} holds \emph{up to higher order terms} if 
it holds modulo $\jideal= \K\big\{\bigcup_{\mu\in\Pcal,\mu>_{lr}\lambda}\jcell_{\mu}\big\}$, which is a two-sided ideal in $\algebra$.
\end{Notation}

\begin{Definition}\label{D:SandwichIdempotent}
If the $\K$-linear span of a $J$-cell
$\jcell$ contains a (nonzero) pseudo-idempotent 
up to higher order terms, 
that is, $e\in\K\jcell\setminus\{0\}$ with 
$e^{2}=s(e)\cdot e\pmod{\jideal}$ for $s(e)\in\K\setminus\{0\}$, 
then we call $\jcell$ \emph{idempotent}. 
We say $\jcell$ is \emph{strictly idempotent} if 
$\jcell$ itself contains a pseudo-idempotent up to higher order terms. 
We use the same terminology for $H$-cells.
\end{Definition}

\begin{Notation}
We write $\jcell_{\lambda}(e)$, 
$\hcell_{\lambda}(e)$ \etc \, for idempotent cells.
\end{Notation}

\begin{Lemma}
For the pair $(\algebra,\cellbasis)$:
\begin{enumerate}
\item Every $H$-cell is contained in some $J$-cell, and every $J$-cell is a disjoint union of $H$-cells.

\item If $s(e)\in\K\setminus\{0\}$ is invertible, then $\sand=\K\hcell_{\lambda}(e)/\jideal$ is an algebra with identity $\frac{1}{s(e)}e$. This algebra is a subalgebra of $\algebra^{\leq_{lr}\lambda}=\algebra/\jideal$ and 
$\algebra^{\lambda}=\algebra^{\leq_{lr}\lambda}\cap\K\jcell_{\lambda}$.
\end{enumerate}
\end{Lemma}

\begin{proof}
See \cite[Lemma 2B.15]{Tu-sandwich-cellular}.
\end{proof}

Furthermore, in the special case that the pair $(\algebra,\cellbasis)$ is \emph{involutive} \ie\, it admits an order two bijection
${\placeholder}^{\star}\colon\cellbasis\to\cellbasis$ 
that gives rise to an antiinvolution on $\algebra$, then there are mutually inverse bijections $\{\text{left cells}\}\leftrightarrow\{\text{right cells}\}$
that preserve containment in $J$-cells, cf. \cite[Lemma 2B.18]{Tu-sandwich-cellular}.

The connection to sandwich cellular algebras is given by:

\begin{Lemma}\label{T:SandwichIsSandwich}
For any sandwich cellular algebra $\algebra$ 
we get a pair $(\algebra,\cellbasis)$ for which the above theory applies 
with the following cell structure:
\begin{enumerate}[label=\emph{\upshape(\roman*)}]

\item The basis $\cellbasis$ is the 
sandwich cellular basis.

\item The poset can be taken (potentially changing the order) to 
be $\Pcal=(\Pcal,<_{lr})$.

\item The set $\Tcal(\lambda)$ indexes the right cells within $\jcell_{\lambda}$, and $\Bcal(\lambda)$ indexes the left cells within $\jcell_{\lambda}$.

\item All left, right and 
$H$-cells within one $J$-cell are of the same size.

\item If $\jcell_{\lambda}$ is idempotent, then 
the respective sandwiched algebra and cell algebra are isomorphic to
$\sand$ and $\algebra^{\lambda}$.

\end{enumerate}
\end{Lemma}

\begin{proof}
See \cite[Theorem 2B.19]{Tu-sandwich-cellular}.
\end{proof}

In other words, to any (involutive) sandwich cellular algebra, we can associate a(n involutive) \emph{sandwich pair} $(\algebra,\cellbasis)$. Finally, given the sandwich cell datum described in \autoref{mobcelldatum}, we describe the cells of $\mxmon{\paraa, \parab, \parac} $, cf. \cite[Proposition 4B.10]{Tu-sandwich-cellular}:

\begin{Proposition}\label{P:SandwichStuff}
We have the following for the pair $\big(\mxmon{\paraa, \parab, \parac} ,\mxmon{1}\big)$.
\begin{enumerate}

\item The $J$-cells of $\mxmon{\paraa, \parab, \parac}$
are given by diagrams with a fixed number of through strands $\lambda$.
The $\leq_{lr}$-order is a total order and increases as the number of through strands decreases. See \autoref{Eq:DAlgebrasTable} for a summary.

\item The left (right) cells of $\mxmon{\paraa, \parab, \parac}$ 
are given by diagrams where one fixes 
the bottom (top) of the diagram.
The $\leq_{l}$ and the $\leq_{r}$-order increases as the number of through strands decreases. For $\#\lcell_{\lambda}=\#\rcell_{\lambda}$ see \autoref{Eq:DAlgebrasTable}.

\item Assume that all evaluation parameters are invertible. Then all $J$-cells 
of $\mxmon{\paraa, \parab, \parac}$ are strictly idempotent. The general case is discussed at the end of this section.

\item The pair $\big(\mxmon{\paraa, \parab, \parac} ,\mxmon{1} \big)$ is an involutive sandwich pair,
that for non-planar $\mxmon{1}$ comes neither from a cellular nor an affine 
cellular algebra.

\end{enumerate}
\end{Proposition}

We have the following table:
\begin{gather}\label{Eq:DAlgebrasTable}
\begin{gathered}
\scalebox{0.9}{\begin{tabular}{c||c|c|c}
Monoid & $\Pcal=\apex$ (*)  &  $\#\lcell_{\lambda}$ & $\sand$ \\
\hline
\hline
$\mpamon{1}$ & $\{n{<_{lr}}n{-}1{<_{lr}}\dots{<_{lr}}0\}$  & $\sum_{t=0}^{n}\begin{Bsmallmatrix}n\\t\end{Bsmallmatrix}\binom{t}{\lambda}(3K)^{t-\lambda}$ & $\mathcal{M} \wr\sym[\lambda]$ \\
\hline
$\mppamon{1}$ & $\{n{<_{lr}}n{-}1{<_{lr}}\dots{<_{lr}}0\}$ & $\tfrac{4\lambda+2}{2n+2\lambda+2}\binom{2n}{(2n{-}2\lambda)/2}(3K)^{n-\lambda}$ & $\mathcal{M}^\lambda$ \\
\hline
$\mrobrmon{1}$ & $\{n{<_{lr}}n{-}1{<_{lr}}\dots{<_{lr}}0\}$ & $\sum_{t=0}^{n}\binom{n}{\lambda}\binom{n-\lambda}{2t}(2t-1)!!(3K)^{n-\lambda-t}$ &  $\mathcal{M} \wr \sym[\lambda]$ \\
\hline
$\mmomon{1}$ & $\{n{<_{lr}}n{-}1{<_{lr}}\dots{<_{lr}}0\}$ & $\sum_{t=0}^{n}\tfrac{\lambda+1}{\lambda+t+1}\binom{n}{\lambda+2t}\binom{\lambda+2t}{t}(3K)^{n-\lambda-t}$ & $\mathcal{M}^\lambda$ \\
\hline
$\mbrmon{1}$ & $\{n{<_{lr}}n{-}2{<_{lr}}\dots{<_{lr}}0,1\}$  & $\binom{n}{\lambda}(n-\lambda-1)!!(3K)^{(n-\lambda)/2}$ & $\mathcal{M} \wr \sym[\lambda]$ \\
\hline
$\mtlmon{1}$ & $\{n{<_{lr}}n{-}2{<_{lr}}\dots{<_{lr}}0,1\}$ & $\tfrac{2\lambda+2}{n+\lambda+2}\binom{n}{(n-\lambda)/2}(3K)^{(n-\lambda)/2}$ & $\mathcal{M}^\lambda$ \\
\hline
$\mromon{1}$ & $\{n{<_{lr}}n{-}1{<_{lr}}\dots{<_{lr}}0\}$ & $\binom{n}{\lambda}(3K)^{n-\lambda}$ &  $\mathcal{M} \wr \sym[\lambda]$ \\
\hline
$\mpromon{1}$ & $\{n{<_{lr}}n{-}1{<_{lr}}\dots{<_{lr}}0\}$  & $\binom{n}{\lambda}(3K)^{n-\lambda}$ &  $\mathcal{M}^\lambda$ \\
\hline
$\msym$ & $\{n\}$ &  $1$ & $\mathcal{M} \wr \sym[\lambda]$ \\
\hline
$\mpsym$ & $\{n\}$  & $1$ & $\mathcal{M}^\lambda$ \\
\end{tabular}}
.
\\
\text{(*) This assumes that all evaluation parameters are invertible;} 
\\
\text{the general case is stated at the end of the proof of \autoref{P:SandwichStuff}.}
\end{gathered}
\end{gather}

\begin{proof}
Let $\algebra=\mxmon{\paraa, \parab, \parac}$. 

\textit{(a)}. The result is standard and omitted.

\textit{(b)}. As above, everything is easy to see, except the counts for $\#\lcell_{\lambda}$: since holes and Möbius dots move freely along strands, we need to count the number of possible configurations with dotted non-through strands. Ignoring holes and Möbius dots, using \cite[Proposition 4B.10]{Tu-sandwich-cellular} for the overall counts, together with simple combinatorial arguments, we obtain that part of each count that involves the number of parts not attached to through strands; call this $x$. Since each part can have up to $K-1$ holes and up to two Möbius dots according to the Möbius relation, we must multiply the relevant part of the count by $(3K)^x$.

\begin{Example}
    Let us do the Temperley--Lieb case as an example, say $n=3$, $\lambda=1$ and $K=2$. The monoid has two half diagrams indexing the left cells:
\begin{gather*}
\begin{tikzpicture}[anchorbase]
\draw[usual] (1,0) to[out=90,in=180] (1.25,0.25) to[out=0,in=90] (1.5,0);
\draw[usual] (0.5,0) to (0.5,0.5);
\end{tikzpicture}
,
\quad
\begin{tikzpicture}[anchorbase,xscale=-1]
\draw[usual] (1,0) to[out=90,in=180] (1.25,0.25) to[out=0,in=90] (1.5,0);
\draw[usual] (0.5,0) to (0.5,0.5);
\end{tikzpicture} \quad
\Rightarrow \quad
\text{$(3-1)/2$ non-through components}
.
\end{gather*}
The non-through component can be labeled with $0$ or $1$ handle dot
\begin{gather*}
\begin{tikzpicture}[anchorbase,xscale=-1]
\draw[usual] (1,0) to[out=90,in=180] (1.25,0.25) to[out=0,in=90] (1.5,0);
\end{tikzpicture}
,\quad
\begin{tikzpicture}[anchorbase,xscale=-1]
\draw[usual,hol=0.5] (1,0) to[out=90,in=180] (1.25,0.25) to[out=0,in=90] (1.5,0);
\end{tikzpicture}
.
\end{gather*}
and each of these, in turn, can carry $0$, $1$ or $2$ M{\"o}bius dots, e.g.
\begin{gather*}
\begin{tikzpicture}[anchorbase,xscale=-1]
\draw[usual] (1,0) to[out=90,in=180] (1.25,0.25) to[out=0,in=90] (1.5,0);
\end{tikzpicture}
,\quad
\begin{tikzpicture}[anchorbase,xscale=-1]
\draw[usual,mob=0.5] (1,0) to[out=90,in=180] (1.25,0.25) to[out=0,in=90] (1.5,0);
\end{tikzpicture}
,\quad
\begin{tikzpicture}[anchorbase,xscale=-1]
\draw[usual,mob=0.33,mob=0.66] (1,0) to[out=90,in=180] (1.25,0.25) to[out=0,in=90] (1.5,0);
\end{tikzpicture}
.
\end{gather*}
The overall scaling is therefore $(3\cdot 2)^{(3-1)/2}$. 
\end{Example}

All other diagram algebras 
can be treated verbatim. (For the planar partition case one can use the isomorphism to the Temperley--Lieb case with twice as many strands, see e.g. \cite[(1-5)]{HaRa-partition-algebras} for this isomorphism.)

\textit{(c)}. To see that $\Pcal = \apex$, we use the following lemma.

\begin{Lemma}\label{L:Idempotents}
An $H$-cell is strictly idempotent if and only if it is strictly idempotent when seen as an $H$-cell of the otherwise same algebra but with
\begin{gather*}
\begin{tikzpicture}[anchorbase]
\draw[usual,hol=0.5] (0,0) to (0,0.5) to (0,1);
\end{tikzpicture}
 = \begin{tikzpicture}[anchorbase]
\draw[usual,mob=0.5] (0,0) to (0,0.5) to (0,1);
\end{tikzpicture}
=
\begin{tikzpicture}[anchorbase]
\draw[usual] (0,0) to (0,0.5) to (0,1);
\end{tikzpicture}
.
\end{gather*}
(So, all handle and M{\"o}bius dots disappear.)
\end{Lemma}

\begin{proof}
The only thing to observe here is that 
any pseudo idempotent with (handle or M{\"o}bius) 
dots must have an underlying diagram without any dots.
\end{proof}

Note that, by \autoref{L:Idempotents}, we can now ignore holes and Möbius dots, and set $p_\paraa(T)=\paraa_0$. Then there exist idempotents in all $J$-cells by \cite[Proposition 4B.10]{Tu-sandwich-cellular} and \cite[Section 2]{FrStTu-twists}; these idempotents remain in the general case, though with possibly a different normalization factor.

We now describe the full list of apexes. Let $\mathbf{x}$ be either of the set of parameters $\paraa,\parab,\parac$.

\begin{gather}\label{Eq:DAlgebrasTable1}
\begin{gathered}
\scalebox{0.9}{\begin{tabular}{c||c|c}
Monoid & $\apex: x_k=0$ for all $k, \mathbf{x}$ &  $\apex:x_k \neq 0$ for some $k, \mathbf{x}$  \\
\hline
\hline
$\mpamon{1}$ & $\{0,\dots,n\}\setminus\{0\}$  & $\{0,\dots,n\}$   \\
\hline
$\mppamon{1}$ & $\{0,\dots,n\}\setminus\{0\}$  & $\{0,\dots,n\}$   \\
\hline
$\mrobrmon{1}$ & $\{\text{0 or 1},\dots,n-2,n\}\setminus\{0\}$    & $\{0,\dots,n\}$  \\
\hline
$\mmomon{1}$ & $\{\text{0 or 1},\dots,n-2,n\}\setminus\{0\}$   & $\{0,\dots,n\}$   \\
\hline
$\mbrmon{1}$ & $\{\text{0 or 1},\dots,n-2,n\}\setminus\{0\}$ $(\forall k \geq 1)$   & $\{\text{0 or 1},\dots,n-2,n\}$ $(\exists  k \geq 1)$  \\
\hline
$\mtlmon{1}$ & $\{\text{0 or 1},\dots,n-2,n\}\setminus\{0\}$ $(\forall k \geq 1)$  &  $\{\text{0 or 1},\dots,n-2,n\}$ $(\exists k \geq 1)$   \\
\hline
$\mromon{1}$ & $\{n\}$  & $\{0,\dots,n\}$  \\
\hline
$\mpromon{1}$ & $\{n\}$   & $\{0,\dots,n\}$   \\
\hline
$\msym$ & $\{n\}$  &  $\{n\}$  \\
\hline
$\mpsym$ & $\{n\}$  & $\{n\}$  \\
\end{tabular}}
\end{gathered}
\end{gather}

\begin{Lemma}\label{L:Apex}
The set of apexes is as above.
\end{Lemma}

\begin{proof}
The following can be used to create idempotents while reducing the number of through strands and without producing internal components, cf. \cite[Theorem 2C.6]{FrStTu-twists}:
\begin{gather*}
\begin{tikzpicture}[anchorbase]
\draw[usual] (0,0) to[out=90,in=90] (0.5,0);
\draw[usual, dot] (0.5,0.5) to (0.5,0.25);
\draw[usual] (0.25,0.15) to[out=90,in=270] (0,0.5);
\end{tikzpicture}
\circ
\begin{tikzpicture}[anchorbase]
\draw[usual] (0,0) to[out=90,in=90] (0.5,0);
\draw[usual, dot] (0.5,0.5) to (0.5,0.25);
\draw[usual] (0.25,0.15) to[out=90,in=270] (0,0.5);
\end{tikzpicture}
=
\begin{tikzpicture}[anchorbase]
\draw[usual] (0,0) to[out=90,in=90] (0.5,0);
\draw[usual, dot] (0.5,0.5) to (0.5,0.25);
\draw[usual] (0.25,0.15) to[out=90,in=270] (0,0.5);
\end{tikzpicture}
,\quad
\begin{tikzpicture}[anchorbase]
\draw[usual] (0,0) to[out=90,in=90] (0.5,0);
\draw[usual] (0.5,0.5) to[out=270,in=270] (1,0.5);
\draw[usual] (1,0) to (0,0.5);
\end{tikzpicture}
\circ
\begin{tikzpicture}[anchorbase]
\draw[usual] (0,0) to[out=90,in=90] (0.5,0);
\draw[usual] (0.5,0.5) to[out=270,in=270] (1,0.5);
\draw[usual] (1,0) to (0,0.5);
\end{tikzpicture}
=
\begin{tikzpicture}[anchorbase]
\draw[usual] (0,0) to[out=90,in=90] (0.5,0);
\draw[usual] (0.5,0.5) to[out=270,in=270] (1,0.5);
\draw[usual] (1,0) to (0,0.5);
\end{tikzpicture}
.
\end{gather*}
In addition, whenever $\xmon{1}$ makes use of  
\begin{gather*}
\begin{tikzpicture}[anchorbase]
\draw[usual, dot] (0,0.6) to (0,0.35);
\draw[usual, dot] (0,0) to (0,0.25);
\end{tikzpicture}
\circ
\begin{tikzpicture}[anchorbase]
\draw[usual, dot] (0,0.6) to (0,0.35);
\draw[usual, dot] (0,0) to (0,0.25);
\end{tikzpicture}
=
\begin{tikzpicture}[anchorbase]
\draw[usual, dot] (0,0.6) to (0,0.35);
\draw[usual, dot] (0,0) to (0,0.25);
\end{tikzpicture}
\hspace{0.5cm} \text{or} \hspace{0.5cm}
\begin{tikzpicture}[anchorbase]
\draw[usual] (0,0) to[out=90,in=90] (0.5,0);
\draw[usual] (0,0.5) to[out=270,in=270] (0.5,0.5);
\end{tikzpicture}
\circ
\begin{tikzpicture}[anchorbase]
\draw[usual] (0,0) to[out=90,in=90] (0.5,0);
\draw[usual] (0,0.5) to[out=270,in=270] (0.5,0.5);
\end{tikzpicture}
=
\begin{tikzpicture}[anchorbase]
\draw[usual] (0,0) to[out=90,in=90] (0.5,0);
\draw[usual] (0,0.5) to[out=270,in=270] (0.5,0.5);
\end{tikzpicture},
\end{gather*}
to create idempotents, simply add the appropriate amount of holes and M{\"o}bius dots to the top component that results in a nonzero multiplication factor \eg\, 
\begin{gather*}
\begin{tikzpicture}[anchorbase]
\draw[usual, hol=0.33, mob=0.66, dot] (0, 1.4) to (0,0.8);
\draw[usual, dot] (0,0.35) to (0,0.6);
\end{tikzpicture}
\circ
\begin{tikzpicture}[anchorbase]
\draw[usual, hol=0.33, mob=0.66, dot] (0, 1.4) to (0,0.8);
\draw[usual, dot] (0,0.35) to (0,0.6);
\end{tikzpicture}
=
\parab_1 \cdot \begin{tikzpicture}[anchorbase]
\draw[usual, hol=0.33, mob=0.66, dot] (0, 1.4) to (0,0.8);
\draw[usual, dot] (0,0.35) to (0,0.6);
\end{tikzpicture},
\hspace{0.5cm}  \hspace{0.5cm}
\begin{tikzpicture}[anchorbase]
\draw[usual] (0,0) to[out=90,in=90] (0.5,0);
\draw[usual, mob=0.5] (0,0.5) to[out=270,in=270] (0.5,0.5);
\end{tikzpicture}
\circ
\begin{tikzpicture}[anchorbase]
\draw[usual] (0,0) to[out=90,in=90] (0.5,0);
\draw[usual, mob=0.5] (0,0.5) to[out=270,in=270] (0.5,0.5);
\end{tikzpicture}
=
\parab_1 \cdot \begin{tikzpicture}[anchorbase]
\draw[usual] (0,0) to[out=90,in=90] (0.5,0);
\draw[usual, mob=0.5] (0,0.5) to[out=270,in=270] (0.5,0.5);
\end{tikzpicture}
\end{gather*}
in the case $\parab_1\neq 0$ \etc

The only difficult part remaining is to prove that it is not possible to construct idempotents for $\{\text{0 or 1},\dots,n-3, n-1\}\setminus\{0\}$ in the case $ \mxmon{1} \in\{\mrobrmon{1},\mmomon{1}\}$ under the conditions $x_k=0$ for all $k, \mathbf{x}$. We first note that such an (undecorated) idempotent must contain \,\begin{tikzpicture}[anchorbase,scale=0.55]
\draw[white] (0,0) to[out=90,in=270] (0,1);
\draw[usual,dot] (0,0.3) to (0,0.6);
\end{tikzpicture}, \,since the number of through strands have been reduced by an odd number from $n$. Furthermore, it is easy to see that the following cannot be used to construct idempotents:
\begin{gather*}
\begin{tikzpicture}[anchorbase]
\draw[usual, dot] (0,0.6) to (0,0.35);
\draw[usual, dot] (0,0) to (0,0.25);
\end{tikzpicture}
, \hspace{0.5cm}
\begin{tikzpicture}[anchorbase]
\draw[usual] (1,0) to[out=90, in=270] (0,0.5);
\draw[usual, dot] (0,0) to (0,0.25);
\node at (0.5,0) {$\ldots$};
\end{tikzpicture}
,\hspace{0.5cm}
\begin{tikzpicture}[anchorbase]
\draw[usual] (-1,0) to[out=90, in=270] (0,0.5);
\draw[usual, dot] (0,0) to (0,0.25);
\node at (-0.5,0) {$\ldots$};
\end{tikzpicture}
.
\end{gather*}
The idempotent must therefore be of the form
\begin{gather*}
\begin{tikzpicture}[anchorbase]
\draw[usual] (0,1) to[out=270, in=270] (1,1);
\draw[usual, dot] (0,0) to (0,0.25);
\node at (0.5,1) {$\ldots$};
\draw[usual] (1,0) to[out=90, in=90] (2,0);
\node at (1.5,0) {$\ldots$};
\draw[usual] (2,1) to[out=270, in=270] (3,1);
\node at (2.5,1) {$\ldots$};
\node at (3,0) {$\ldots$};
\draw[usual, dashed] (0,0) to (0,1);
\end{tikzpicture},
\end{gather*}
or its mirror image about the dashed line. However, this is impossible since $n$ is finite and the sequence must eventually terminate in either a through strand at the top or bottom of the diagram, or \,\begin{tikzpicture}[anchorbase,scale=0.55]
\draw[white] (0,0) to[out=90,in=270] (0,1);
\draw[usual,dot] (0,0.3) to (0,0.6);
\end{tikzpicture}, \,\begin{tikzpicture}[anchorbase,scale=0.55]
\draw[white] (0,0) to[out=90,in=270] (0,1);
\draw[usual,dot] (0,0.6) to (0,0.3);
\end{tikzpicture}, which cannot result in an idempotent. 
\end{proof}

\textit{(d)}. We use the result that given any monoid $\monoid$ (with an antiinvolution) whose $J$-cells are either idempotent or have $H$-cells of size one, we can construct a(n involutive) sandwich pair $(\K\monoid,\monoid)$ \cite[Proposition 2E.1]{Tu-sandwich-cellular}. Taking $\monoid = \mxmon{1}$, the construction of the sandwich pair $(\K\mxmon{1}, \mxmon{1})$ can be pulled over to $\big(\mxmon{\paraa, \parab, \parac},\mxmon{1}\big)$.
\end{proof}


\section{Classification}\label{R:Representations}


We now classify the $\K$-linear representations of $\mxmon{\paraa, \parab, \parac}$, for which we have  \autoref{T:DAlgebrasDiaSimples} below, together with \autoref{T:SandwichCMP}.

Consider the sandwich pair $\big(\mxmon{\paraa, \parab, \parac} ,\mxmon{1}\big)$. The set of apexes $\apex$ for simple 
$\mxmon{\paraa, \parab, \parac}$-modules is given in \autoref{Eq:DAlgebrasTable}.
Set $s= 1+3r$ when $\K=\bar{\K}$ is of characteristic zero, or more generally
\begin{gather}\label{Eq:GenCount}
\begin{gathered}
s=1+N_{\K}(r)+N_{\K}(2r),
\\
\text{where }
N_{\K}(k)=\#\{\text{monic irreducible factors of }x^{m(k)}-1\text{ in }\K[x]\}
\end{gathered}
\end{gather}
and $m(k)$ is defined explicitly by $m(k)=k \text{ if }\mathrm{char}(\K)=0$, or implicitly by $k=p^{e}m(k)$ for some nonnegative integer $e$ and $p=\mathrm{char}(\K)>0,\ p\nmid m(k)$.

Let $p(\lambda)$ denote the number of (integer) partitions of $\lambda$. 

\begin{Theorem}\label{T:DAlgebrasDiaSimples}
Given $\lambda\in\apex$:
\begin{enumerate}[label=(\roman*)]
\item In the non-planar case with $\K=\bar{\K}$ of characteristic zero, there are precisely 
\begin{align*}
\sum_{(\lambda)} p(\lambda_1)\cdots p(\lambda_s)
\end{align*}
simple  
$\mxmon{\paraa, \parab, \parac}$-modules of apex $\lambda$, where the sum is taken over all $s$-tuples $(\lambda_1,  \ldots, \lambda_s)$ of nonnegative integers such that $\sum \lambda_i = \lambda$; 
for general $\K$ this is an upper bound;

\item In the planar case, there are precisely $s^\lambda$ simple  
$\mxmon{\paraa, \parab, \parac}$-modules of apex $\lambda$.
\end{enumerate}
\end{Theorem}

\begin{Remark}
We impose a restriction on $\K$ in the non-planar cases of \autoref{T:DAlgebrasDiaSimples} arising from our use of generalized conjugacy classes to enumerate the simple representations of monoids. (No such restriction is required in the planar case, where the argument works over an arbitrary field.) The results of \cite{MaQuSt-characters-monoids} should allow this restriction to be removed.
\end{Remark}

The rest of this section is devoted to proving \autoref{T:DAlgebrasDiaSimples}, the main ingredients being $H$-reduction \autoref{T:SandwichCMP}, and the classification of $\mathcal{M}$-simples and $\mathcal{M} \wr \sym[\lambda]$-simples more generally. 

We begin by adapting the general concepts outlined in the previous section to the case of the monoid $\mathcal{M}$, in order to classify its simple representations. In particular, we want to describe the $J$-cells of $\mathcal{M}$, for which there is the following proposition:

\begin{Proposition}\label{P:CellStructure}
In the diagram \autoref{Eq:Cells} below:
\begin{itemize}
\item[(a)] The $J$-cells (equals left, right and $H$-cells) of $\mathcal{M}$ are indicated by boxes (dotted or solid);
\item[(b)] The solid boxes indicate that a $J$-cell contains an idempotent, which is also an identity element with respect to the following $J$-cell group structures:
\begin{itemize}
\item[$\jcell_{1}$:] The trivial group;
\item[$\jcell_{r}$:] $\mathbb{Z}/r\mathbb{Z}$, the cyclic group of order $r$;
\item[$\jcell_{2r}$:] $\mathbb{Z}/2r\mathbb{Z}$, the cyclic group of order $2r$.
\end{itemize}
\item[(c)] The dashed arrows give the preordering on the $J$-cells.
\end{itemize}

\begin{gather}\label{Eq:Cells}
\begin{tikzpicture}
[box/.style = {draw,inner sep=5pt},anchorbase]
\node[rectangle, draw] (1) at (0,0) {$1$};
\node[rectangle, draw, dotted, thick] (2) at (2.5,0)  {$b$};
\node[rectangle, draw, dotted, thick] (3) at (5,0){$b^2$};
\node[rectangle, draw, dotted, thick] (4) at (0,1) {$a$};
\node[rectangle, draw, dotted, thick] (5) at (2.5,1)  {$ab$};
\node[rectangle, draw, dotted, thick] (6) at (5,1){$ab^2$};
\node (7) at (0,2){$\vdots$};
\node (8) at (2.5,2){$\vdots$};
\node (9) at (5,2){$\vdots$};
\node[rectangle, draw, dotted, thick] (10) at (0,3) {$a^{K-r-1}$};
\node[rectangle, draw, dotted, thick] (11) at (2.5,3)  {$a^{K-r-1}b$};
\node[rectangle, draw, dotted, thick] (12) at (5,3){$a^{K-r-1}b^2$};
\node (13) at (0,4.5){$a^{K-r}$};
\node (14) at (2.5,4.5){$a^{K-r}b$};
\node (15) at (5,4.5){$a^{K-r}b^2$};
\node (16) at (0,5){$\vdots$};
\node (17) at (2.5,5){$\vdots$};
\node (18) at (5,5){$\vdots$};
\node (19) at (0,5.5){$a^{K-1}$};
\node (20) at (2.5,5.5){$a^{K-1}b$};
\node (21) at (5,5.5){$a^{K-1}b^2$};
\node[box, name= box2, rotate fit=0,fit=(14)(15)(17)(18)(20)(21)] {};
\node[box,name = box, rotate fit=0,fit=(13)(16)(19)] {};
\node (22) at (-1,0){$\jcell_{1}$};
\node (23) at (-1.5,5){$\jcell_{r}$};
\node (23) at (6.5,5){$\jcell_{2r}$};
\draw [->,dashed] (1.east) -- (2.west);
\draw [->,dashed] (2.east) -- (3.west);
\draw [->,dashed] (4.east) -- (5.west);
\draw [->,dashed] (5.east) -- (6.west);
\draw [->,dashed] (10.east) -- (11.west);
\draw [->,dashed] (11.east) -- (12.west);
\draw [->,dashed] (box.east) -- (box2.west);
\draw [->,dashed] (1.north) -- (4.south);
\draw [->,dashed] (2.north) -- (5.south);
\draw [->,dashed] (3.north) -- (6.south);
\draw [->,dashed] (4.north) -- (7.south);
\draw [->,dashed] (5.north) -- (8.south);
\draw [->,dashed] (6.north) -- (9.south);
\draw [->,dashed] (7.north) -- (10.south);
\draw [->,dashed] (8.north) -- (11.south);
\draw [->,dashed] (9.north) -- (12.south);
\draw [->,dashed] (10.north) -- (box.south);
\draw [->,dashed] (11.north) -- (box2);
\draw [->,dashed] (12.north) -- (box2);
\end{tikzpicture}.
\end{gather}
\end{Proposition}

\begin{proof}
\textit{(a)}. We begin by observing that every element of the commutative monoid $\mathcal{M}$ can be written in the form $a^ib^j$ with $0 \leq i<K$, $0 \leq j\leq 2$, so that every element is included in the diagram. Moreover, $J$-cells are left and right cells.

Next, for any given element $a^kb^\ell$ with $0\leq k < K-r$, $0 \leq \ell \leq 2$, there is no element in $\mathcal{M}$ that can be multiplied to reduce the power of $a$: repeated multiplication by $a$ either increases the power of $a$ up to $K-1$, and beyond which the power is always at least $K-r$ since $a^K=a^{K-r}$; repeated multiplication by $b$, on the other hand, by the relation $ab=b^3$, either increases the power of $b$, or $a$ and we are back in the previous case. Therefore, the rows of the diagram below the $(K-r)th$ make up separate $J$-cells. Since it is also not possible to reduce the power of $b$ while maintaining the power of $a$, the corresponding columns also lie in separate $J$-cells. 

We now turn to the elements lying in or above the $(K-r)th$ row. Firstly, nonzero powers of $b$ can never be reduced to a zero power, and therefore, the last two columns must comprise separate $J$-cells to the first. Given any two elements $a^k$, $a^{k'}$, $K-r \leq k < k' < K$, we clearly have $a^{k'-k} \cdot a^{k}=a^{k'}$; on the other hand, 
$a^{(K-k')+[k-(K-r)]} \cdot a^{k'}=a^{K+[k-(K-r)]}=a^{(K-r) +[k-(K-r)]}=a^k$, and so the first column forms a single $J$-cell. 

The last two columns, call them $c$ and $c'$ respectively, must also each lie in a single $J$-cell by the previous argument. Clearly, multiplying an element $c$ gives an element in $c'$. Multiplying an element in $c'$ by $b$ results in an element with the power of $b$ reduced by one and $a$ at least $K-r$ \ie\, an element in $c$. Therefore, $c$ and $c'$ form a single $J$-cell.

\textit{(b)}. The case for $\jcell_{1}$ is clear. Set $\rho \equiv -K\mod r$, $0 \leq \rho < r$; in other words, $\rho +K= qr$ for some integer $q$ where $qr \geq K>0$ and hence $q >0$ since $r>0$. We also have that $a^{K-r+\delta r}= a^{K-r}$ for any integer $\delta \geq 0$ by induction. Notice that $a^{K-r + \rho}\cdot a^{K-r + \rho}= a^{(K-r+\rho)+(-r+qr)}= a^{(K+\rho)+(-r+qr)}= a^{K-r + \rho}$ is an idempotent in $\jcell_{r}$. Now set $m_r = a^{K-r + \rho+1}$. It is relatively straightforward to see that $m_r$ generates all $r$ elements in $\jcell_{r}$ since $a^{K-r + \delta} \cdot a^{K-r + \rho+1}= a^{(K-r+\delta)+(-r+qr+1)}=a^{(K+ \delta)+(-r+qr+1)}= a^{K-r+\delta+1}$ for any integer $\delta \geq 0$.

Considering $\jcell_{2r}$, first we set $\rho' \equiv -K -1\mod r$, $0 \leq \rho' < r$ so that $\rho' +K+1=q'r$ for some $q'>0$ as above. It then follows that $a^{K-r + \rho'}b^2\cdot a^{K-r + \rho'}b^2= a^{(K-r+\rho')+(-r+q'r-1)}\cdot ab^2= a^{K-r + \rho'}b^2$ is an idempotent in $\jcell_{2r}$. Set $m_{2r}=a^{K-r+\rho}b$ with $\rho$ as above. Since $a^{K-r+\rho}$ acts as the identity in $\set{a^{K-r}, \dots, a^{K-1}}$, repeated multiplications of $m_{2r}$ by itself results in the following $2r$ cycle:
\begin{gather*}
\begin{tikzpicture}
[box/.style = {draw,inner sep=5pt},anchorbase]
\node (14) at (0,0){$a^{K-r}b$};
\node (15) at (3,0){$a^{K-r}b^2$};
\node (17) at (0,4){$\vdots$};
\node (18) at (3,4){$\vdots$};
\node (19) at (0,1){$\vdots$};
\node (20) at (3,1){$\vdots$};
\node (20a) at (0,5){$a^{K-1}b$};
\node (21) at (3,5){$a^{K-1}b^2$};
\node (22) at (0,2){$m_{2r} \coloneqq a^{K-r+\rho}b$};
\node (23) at (3,2){$a^{K-r+\rho}b^2$};
\node (24) at (0,3){$a^{K-r+\rho+1}b$};
\node (25) at (3,3){$a^{K-r+\rho+1}b^2$};
\node (26) at (4.5,5){$\longrightarrow *$};
\node (27) at (-0.8,0){$*$};
\node (28) at (-2.5,2){\textit{Start} $\longrightarrow$};
\draw [->] (14.east) -- (15.west);
\draw [->] (15.north) -- (19.south);
\draw [->] (20.north) -- (22.south);
\draw [->] (22.east) -- (23.west);
\draw [->] (23.north) -- (24.south);
\draw [->] (24.east) -- (25.west);
\draw [->] (25.north) -- (17.south);
\draw [->] (18.north) -- (20a.south);
\draw [->] (20a.east) -- (21.west);
\end{tikzpicture}.
\end{gather*}
Finally, using similar arguments to \textit{(a)}, it is straightforward to see there are no other idempotent $J$-cells.

\textit{(c)}. Multiplication by $b$ results in a horizontal progression through $J$-cells, whereas multiplication by $a$ results in a vertical progression. Putting the two together then gives the result.
\end{proof}

Using a version of $H$-reduction for monoids, also referred to as the \emph{Clifford-Munn-Ponizovskiǐ theorem} (for details, we refer the reader to \cite[Theorem 5.5]{St-rep-monoid} in particular, or \cite{GaMaSt-irreps-semigroups}), we conclude that:
\begin{gather*}
\begin{gathered}
\left\{
\text{simple $\K \mathcal{M}$-modules}
\right\}/\cong
\\
\xleftrightarrow{1:1}
\\
\left\{
\text{simple $\K \jcell_{1}$-, $\K \jcell_{r}$-, or $\K \jcell_{2r}$-modules}
\right\}/\cong
.
\end{gathered}
\end{gather*}
In other words, there are $s$ simple $\K\mathcal{M}$-modules, with $s$ as in \autoref{Eq:GenCount}, by the following (well-known) result \autoref{L:CyclicSimplesCount}:

\begin{Lemma}\label{L:CyclicSimplesCount}
The number of isomorphism classes of simple $\K[\mathbb Z/k\mathbb Z]$-modules equals
$N_{\K}(k)$.
\end{Lemma}

\begin{proof}
Since $\mathbb Z/k\mathbb Z\cong \mathbb Z/p^{a}\mathbb Z\times \mathbb Z/m(k)\mathbb Z$, we have
$\K[\mathbb Z/k\mathbb Z]\cong \K[\mathbb Z/p^{a}\mathbb Z]\otimes_{\K}\K[\mathbb Z/m(k)\mathbb Z]$.
Since the algebra $\K[\mathbb Z/p^{a}\mathbb Z]$ is local, and therefore has a unique simple module (the trivial one),
the simples of $\K[\mathbb Z/k\mathbb Z]$ are exactly the simples of $\K[\mathbb Z/m(k)\mathbb Z]$.
Finally $\K[\mathbb Z/m(k)\mathbb Z]\cong \K[x]/(x^{m(k)}-1)$, and its simple modules correspond to the irreducible
factors of $x^{m(k)}-1$.
\end{proof}

We now turn our attention to the classification of the monoid wreath product $\mathcal{M} \wr \sym[\lambda]$. For this, we mainly follow \cite[Chapter 4]{JaKe-reptheory-symmetric-group}, describing representations of $G \wr \sym[\lambda]$ given a group $G$, adapted to the monoid case using \cite{St-rep-monoid}. 

Fix a finite monoid $M$. We begin with the basic definition:

\begin{Definition}
The \emph{(monoid) wreath product} is given by the set 

\begin{align*}
M \wr \sym[\lambda] \coloneqq \set{(f; \pi) | f: \set{1, \ldots, n} \rightarrow M, \pi \in \sym[\lambda]},
\end{align*}
together with the binary operation
\begin{align*}
(f; \pi)(f'; \pi') = (ff'_\pi; \pi \pi')
\end{align*}
where 
\begin{align*}
f_\pi \coloneqq f \circ \pi^{-1} \in M^n\end{align*}
and 
\begin{align*}
(f f')(i) = f(i) \cdot f'(i), \, i \in \set{1, \ldots, n}\end{align*}
defines multiplication on $M^n$.
\end{Definition}

It is straightforward to show that $M\wr \sym[\lambda]$ is indeed a monoid with identity element $(e; 1_{\sym[\lambda]})$ where $e(i) \coloneqq 1_M$ for all $i \in \set{1, \ldots, n}$, and also that 
\begin{align}
(f_\pi)_{\pi'}=f_{\pi' \pi}.
\end{align}

Given the notion of a \emph{generalized conjugacy class} of $M$, which we describe below, there is the following crucial fact:

\begin{Lemma}\label{gencontheorem}
Assume that $\K=\bar{\K}$ is of characteristic zero.
The number of isomorphism classes of simple $\K M$-modules coincides with the number of generalized conjugacy classes of $M$.
\end{Lemma}

\begin{proof}
Well-known, cf. \cite[Theorem 7.10]{St-rep-monoid}.
\end{proof}

In order to describe the generalized conjugacy classes, we first collect some properties of finite monoids (or semigroups more generally). To begin, fix a finite semigroup $S$ and let $s \in S$. It is straightforward to see that there exists a smallest positive integer $c$, the \emph{index} of $s$, such that $s^c = s^{c +d}$ for some $d >0$; the smallest such $d$ is called the \emph{period} of $s$.

\begin{Lemma} \label{semigroupstuff}
Given $s \in S$ with index $c$ and period $d$, we have the following:
\begin{enumerate}[label=(\roman*)]
\item $s^i = s^j$ if and only if $i = j$ or $i, j \geq c$ and $i \equiv j \mod d$.
\item The subsemigroup $C = \set{s^n | n \geq c}$ of $\langle c \rangle$ is a cyclic group of order $d$. The identity of $C$ is the unique idempotent in $\langle c \rangle$ and is given by $s^m$ where $m \geq c$ and $m \equiv 0 \mod d$; denote this element generically by $s^\omega$. If $s^{\omega +1} \coloneqq s^\omega s$, then $C = \langle s^{\omega +1} \rangle$.
\item There exists a positive integer $a$ such that $s^a=s^\omega$ for all $s \in S$.
\end{enumerate}
\end{Lemma}

\begin{proof}
The proofs for \textit{(i)} and \textit{(ii)} are given by Proposition 1.1 and Corollary 1.2 of \cite{St-rep-monoid} respectively. Finally, for \textit{(iii)}, if $S$ is of order $n$, then $n!$ is such an integer (see \cite[Remark 1.3]{St-rep-monoid}).
\end{proof}

Now define a relation $\sim $ on $M$ by $m \sim n$ if and only if there exist $x, x'$ such that $x x'x=x$, $x'xx'=x'$, $x'x=m^\omega$, $xx' = n^\omega$, and $xm^{\omega +1}x'=n^{\omega +1}$. The relation $\sim$ is an equivalence on $M$ \cite[Proposition 7.2]{St-rep-monoid} and the $\sim$-classes are the generalized conjugacy classes.

\begin{Example}
In case $M$ is a group, then $m^{\omega}$ is the identity for all $m \in M$, and the relation $\sim$ is that of standard conjugacy. It is useful therefore to think of $x'$ as being a generalized inverse of $x$, and of $m^{\omega +1}$ and $n^{\omega +1}$ as being conjugate.
\end{Example}

\begin{Remark}
From now on, for concreteness, we can consider $m^\omega$ as been given by $m^{a}$ for all $m \in M$ using \autoref{semigroupstuff} (iii).
\end{Remark}

The goal is to now determine the number of generalized conjugacy classes of $M \wr \sym[\lambda]$. First suppose that the generalized conjugacy classes of $M$ are given by $C^1, \ldots, C^\mathscr{C}$ and consider an element $(f; \pi) \in M \wr \sym[\lambda]$. We associate with the $\nu th$ cyclic factor $(j_\nu \ldots \pi^{\ell_\nu -1}(j_\nu)) $ of $\pi$ an element 
\begin{align*}
m_\nu (f; \pi) &\coloneqq f(j_\nu)f(\pi^{-1}(j_\nu))\cdots f(\pi^{-\ell_\nu +1}(j_\nu)) \\
& = ff_\pi\cdots f_{\pi^{\ell_\nu-1}}(j_\nu)
\end{align*}
of $M$ called the $\nu th$ \emph{cycle product} of $(f; \pi)$, cf. \cite[4.2.1]{JaKe-reptheory-symmetric-group}; observe that this element is well-defined up to generalized conjugation:

\begin{Lemma}\label{congdef}
For all $\nu, t$, $f \cdots f_{\pi^{\ell_\nu -1}}(j_\nu) \sim f \cdots f_{\pi^{\ell_\nu -1}}(\pi^t(j_\nu))$. 
\end{Lemma}

\begin{proof}
The proof is identical to \cite[4.2.5]{JaKe-reptheory-symmetric-group} since we also have that $m'm \sim mm'$ for all $m', m\in M$. The later result can be seen by setting $x=m'(mm')^{2a-1}=(m'm)^{2a-1}m'$ and $x' =(mm')^am=m(m'm)^a$, see \cite[Proposition 7.3]{St-rep-monoid}.
\end{proof}

\autoref{congdef} motivates the following characterization of cycle products:

\begin{Definition}
Form the following $r \times \lambda$ matrix
\begin{align*}
a(f; \pi) \coloneqq (a_{ik}(f; \pi))
\end{align*} 
where $a_{ik}(f; \pi)$ denotes the number of cycle products associated with a $k$-cycle and which belong to the generalized conjugacy class $C^i$ of $M$. We call this matrix the \emph{type} of $(f ; \pi)$.
\end{Definition}

We would now like to show that two elements of $M \wr \sym[\lambda]$ are of the same type if and only if they are conjugate, cf. \cite[Theorem 4.2.8]{JaKe-reptheory-symmetric-group}. We first state some preliminary results:

\begin{Lemma}\label{prelims}
Given an element $(f ; \pi) \in M \wr \sym[\lambda]$, we have the following:
\begin{enumerate}[label=(\roman*)]
\item $a(f; \pi)= a(f_{\pi'};\pi'\pi\pi'^{-1})$;
\item If $a(f;\pi)=a(f';\pi')$, then there is a $\pi'' \in S_\lambda$ satisfying $\pi = \pi''\pi'\pi''^{-1}$ and with the property that for each $\nu$, $m_\nu(f;\pi) \sim m_\nu(f'_{\pi''}; \pi)$.
\end{enumerate}
\end{Lemma}

\begin{proof}
The proofs for \textit{(i)} and \textit{(ii)} are identical to \cite[4.2.6]{JaKe-reptheory-symmetric-group} and \cite[Lemma 4.2.7]{JaKe-reptheory-symmetric-group} respectively.
\end{proof}

\begin{Lemma}\label{theoremclassmain}
Two elements $(f;\pi)$ and $(f';\pi')$ of $M \wr \sym[\lambda]$ have the property that $(f;\pi) \sim  (f';\pi')$ if and only if $a(f;\pi)= a(f';\pi')$.
\end{Lemma}

\begin{proof}
Suppose that $(f;\pi) \sim  (f';\pi')$. Then there exist $(x; \sigma)$, $(x'; \sigma') \in M \wr \sym[\lambda]$ such that
\begin{enumerate}
\item[$(a)$] $(x; \sigma)(x';\sigma')(x; \sigma)=(x;\sigma)$;
\item[$(b)$] $(x'; \sigma')(x; \sigma)(x';\sigma')=(x'; \sigma')$;
\item[$(c)$] $(x';\sigma')(x; \sigma)=(f;\pi)^\omega$;
\item[$(d)$] $(x;\sigma)(x'; \sigma')=(f';\pi')^\omega$;
\item[$(e)$] $(x;\sigma)(f;\pi)^{\omega +1}(x';\sigma')=(f';\pi')^{\omega+1}$.
\end{enumerate}
From $(a)$ it follows immediately that $\sigma'=\sigma^{-1}$, and therefore that $|\pi| \, | \, a$, $|\pi'| \, | \, a$, and $\pi' =\sigma \pi \sigma^{-1}$ from $(c)$, $(d)$, and $(e)$ respectively. 

Now consider only the $M$ component of $(a)$--$(e)$ above:
\begin{enumerate}
\item[$(i)$] $xx'_\sigma x=x$;
\item[$(ii)$] $x' x_{\sigma^{-1}}x'=x'\Rightarrow x'_\sigma x x'_\sigma = x'_\sigma$;
\item[$(iii)$] $x' x_{\sigma^{-1}}=ff_\pi\cdots f_{\pi^{a-1}}\Rightarrow x'_\sigma x=f_\sigma f_{\sigma \pi}\cdots f_{\sigma \pi^{a-1}}=f_\sigma f_{\pi' \sigma}\cdots f_{\pi'^{a-1}\sigma}$;
\item[$(iv)$] $x x'_\sigma= f'f'_{\pi'}\cdots f'_{\pi'^{a-1}}$;
\item[$(v)$] $xf_\sigma f_{\sigma \pi}\cdots f_{\sigma \pi^{a-1}}f_\sigma x'_{\sigma \pi}=f'f'_{\pi'}\cdots f'_{\pi^{a-1}}f' \Rightarrow xf_\sigma f_{\pi' \sigma }\cdots f_{\pi'^{a-1}\sigma}f_\sigma x'_{\pi' \sigma }=f'f'_{\pi'}\cdots f'_{\pi^{a-1}}f'$.
\end{enumerate}
Suppose that the $\nu th$ cyclic factor of $\pi'$ is given by $(j_\nu \ldots \pi'^{\ell_\nu -1}(j_\nu)) $. Taking $(v)$, precomposing each side with increasing powers of $\pi'$ $\ell_\nu-1$ times and multiplying all the results together in order starting with $(v)$ gives 
\begin{align*}
xf_\sigma f_{\pi' \sigma }\cdots f_{\pi'^{a-1}\sigma}f_\sigma x'_{\pi' \sigma } x_{\pi'}f_{\pi' \sigma} f_{\pi'^2 \sigma }\cdots f_{\pi'^{\ell_\nu+a-2}\sigma}f_{\pi'^{\ell_\nu-1}\sigma} x'_{ \pi'^{\ell_\nu }\sigma } \\=f'f'_{\pi'} \cdots f'_{\pi'^{a-1}}f'f'_{\pi'} f'_{\pi'^2}\cdots  f'_{\pi'^{\ell_\nu+a-2}}f'_{\pi'^{\ell_\nu-1}},
\end{align*}
which upon substituting $(iii)$ with the appropriate precomposition with powers of $\pi'$ gives
\begin{align*}
x(f_\sigma f_{\pi' \sigma }\cdots f_{\pi'^{a-1}\sigma})^{\ell_\nu}f_\sigma f_{\pi' \sigma }\cdots f_{\pi'^{\ell_\nu-1}\sigma}x'_{ \pi'^{\ell_\nu }\sigma } =(f'f'_{\pi'} \cdots f'_{\pi'^{a-1}})^{\ell_\nu}f'f'_{\pi'} \cdots f'_{\pi'^{\ell_\nu-1}},
\end{align*}
or in other words,
\begin{align*}
x(j_\nu)(f_\sigma f_{\pi' \sigma }\cdots f_{\pi'^{a-1}\sigma})^{\ell_\nu}(j_\nu)m_\nu(f_\sigma;\pi')x'_{ \sigma }(j_\nu) =(f'f'_{\pi'} \cdots f'_{\pi'^{a-1}})^{\ell_\nu}(j_\nu)m_\nu(f';\pi').
\end{align*}

Furthermore, since $|\pi'| \, | \, a$, the quantities $f_\sigma f_{\pi' \sigma}\cdots f_{\pi'^{a-1}\sigma}(j_\nu)$ and $f'f'_{\pi'}\cdots f'_{\pi'^{a-1}}(j_\nu)$ are powers of $m_\nu(f_\sigma;\pi')$ and $m_\nu(f';\pi')$ respectively satisfying the conditions of \autoref{semigroupstuff} $(ii)$. Hence we have found generalized inverses, namely $x(j_\nu)$ and $x'_\sigma (j_\nu)$, such that $m_\nu(f_\sigma;\pi') \sim m_\nu(f';\pi')$.

Finally, using \autoref{prelims} $(i)$, it follows that $a(f';\pi')=a(f_\sigma; \pi')=a(f;\pi)$.

Conversely, suppose that $a(f; \pi)=a(f';\pi')$. From \autoref{prelims} $(ii)$, we know that there is a $\pi'' \in S_\lambda$ satisfying $\pi = \pi''\pi'\pi''^{-1}$ and with the property that for each $\nu$, $m_\nu(f;\pi) \sim m_\nu(f'_{\pi''}; \pi)$. In particular, $(f'; \pi')\sim (f'_{\pi''};\pi)$, since we can find generalized inverses $x=(e;\pi'')$ and $x'=(e; \pi''^{-1})$. Therefore, it suffices to prove $(f;\pi)\sim (f'_{\pi''};\pi) $.

To begin, since that for each $\nu$, $m_\nu(f;\pi) \sim m_\nu(f'_{\pi''}; \pi)$, we know that there exists $x_\nu, x'_\nu \in M$ such that 
\begin{enumerate}
\item[$(a)$] $x_\nu x'_\nu x_\nu = x_\nu$;
\item[$(b)$] $x'_\nu x_\nu x'_\nu = x'_\nu$;
\item[$(c)$] $x'_\nu x_\nu=(f \cdots f_{\pi^{\ell_\nu-1}}(j_\nu))^\omega$;
\item[$(d)$] $x_\nu x'_\nu = (f'_{\pi''}\cdots f'_{ \pi^{\ell_\nu-1} \pi'' }(j_\nu))^\omega$;
\item[$(e)$] $ x_\nu (f \cdots f_{\pi^{\ell_\nu-1}}(j_\nu))^{\omega+1} x'_\nu=(f'_{\pi''}\cdots f'_{ \pi^{\ell_\nu-1}\pi''}(j_\nu))^{\omega+1}$.
\end{enumerate}

Given a nonnegative integer $m$, $0 \leq m \leq \ell_\nu-1$, define the function $\tilde{f}^\nu:\set{j_\nu, \ldots, \pi^{\ell_\nu -1}(j_\nu)}\rightarrow M$ by $\tilde{f}_{\pi^m}^\nu(j_\nu)= \tilde{f}^\nu_m $ where
\begin{align*}
\tilde{f}_{m}^\nu \coloneqq x'_{\pi^{-m}(j_\nu)}(f'_{\pi^{m}\pi''}f'_{\pi^{m+1}\pi''} \cdots f'_{\pi^{m-1}\pi''})^\omega f'_{\pi^{m}\pi''} x_{\pi^{-m-1}(j_\nu)}
\end{align*}
where
\begin{align*}
x'_{\pi^{-m}(j_\nu)} &= x'_{ff_{\pi} \cdots f_{\pi^{m-1}}} x'_\nu x_{f'_{\pi''} f_{\pi \pi''} \cdots f'_{\pi^{m-1} \pi''}} \\
x_{\pi^{-m}(j_\nu)} &= x'_{f'_{\pi''} f_{\pi \pi''} \cdots f'_{\pi^{m-1} \pi''}} x_\nu x_{ff_{\pi} \cdots f_{\pi^{m-1}}},
\end{align*}
and
\begin{align*}
x_{ff_{\pi} \cdots f_{\pi^{m}}} &= (ff_{\pi} \cdots {f_{\pi^{\ell_\nu-1}})^{\omega} ff_{\pi} \cdots f_{\pi^m}} \\
x'_{ff_{\pi} \cdots f_{\pi^{m}}} &= f_{\pi^{m+1}}  \cdots f_{\pi^{\ell_\nu -1}}(ff_{\pi} \cdots f_{\pi^{\ell_\nu-1}})^{2a-1} \\
x_{f'_{\pi''} f'_{\pi \pi''} \cdots f'_{\pi^{m} \pi''}} &= (f'_{\pi''} f'_{\pi \pi''} \cdots f'_{\pi^{\ell_\nu-1}\pi''})^\omega f'_{\pi''} f'_{\pi \pi''} \cdots f'_{\pi^{m}\pi''} \\
x'_{f'_{\pi''}f'_{\pi \pi''} \cdots f'_{\pi^{m}\pi''}} &= f'_{\pi^{m+1}\pi''}  \cdots f'_{\pi^{\ell_\nu -1} \pi''}(f'_{\pi''}f'_{\pi \pi''} \cdots f'_{\pi^{\ell_\nu-1}\pi''})^{2a-1}.
\end{align*}

\begin{Remark}
All functions appearing in the expression for $\tilde{f}_{m}^\nu$ above are evaluated at $j_\nu$.
\end{Remark}

\begin{Example}
When $m=0$, we have
\begin{align*}
\tilde{f}^\nu(j_\nu)= \tilde{f}^\nu_0 &= x'_\nu (f'_{\pi''}f'_{\pi \pi''}\ldots f'_{\pi^{\ell_\nu-1}\pi''})^\omega f'_{\pi''} x'_{f'_{\pi''}}x_\nu x_f \\ 
&=x'_\nu (f'_{\pi''}f'_{\pi \pi''}\ldots f'_{\pi^{\ell_\nu-1}\pi''})^{\omega+1} x_\nu (ff_{\pi} \cdots f_{\pi^{\ell_\nu-1}})^{2a-1} f \\
&= (ff_{\pi} \cdots f_{\pi^{\ell_\nu-1}})^{\omega+1}(ff_{\pi} \cdots f_{\pi^{\ell_\nu-1}})^{2a-1} f \hspace{0.5cm} \text{by $(e)$}
\\
&=(ff_{\pi} \cdots f_{\pi^{\ell_\nu-1}})^{\omega}f.
\end{align*}
\end{Example}

In fact, it is not too difficult to see that in general
\begin{align*}
\tilde{f}^\nu_m =(f_{\pi^m}f_{\pi^{m+1}} \cdots f_{\pi^{m-1}})^{\omega}f_{\pi^m},
\end{align*}
and hence
\begin{align*}
\tilde{f}_m^\nu\tilde{f}^\nu_{m+1} \cdots \tilde{f}^\nu_m  &=(f_{\pi^m}f_{\pi^{m+1}} \cdots f_{\pi^{m-1}})^{\omega}f_{\pi^m}(f_{\pi^{m+1}}f_{\pi^{m+2}} \cdots f_{\pi^{m}})^{\omega}f_{\pi^{m+1}} \cdots (f_{\pi^m}f_{\pi^{m+1}} \cdots f_{\pi^{m-1}})^{\omega}f_{\pi^m} \\
&=(f_{\pi^m}f_{\pi^{m+1}} \cdots f_{\pi^{m-1}})^{\omega}f_{\pi^m}.
\end{align*}

On the other hand,
\begin{align*}
x_{\pi^{-m}(j_\nu)} x'_{\pi^{-m}(j_\nu)} &= x'_{f'_{\pi''} f_{\pi \pi''} \cdots f'_{\pi^{m-1} \pi''}} x_\nu x_{ff_{\pi} \cdots f_{\pi^{m-1}}}x'_{ff_{\pi} \cdots f_{\pi^{m-1}}} x'_\nu x_{f'_{\pi''} f_{\pi \pi''} \cdots f'_{\pi^{m-1} \pi''}} \\
&= x'_{f'_{\pi''} f_{\pi \pi''} \cdots f'_{\pi^{m-1} \pi''}} x_\nu (ff_{\pi} \cdots {f_{\pi^{\ell_\nu-1}}})^{\omega} x'_\nu x_{f'_{\pi''} f_{\pi \pi''} \cdots f'_{\pi^{m-1} \pi''}} \\
&= x'_{f'_{\pi''} f_{\pi \pi''} \cdots f'_{\pi^{m-1} \pi''}} x_\nu  x'_\nu x_{f'_{\pi''} f_{\pi \pi''} \cdots f'_{\pi^{m-1} \pi''}} \hspace{0.5cm} \text{by $(a)$ and $(c)$} \\
&=x'_{f'_{\pi''} f_{\pi \pi''} \cdots f'_{\pi^{m-1} \pi''}} (f'_{\pi''} f'_{\pi \pi''} \cdots f'_{\pi^{\ell_\nu-1}\pi''})^\omega x_{f'_{\pi''} f_{\pi \pi''} \cdots f'_{\pi^{m-1} \pi''}} \hspace{0.5cm} \text{by $(d)$} \\
&= f'_{\pi^{m}\pi''}  \cdots f'_{\pi^{\ell_\nu -1} \pi''}(f'_{\pi''}f'_{\pi \pi''} \cdots f'_{\pi^{\ell_\nu-1}\pi''})^{\omega} f'_{\pi''} f'_{\pi \pi''} \cdots f'_{\pi^{m-1}\pi''} \\
&= (f'_{\pi^m \pi''}f'_{\pi^{m+1} \pi''} \cdots f'_{\pi^{m-1}\pi''})^{\omega} 
\end{align*}
and similarly, $x'_{\pi^{-m}(j_\nu)} x_{\pi^{-m}(j_\nu)}= (f_{\pi^m}f_{\pi^{m+1} } \cdots f_{\pi^{m-1}})^{\omega} $. Therefore, it also straightforward to see from the definition of $\tilde{f}_m^\nu$ that
\begin{align*}
\tilde{f}_m^\nu\tilde{f}^\nu_{m+1} \cdots \tilde{f}^\nu_m = x'_{\pi^{-m}(j_\nu)}(f'_{\pi^m \pi''}f'_{\pi^{m+1} \pi''} \cdots f'_{\pi^{m-1}\pi''})^{\omega} f'_{\pi^m \pi''} x_{\pi^{-m-1}(j_\nu)}.
\end{align*}

Furthermore, since 
\begin{align*}
(f_{\pi^m}f_{\pi^{m+1} } \cdots f_{\pi^{m-1}\pi''})^{\omega}f_{\pi^{m}}  \cdots f_{\pi^{\ell_\nu -1}}(ff_{\pi} \cdots f_{\pi^{\ell_\nu-1}})^{2a-1} x'_\nu \\
= f_{\pi^{m}}  \cdots f_{\pi^{\ell_\nu -1}}(ff_{\pi} \cdots f_{\pi^{\ell_\nu-1}})^{2a-1} x'_\nu \hspace{0.5cm} \text{by $(b)$ and $(c)$},
\end{align*}
it follows that $x'_{\pi^{-m}(j_\nu)} x_{\pi^{-m}(j_\nu)}x'_{\pi^{-m}(j_\nu)}=x'_{\pi^{-m}(j_\nu)}$. By a similar argument, we also have that $x_{\pi^{-m}(j_\nu)}x'_{\pi^{-m}(j_\nu)}x_{\pi^{-m}(j_\nu)}=x_{\pi^{-m}(j_\nu)}$.

Finally, we define the functions $x(')_{\pi^m}^\nu(j_\nu) \coloneqq  x(')_{\pi^{-m}(j_\nu)}  $ on $\set{j_\nu, \ldots, \pi^{\ell_\nu -1}(j_\nu)}$, and since these functions are defined for all $\nu$, we can patch them together to produce functions $x(')$ defined on all $\set{1, \ldots, \lambda}$. Given the relations proved above, it is clear that the $x(')$ are generalized inverses, and hence $(f;\pi)\sim (f'_{\pi''};\pi) $.
\end{proof}

The following now also counts the number of simple $\mathcal{M} \wr \sym[\lambda]$-modules (for algebraically closed fields of characteristic zero):

\begin{Proposition}\label{typecount}
Suppose that there are $\mathscr{C}$ generalized conjugacy classes of $M$. The number of permissible types of elements of $M \wr \sym[\lambda]$ is given by 
\begin{align*}
\sum_{(\lambda)} p(\lambda_1)\cdots p(\lambda_\mathscr{C})
\end{align*}
where the sum is taken over all $\mathscr{C}$-tuples $(\lambda_1,  \ldots, \lambda_\mathscr{C})$ of nonnegative integers such that $\sum \lambda_i = \lambda$.
\end{Proposition}

\begin{proof}
By definition, we know that the entries of the matrix $a(f; \pi)$ satisfy 
\begin{align*}
\lambda_i\coloneqq \sum_k k \cdot a_{ik} \in \mathbb{N}_0 \quad \text{and} \quad \sum_i
\lambda_i = \lambda \end{align*}
with $i \in \set{1, \ldots, \mathscr{C}}$. Conversely, given a matrix $(b_{ik})$ of the appropriate size with entries satisfying  
\begin{align*}
\lambda'_i\coloneqq \sum_k k \cdot b_{ik} \in \mathbb{N}_0 \quad \text{and} \quad \sum_i
\lambda'_i = \lambda, \end{align*}
$i \in \set{1, \ldots, \mathscr{C}}$, it is easy to find an element $(f';\pi')$ whose type is given by $(b_{ik})$. Counting the number of such matrices is a straightforward exercise, cf. \cite[Lemma 4.2.9]{JaKe-reptheory-symmetric-group}.
\end{proof}

\begin{proof}[Proof of \autoref{T:DAlgebrasDiaSimples}]
Firstly, we know that there are $s$ simple $\K \mathcal{M}$-modules, with $s$ given in \autoref{Eq:GenCount}. Moreover, if $\K=\bar{\K}$ of characteristic zero, there are $s=1+3r$ simple $\K \mathcal{M}$-modules and hence generalized conjugacy classes of $\mathcal{M}$ by \autoref{gencontheorem}. 

In the non-planar case, $\sand=\mathcal{M} \wr \sym[\lambda]$ for all $\lambda \in \apex$. We can use \autoref{theoremclassmain}, \autoref{typecount}, together with $H$-reduction to obtain the result for $(i)$.

Finally, given that in the planar case, $\sand=\mathcal{M}^\lambda$ for all $\lambda \in \apex$, the result for $(ii)$ follows immediately by $H$-reduction.
\end{proof}


\section{Gram matrices}\label{G:Gram}


In this final section, we investigate the dimensions of the simple $\mxmon{\paraa, \parab, \parac}$-modules using the concept of a Gram matrix.

\begin{Notation}
In the context of $H$-reduction \autoref{T:SandwichCMP}, we write $L(\lambda, K)$ for the simple $\algebra$-modules associated to $\lambda \in \apex$ and a simple $\sand$-module $K$.
\end{Notation}

\begin{Definition}\label{D:GramMatrix}
For a sandwich pair $(\algebra,\cellbasis)$, the \emph{Gram matrix} $\gmatrix$ 
(associated to $\lambda\in\apex$) is the 
$(\#\rcell_{\lambda})$-$(\#\lcell_{\lambda})$-matrix 
with values in $\K$ defined by
\begin{gather*}
(\gmatrix)_{i,j}=
\begin{cases}
s(e)&\text{if $\hcell_{ij}$ is a pseudo idempotent with eigenvalue $s(e)$ (*)},
\\
0&\text{else}.
\end{cases}
\end{gather*}
(Recall that (*) means that there exists $e\in\hcell_{ij}$ with $e^2=s(e)\cdot e$. Moreover, if such an element exists, it is unique, and so the condition (*) is unambiguous.)
\end{Definition}

\begin{Example}\label{roexp}
Fix rational functions 
\begin{gather}\label{gencondext}
Z_{\paraa}(T) = \frac{\paraa_0}{1-T}, \hspace{1cm} Z_{\parab}(T) = \frac{\parab_0}{1- T}, \hspace{1cm} Z_{\parac}(T) = \frac{\parac_0}{1-T},
\end{gather}
with polynomial degrees satisfying the conditions of \autoref{D:DAlgebrasTheAlgebras}; in particular, $K=1$ so there are no holes. Given $\mromon[1]{\paraa, \parab, \parac}$ with $\lambda =0$ (i.e. the one strand case with no through strands), we can define the following matrix:
\begin{itemize}
\item The rows (columns) index the right (left) cells;
\item The entries $(i,j)$ are given by the ``middle'' part of the diagram $h_{ij}^2$ where $h_{ij} \in \hcell_{ij}$ (there is only one element in this case).
\end{itemize}
We can then read off the Gram matrix immediately:
\begin{gather*}
\xy
(0,0)*{\begin{gathered}
\begin{tabular}{>{\centering\arraybackslash}p{1cm}||>{\centering\arraybackslash}p{1cm}|>{\centering\arraybackslash}p{1cm}|>{\centering\arraybackslash}p{1cm}}
\arrayrulecolor{tomato}
$\rcell/\lcell$ &  \begin{tikzpicture}[anchorbase]
\draw[usual,dot] (0,0) to (0,0.6);
\end{tikzpicture}   &  \begin{tikzpicture}[anchorbase]
\draw[usual,mob=0.5, dot] (0,0) to (0,0.6);
\end{tikzpicture}  & \begin{tikzpicture}[anchorbase]
\draw[usual,mob=0.33, mob=0.66, dot] (0,0) to (0,0.6);
\end{tikzpicture} 
\\
\hline
\hline
\begin{tikzpicture}[anchorbase]
\draw[usual,dot] (0,0) to (0,-0.6);
\end{tikzpicture}	& \begin{tikzpicture}[anchorbase]
\draw[usual,dot] (0,0) to (0,0.3);
\draw[usual,dot] (0,0) to (0,-0.3);
\end{tikzpicture} & \begin{tikzpicture}[anchorbase]
\draw[usual,mob=0, dot] (0,0) to (0,0.3);
\begin{scope}[on background layer]
\draw[usual,dot] (0,0) to (0,-0.3);
\end{scope}
\end{tikzpicture} & \begin{tikzpicture}[anchorbase]
\draw[usual,mob=0.25, dot] (0,0) to (0,0.3);
\draw[usual,mob=0.25, dot] (0,0) to (0,-0.3);
\end{tikzpicture}
\\
\hline
\begin{tikzpicture}[anchorbase]
\draw[usual,mob=0.5, dot] (0,0) to (0,-0.6);
\end{tikzpicture}	& \begin{tikzpicture}[anchorbase]
\draw[usual,mob=0, dot] (0,0) to (0,0.3);
\begin{scope}[on background layer]
\draw[usual,dot] (0,0) to (0,-0.3);
\end{scope}
\end{tikzpicture} & \begin{tikzpicture}[anchorbase]
\draw[usual,mob=0.25, dot] (0,0) to (0,0.3);
\draw[usual,mob=0.25, dot] (0,0) to (0,-0.3);
\end{tikzpicture} & \begin{tikzpicture}[anchorbase]
\draw[usual,mob=0, dot] (0,0) to (0,0.3);
\begin{scope}[on background layer]
\draw[usual,dot] (0,0) to (0,-0.3);
\end{scope}
\end{tikzpicture}
\\
\hline
\begin{tikzpicture}[anchorbase]
\draw[usual, mob=0.33, mob=0.66, dot] (0,0) to (0,-0.6);
\end{tikzpicture}	& \begin{tikzpicture}[anchorbase]
\draw[usual,mob=0.25, dot] (0,0) to (0,0.3);
\draw[usual,mob=0.25, dot] (0,0) to (0,-0.3);
\end{tikzpicture} & \begin{tikzpicture}[anchorbase]
\draw[usual,mob=0, dot] (0,0) to (0,0.3);
\begin{scope}[on background layer]
\draw[usual,dot] (0,0) to (0,-0.3);
\end{scope}
\end{tikzpicture} & \begin{tikzpicture}[anchorbase]
\draw[usual,mob=0.25, dot] (0,0) to (0,0.3);
\draw[usual,mob=0.25, dot] (0,0) to (0,-0.3);
\end{tikzpicture}
\end{tabular}
\end{gathered}};
\endxy
\quad
\raisebox{-0.3cm}{$\leftrightsquigarrow
\gmatrix[0]=
\begin{pmatrix}
\paraa_0 & \parab_0 & \parac_0
\\
\parab_0 & \parac_0 & \parab_0
\\
\parac_0 & \parab_0 & \parac_0
\end{pmatrix}.$}
\end{gather*}
Performing some row reductions on $\gmatrix[0]$ gives
\begin{align*}
\begin{pmatrix}
\paraa_0 & \parab_0 & \parac_0
\\
\parab_0 & \parac_0 & \parab_0
\\
\parac_0 & \parab_0 & \parac_0
\end{pmatrix} \quad \stackrel{\text{row} \,1-3}{\sim} \quad \begin{pmatrix}
\paraa_0-\parac_0  & 0 & 0
\\
\parab_0 & \parac_0 & \parab_0
\\
\parac_0 & \parab_0 & \parac_0
\end{pmatrix}.
\end{align*}
Hence $\det(\gmatrix[0])=(\paraa_0-\parac_0)(\parac_0^2-\parab_0^2)$ can be easily read off, and so $\gmatrix[0]$ is full rank when $\paraa_0 \neq \parac_0$, $\parac_0 \neq \pm\parab_0$. Of course, the determinant of $\gmatrix[0]$ can be directly computed, however performing row reductions in this manner is foundational for the proof of a more general statement in \autoref{P:DAlgebrasGram} below. 
\end{Example}

In the case of $\mromon{\paraa, \parab, \parac}$, we have the following theorem relating Gram matrices to the dimensions of simple representations:

\begin{Theorem}\label{T:DAlgebrasDSimples}
Consider the involutive sandwich pair $\big(\mromon{\paraa, \parab, \parac},\mromon{1}\big)$, where we assume that $\K=\overline{\K}$ with $\chark\nmid\{2,r\}$ (e.g. $\chark=0$). The dimension of the simple $\mromon{\paraa, \parab, \parac}$-modules for $\lambda\in\apex$ given a simple $\sand[\lambda]$-module $K$ is 
\begin{gather*}
\dim\big(\lmod[{\lambda,K}]\big)
=\mrk(\gmatrix)\cdot\dim(K)=\mrk(\gmatrix).
\end{gather*}
\end{Theorem}

\begin{proof}
The second equality follows from the fact that the simple representations of commutative algebras are of dimension one. By (copying) the proof of \cite[Theorem 4B.25]{Tu-sandwich-cellular} or \cite[Theorem 3.5.3]{Li-cyclic-monoid}, the first equation holds whenever the sandwiched algebra is commutative and all contributing representations come from semisimple cyclic groups. Then we are done by \autoref{P:CellStructure}.
\end{proof}

For the remainder of this section, we compute Gram matrix ranks for a select number of cases, including $\mromon{\paraa, \parab, \parac}$. (All of this, in contrast to \autoref{T:DAlgebrasDSimples}, works over an arbitrary field.) To summarize, we have the following:

\begin{Proposition}\label{P:DAlgebrasGram}
Given the 
involutive sandwich pair $\big(\mxmon{\paraa, \parab, \parac} ,\mxmon{1} \big)$, the ranks of (some of the) Gram matrices are as follows:
\begin{enumerate}[label=(\roman*)]

\item We have $\mrk(\gmatrix[n])=1$ whenever $n\in\apex$;
\end{enumerate}
Furthermore, assuming the conditions \autoref{gencondext} set out in \autoref{roexp}:

\begin{enumerate}[label=(\roman*), resume]
\item For $ \mxmon{1} \in\{\mromon{1},\mpromon{1}\}$, we have $\mrk(\gmatrix)=\binom{n}{\lambda} \cdot 3^{n-\lambda}$ 
for $\lambda\in\apex$ in case that
\begin{align}\label{gramcond}
\Bigg[\prod_{i=1}^{\bar{\lambda}}\Big((\paraa_0^i - \parac_0^i)-\sum_{k=1}^{i-1}\binom{i}{k}(\paraa_0^{i-k}\parac_0^k - \parac_0^i)  \Big)^{\binom{\bar{\lambda}}{i}\cdot 2^{\bar{\lambda}-i}}  \Bigg](\parac_0^2-\parab_0^2)^{\bar{\lambda} \cdot 3^{\bar{\lambda}-1}} \neq 0
\end{align}
where $\bar{\lambda}=n -\lambda$.

\item For $\mxmon{1} \in \{\mrobrmon{1},\mmomon{1}\}$ and $n-1\in\apex$, which also satisfies \autoref{gramcond},
we have $\mrk(\gmatrix[{n-1}])=3n$.

\end{enumerate}
\end{Proposition}

\begin{proof}
\textit{(i).} The condition $\lambda = n$ implies that every strand is a through strand, and hence there is only one right cell that is also a left cell and the identity diagram is clearly the idempotent. 

\textit{(ii).} We first note that for $n=0$, $0\in \apex$, $\mxmon{\paraa, \parab, \parac} \cong \K$, and the statement is clearly true. We therefore concentrate on the case $n >0$.

The strategy is to prove that for every $n >0$, $\gmatrix[0]$ can be row reduced to block lower triangular form 
\begin{align*}
\begin{pmatrix}
A_n  &   0
\\
 *  & B_n
\end{pmatrix}
\end{align*}
where the matrix $B_n$ with $\det(B_n)=(\parac_0^2-\parab_0^2)^{n\cdot 2^{n-1}}$ does not containing $\paraa_0$ and the matrix $A_n$ is block lower triangular with diagonal blocks of the form 
\begin{align*}
\Big((\paraa_0^i - \parac_0^i)-\sum_{k=1}^{i-1}\binom{i}{k}(\paraa_0^{i-k}\parac_0^k - \parac_0^i)  \Big) B_{n-i},
\end{align*}
for $i=1, \ldots, n$ and $B_0 \coloneqq (1)$; the number of such blocks is given by the binomial coefficient $\binom{n}{i}$. 

We then have
\begin{align*}
\det(A_n) = \prod_{i=1}^{n}\Big((\paraa_0^i - \parac_0^i)-\sum_{k=1}^{i-1}\binom{i}{k}(\paraa_0^{i-k}\parac_0^k - \parac_0^i)  \Big)^{\binom{n}{i}\cdot 2^{n-i}} (\parac_0^2-\parab_0^2)^{\binom{n}{i}(n-i) \cdot 2^{n-i-1}},
\end{align*}
and so 
\begin{align*}
\det(\gmatrix[0])=\Bigg[\prod_{i=1}^{n}\Big((\paraa_0^i - \parac_0^i)-\sum_{k=1}^{i-1}\binom{i}{k}(\paraa_0^{i-k}\parac_0^k - \parac_0^i)  \Big)^{\binom{n}{i}\cdot 2^{n-i}}  \Bigg](\parac_0^2-\parab_0^2)^{n \cdot 3^{n-1}},
\end{align*}
since 
\begin{align*}
\sum_{i=1}^n \binom{n}{i} (n-i) \cdot 2^{n-i-1} + n \cdot 2^{n-1}&=\sum_{i=0}^n \binom{n}{i} (n-i) \cdot 2^{n-i-1} \\
& = \sum_{i=0}^n \binom{n}{i} i \cdot 2^{i-1} = n \cdot 3^{n-1},
\end{align*}
where the last equality can be seen by differentiating both sides of $(1+x)^n = \sum_{r=0}^n \binom{n}{r}x^r$ with respect to $x$ before setting $x=2$. 

Once we have the form of $\gmatrix[0]$ for fixed $n>0$, we can then determine $\gmatrix[\lambda]$ for $\lambda \in \apex$: the nonzero entries of $\gmatrix$ occur when through strands align, for example in the case $n=3$, $\lambda=1$ (ignoring Möbius dots) we have
\begin{gather*}
\xy
(0,0)*{\begin{gathered}
\begin{tabular}{C||C|C|C}
\arrayrulecolor{tomato}
\rcell/\lcell & \begin{tikzpicture}[anchorbase]
\draw[usual,dot] (0,0) to (0,0.15);
\draw[usual,dot] (0.5,0) to (0.5,0.15);
\draw[usual] (1,0) to (1,0.25);
\end{tikzpicture} & \begin{tikzpicture}[anchorbase]
\draw[usual,dot] (0,0) to (0,0.15);
\draw[usual,dot] (1,0) to (1,0.15);
\draw[usual] (0.5,0) to (0.5,0.25);
\end{tikzpicture} & \begin{tikzpicture}[anchorbase]
\draw[usual,dot] (1,0) to (1,0.15);
\draw[usual,dot] (0.5,0) to (0.5,0.15);
\draw[usual] (0,0) to (0,0.25);
\end{tikzpicture}
\\
\hline
\hline
\begin{tikzpicture}[anchorbase]
\draw[usual,dot] (0,0) to (0,-0.15);
\draw[usual,dot] (0.5,0) to (0.5,-0.15);
\draw[usual] (1,0) to (1,-0.25);
\end{tikzpicture}	& \begin{tikzpicture}[anchorbase]
\draw[usual,dot] (0,0) to (0,0.15);
\draw[usual,dot] (0.5,0) to (0.5,0.15);
\draw[usual] (1,0) to (1,0.25);
\draw[usual,dot] (0,0) to (0,-0.15);
\draw[usual,dot] (0.5,0) to (0.5,-0.15);
\draw[usual] (1,0) to (1,-0.25);
\end{tikzpicture} & \begin{tikzpicture}[anchorbase]
\draw[usual,dot] (0,0) to (0,0.15);
\draw[usual,dot] (1,0) to (1,0.15);
\draw[usual] (0.5,0) to (0.5,0.25);
\draw[usual,dot] (0,0) to (0,-0.15);
\draw[usual,dot] (0.5,0) to (0.5,-0.15);
\draw[usual] (1,0) to (1,-0.25);
\end{tikzpicture} & \begin{tikzpicture}[anchorbase]
\draw[usual,dot] (1,0) to (1,0.15);
\draw[usual,dot] (0.5,0) to (0.5,0.15);
\draw[usual] (0,0) to (0,0.25);
\draw[usual,dot] (0,0) to (0,-0.15);
\draw[usual,dot] (0.5,0) to (0.5,-0.15);
\draw[usual] (1,0) to (1,-0.25);
\end{tikzpicture}
\\
\hline
\begin{tikzpicture}[anchorbase]
\draw[usual,dot] (0,0) to (0,-0.15);
\draw[usual,dot] (1,0) to (1,-0.15);
\draw[usual] (0.5,0) to (0.5,-0.25);
\end{tikzpicture}	& \begin{tikzpicture}[anchorbase]
\draw[usual,dot] (0,0) to (0,0.15);
\draw[usual,dot] (0.5,0) to (0.5,0.15);
\draw[usual] (1,0) to (1,0.25);
\draw[usual,dot] (0,0) to (0,-0.15);
\draw[usual,dot] (1,0) to (1,-0.15);
\draw[usual] (0.5,0) to (0.5,-0.25);
\end{tikzpicture} & \begin{tikzpicture}[anchorbase]
\draw[usual,dot] (0,0) to (0,0.15);
\draw[usual,dot] (1,0) to (1,0.15);
\draw[usual] (0.5,0) to (0.5,0.25);
\draw[usual,dot] (0,0) to (0,-0.15);
\draw[usual,dot] (1,0) to (1,-0.15);
\draw[usual] (0.5,0) to (0.5,-0.25);
\end{tikzpicture} & \begin{tikzpicture}[anchorbase]
\draw[usual,dot] (1,0) to (1,0.15);
\draw[usual,dot] (0.5,0) to (0.5,0.15);
\draw[usual] (0,0) to (0,0.25);
\draw[usual,dot] (0,0) to (0,-0.15);
\draw[usual,dot] (1,0) to (1,-0.15);
\draw[usual] (0.5,0) to (0.5,-0.25);
\end{tikzpicture}
\\
\hline
\begin{tikzpicture}[anchorbase]
\draw[usual,dot] (1,0) to (1,-0.15);
\draw[usual,dot] (0.5,0) to (0.5,-0.15);
\draw[usual] (0,0) to (0,-0.25);
\end{tikzpicture}	& \begin{tikzpicture}[anchorbase]
\draw[usual,dot] (0,0) to (0,0.15);
\draw[usual,dot] (0.5,0) to (0.5,0.15);
\draw[usual] (1,0) to (1,0.25);
\draw[usual,dot] (1,0) to (1,-0.15);
\draw[usual,dot] (0.5,0) to (0.5,-0.15);
\draw[usual] (0,0) to (0,-0.25);
\end{tikzpicture} & \begin{tikzpicture}[anchorbase]
\draw[usual,dot] (0,0) to (0,0.15);
\draw[usual,dot] (1,0) to (1,0.15);
\draw[usual] (0.5,0) to (0.5,0.25);
\draw[usual,dot] (1,0) to (1,-0.15);
\draw[usual,dot] (0.5,0) to (0.5,-0.15);
\draw[usual] (0,0) to (0,-0.25);
\end{tikzpicture} & \begin{tikzpicture}[anchorbase]
\draw[usual,dot] (1,0) to (1,0.15);
\draw[usual,dot] (0.5,0) to (0.5,0.15);
\draw[usual] (0,0) to (0,0.25);
\draw[usual,dot] (1,0) to (1,-0.15);
\draw[usual,dot] (0.5,0) to (0.5,-0.15);
\draw[usual] (0,0) to (0,-0.25);
\end{tikzpicture}
\end{tabular}
\end{gathered}};
\endxy
\quad
\raisebox{-0.3cm}{$\leftrightsquigarrow
\gmatrix[1]=
\begin{pmatrix}
\paraa_0^{2} & 0 & 0
\\
0 & \paraa_0^{2} & 0
\\
0 & 0 & \paraa_0^{2}
\end{pmatrix}$.}
\end{gather*}
Since a diagram that is not idempotent cannot be made so by adding Möbius dots (see \autoref{L:Idempotents}), it is relatively straightforward to see that $\gmatrix$ in general can be put into block diagonal form with each block given by $\gmatrix[0]$ evaluated for the case $n-\lambda$. The condition that $\gmatrix$ be full rank, \ie \, $\det(\gmatrix) \neq 0$, is therefore given by \autoref{gramcond}, and the size of $\gmatrix$ is determined by $\#\lcell_{\lambda}$ in \autoref{Eq:DAlgebrasTable} with $K=1$, \ie \, $\binom{n}{\lambda} \cdot 3^{n-\lambda}$.

We now turn our attention to proving that matrices $A_n$ and $B_n$ exist with the required properties. We already saw in \autoref{roexp} that there are such matrices $A_1$, $B_1$, and it is the procedure shown there to find them that we would like to generalise.

To begin, consider the following ordering of left cells: first define the sequence 
\begin{align*}
s_1:\quad 
\begin{tikzpicture}[anchorbase]
\draw[usual,mob=0.5, dot] (0,0) to (0,0.6);
\end{tikzpicture}, \quad \begin{tikzpicture}[anchorbase]
\draw[usual,mob=0.33, mob=0.66, dot] (0,0) to (0,0.6);
\end{tikzpicture},
\end{align*}
and then inductively for $k>1$ the sequence
\begin{align*}
s_k: \quad 
\begin{tikzpicture}[anchorbase]
\begin{scope}[on background layer]
\draw[usual, dot] (0,0) to (0,0.8);
\draw[usual,dot] (0.5,0) to (0.5,0.8);
\draw[usual,dot] (1.5,0) to (1.5,0.8);
\draw[usual, mob=0.5, dot] (2,0) to (2,0.8);
\node at (1,0.8) {$\ldots$};
\node at (1,0) {$\ldots$};
\end{scope}
\node[rectangle, draw, fill=white, minimum width=1.7cm, minimum height = 0.25cm] at (0.75,0.4) {$s_{k-1}$};
\end{tikzpicture}, \quad \begin{tikzpicture}[anchorbase]
\begin{scope}[on background layer]
\draw[usual, dot] (0,0) to (0,0.8);
\draw[usual,dot] (0.5,0) to (0.5,0.8);
\draw[usual,dot] (1.5,0) to (1.5,0.8);
\draw[usual,mob=0.33, mob=0.66, dot] (2,0) to (2,0.8);
\node at (1,0.8) {$\ldots$};
\node at (1,0) {$\ldots$};
\end{scope}
\node[rectangle, draw, fill=white, minimum width=1.7cm, minimum height = 0.25cm] at (0.75,0.4) {$s_{k-1}$};
\end{tikzpicture}.
\end{align*}
The ordering of left cells is then given by first grouping bottom elements together if they contain Möbius dots on the same $i$ components and ordering them according to $s_i$. Now suppose that two groups $I$ and $J$ containing Möbius dots on $i$ and $j$ components respectively. If $i < j$, group $J$ is placed to the right of group $I$ in the ordering, otherwise if $i=j$, components containing Möbius dots that are furthest to the left in the diagrams appear first.

\begin{Example}\label{n2}
For $n=2$ the ordering of left cells is given by
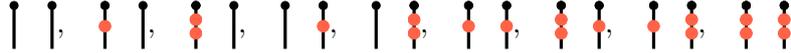
\begin{align*}
\begin{tikzpicture}[anchorbase]
\draw[usual,dot] (0,0) to (0,0.6);
\draw[usual,dot] (0.5,0) to (0.5,0.6);
\end{tikzpicture}, \quad \begin{tikzpicture}[anchorbase]
\draw[usual, mob=0.5,dot] (0,0) to (0,0.6);
\draw[usual,dot] (0.5,0) to (0.5,0.6);
\end{tikzpicture}, \quad \begin{tikzpicture}[anchorbase]
\draw[usual, mob=0.33, mob=0.66,dot] (0,0) to (0,0.6);
\draw[usual,dot] (0.5,0) to (0.5,0.6);
\end{tikzpicture} , \quad \begin{tikzpicture}[anchorbase]
\draw[usual,dot] (0,0) to (0,0.6);
\draw[usual,mob=0.5, dot] (0.5,0) to (0.5,0.6);
\end{tikzpicture} , \quad \begin{tikzpicture}[anchorbase]
\draw[usual,dot] (0,0) to (0,0.6);
\draw[usual,mob=0.33, mob=0.66, dot] (0.5,0) to (0.5,0.6);
\end{tikzpicture} , \quad \begin{tikzpicture}[anchorbase]
\draw[usual, mob=0.5, dot] (0,0) to (0,0.6);
\draw[usual,mob=0.5, dot] (0.5,0) to (0.5,0.6);
\end{tikzpicture} , \quad \begin{tikzpicture}[anchorbase]
\draw[usual,mob=0.33, mob=0.66, dot] (0,0) to (0,0.6);
\draw[usual,mob=0.5, dot] (0.5,0) to (0.5,0.6);
\end{tikzpicture} , \quad \begin{tikzpicture}[anchorbase]
\draw[usual,mob=0.5, dot] (0,0) to (0,0.6);
\draw[usual,mob=0.33, mob=0.66, dot] (0.5,0) to (0.5,0.6);
\end{tikzpicture}, \quad \begin{tikzpicture}[anchorbase]
\draw[usual, mob=0.33, mob=0.66, dot] (0,0) to (0,0.6);
\draw[usual,mob=0.33, mob=0.66, dot] (0.5,0) to (0.5,0.6);
\end{tikzpicture}.
\end{align*}
This pattern generalizes.
\end{Example}

The number of groups $I$ containing Möbius dots on $i$ components is $\binom{n}{i}$, and they each contain $2^{i}$ elements, with the total length of the sequence being $\sum_{i=0}^n\binom{n}{i}\cdot2^i = 3^n$ as expected from \autoref{Eq:DAlgebrasTable}.

We now order the right cells by simply taking the involution of the left cell ordering so that $\gmatrix[0]$ is a symmetric matrix. The next step is to row reduce $\gmatrix[0]$ such that we are left with two diagonal blocks: a $2^n \times 2^n$ minor in the bottom right, which we identify with $B_n$, and a complementary block, from which will emerge $A_n$ following further reductions. 

We begin by observing that the first $u \coloneqq \sum_{i=0}^{n-1}\binom{n}{i}\cdot 2^i$ rows of $\gmatrix[0]$ each contain at least one entry $\paraa_0^a\parab_0^b\parac_0^c$ with $a$ nonzero; for example, take the diagonal. The last $2^n$ rows (and last $2^n$ columns by symmetry), by contrast, have entries with $a=0$ since there is no way to eliminate a Möbius dot. We now consider each of the first $u$ rows, where there are components without any Möbius dots in the corresponding right cell ordering, and subtract a row from the last $2^n$ in which those same components are replaced by a component containing two Möbius dots.

\begin{Example}
For $n=4$, the second row illustrated is subtracted from the first: 
\begin{gather*}
\xy
(0,0)*{\begin{gathered}
\begin{tabular}{C||C|C}
\arrayrulecolor{tomato}
\rcell/\lcell & \begin{tikzpicture}[anchorbase]
\draw[usual,dot] (0,0) to (0,0.25);
\draw[usual,dot] (0.5,0) to (0.5,0.25);
\draw[usual, dot] (1,0) to (1,0.25);
\draw[usual, dot] (1.5,0) to (1.5,0.25);
\end{tikzpicture} & \begin{tikzpicture}[anchorbase]
\node at (0,0) {$\ldots$};
\end{tikzpicture} 
\\
\hline
\hline
\begin{tikzpicture}[anchorbase]
\node at (0,0) {$\ldots$};
\end{tikzpicture}	& \begin{tikzpicture}[anchorbase]
\node at (0,0) {$\ldots$};
\end{tikzpicture} & \begin{tikzpicture}[anchorbase]
\node at (0,0) {$\ldots$};
\end{tikzpicture} 
\\
\hline
\begin{tikzpicture}[anchorbase]
\draw[usual,dot] (0,0) to (0,-0.6);
\draw[usual,dot] (0.5,0) to (0.5,-0.6);
\draw[usual, mob=0.33, mob=0.66, dot] (1,0) to (1,-0.6);
\draw[usual, mob=0.5, dot] (1.5,0) to (1.5,-0.6);
\end{tikzpicture}	& \begin{tikzpicture}[anchorbase]
\draw[usual,dot] (0,0) to (0,0.3);
\draw[usual,dot] (0.5,0) to (0.5,0.3);
\draw[usual, mob=0.25, dot] (1,0) to (1,0.3);
\draw[usual, mob=0, dot] (1.5,0) to (1.5,0.3);
\draw[usual,dot] (0,0) to (0,-0.3);
\draw[usual,dot] (0.5,0) to (0.5,-0.3);
\draw[usual, mob=0.25, dot] (1,0) to (1,-0.3);
\begin{scope}[on background layer]
\draw[usual, dot] (1.5,0) to (1.5,-0.3);
\end{scope}
\end{tikzpicture} & \begin{tikzpicture}[anchorbase]
\node at (0,0) {$\ldots$};
\end{tikzpicture} 
\\
\hline
\begin{tikzpicture}[anchorbase]
\node at (0,0) {$\ldots$};
\end{tikzpicture}	& \begin{tikzpicture}[anchorbase]
\node at (0,0) {$\ldots$};
\end{tikzpicture} & \begin{tikzpicture}[anchorbase]
\node at (0,0) {$\ldots$};
\end{tikzpicture} 
\\
\hline
\begin{tikzpicture}[anchorbase]
\draw[usual, mob=0.33, mob=0.66, dot] (0,0) to (0,-0.6);
\draw[usual, mob=0.33, mob=0.66, dot] (0.5,0) to (0.5,-0.6);
\draw[usual, mob=0.33, mob=0.66, dot] (1,0) to (1,-0.6);
\draw[usual, mob=0.5, dot] (1.5,0) to (1.5,-0.6);
\end{tikzpicture}	& \begin{tikzpicture}[anchorbase]
\draw[usual, mob=0.25, dot] (0,0) to (0,0.3);
\draw[usual,mob=0.25, dot] (0.5,0) to (0.5,0.3);
\draw[usual, mob=0.25, dot] (1,0) to (1,0.3);
\draw[usual, mob=0, dot] (1.5,0) to (1.5,0.3);
\draw[usual,mob=0.25, dot] (0,0) to (0,-0.3);
\draw[usual,mob=0.25, dot] (0.5,0) to (0.5,-0.3);
\draw[usual, mob=0.25, dot] (1,0) to (1,-0.3);
\begin{scope}[on background layer]
\draw[usual, dot] (1.5,0) to (1.5,-0.3);
\end{scope}
\end{tikzpicture} & \begin{tikzpicture}[anchorbase]
\node at (0,0) {$\ldots$};
\end{tikzpicture} \\
\hline
\begin{tikzpicture}[anchorbase]
\node at (0,0) {$\ldots$};
\end{tikzpicture}	& \begin{tikzpicture}[anchorbase]
\node at (0,0) {$\ldots$};
\end{tikzpicture} & \begin{tikzpicture}[anchorbase]
\node at (0,0) {$\ldots$};
\end{tikzpicture} 
\end{tabular}
\end{gathered}};
\endxy
\quad
\raisebox{-0.3cm}{$\leftrightsquigarrow
\gmatrix[0]=
\begin{pmatrix}
\cdots & \cdots
\\
\paraa_0^2\parab_0\parac_0 & \cdots
\\
\cdots & \cdots\\
\parab_0\parac_0^3 & \cdots \\
\cdots & \cdots
\end{pmatrix}$} \quad 
\raisebox{-0.3cm}{$\sim \quad  \begin{pmatrix}
\cdots & \cdots
\\
(\paraa_0^2-\parac_0^2)\parab_0\parac_0 & \cdots
\\
\cdots & \cdots\\
\parab_0\parac_0^3 & \cdots \\
\cdots & \cdots
\end{pmatrix},$}
\end{gather*}
as desired.
\end{Example}

Importantly, since 
\begin{gather*}
\begin{tikzpicture}[anchorbase]
\draw[usual, mob=0.25, mob=0.5, mob=0.75] (0,0) to (0,1);
\end{tikzpicture}
=
\begin{tikzpicture}[anchorbase]
\draw[usual,  mob=0.5] (0,0) to (0,1);
\end{tikzpicture}
\hspace{1cm} \text{and} \hspace{1cm} \begin{tikzpicture}[anchorbase]
\draw[usual, mob=0.2, mob=0.4, mob=0.6, mob=0.8] (0,0) to (0,1);
\end{tikzpicture}
=
\begin{tikzpicture}[anchorbase]
\draw[usual, mob=0.66, mob=0.33] (0,0) to (0,1);
\end{tikzpicture},
\end{gather*}
the resulting $u \times 2^n$ minor in the upper right corner is the (identically) zero matrix as the replacement does not alter the number of Möbius dots originally present. 

\begin{Remark}
When performing the reduction, we are concerned with generic parameters $\paraa_0, \parab_0, \parac_0$, and so are in particular, not identically zero. The point is to produce identically zero entries in the reduction. 
\end{Remark}

With the first step complete, we now examine $B_n$, the $2^n \times 2^n$ minor in the bottom right corner. By construction of the cell ordering, it is not too difficult to see that $B_n$ is of the form
\begin{align*}
\begin{pmatrix}
\parac_0 B_{n-1}& \parab_0 B_{n-1} \\
\parab_0 B_{n-1} & \parac_0 B_{n-1}
\end{pmatrix} \hspace{0.25cm} = \hspace{0.25cm}\begin{pmatrix}
\parac_0 & \parab_0  \\
\parab_0  & \parac_0 
\end{pmatrix} \otimes B_{n-1},
\end{align*}
and therefore $\det(B_n)=(\parac_0^2-\parab_0^2)^{n\cdot 2^{n-1}}$ by induction.

We now turn our attention to the remaining $u \times u$ minor in the top left corner, which is currently of the form 
\begin{align}\label{preaform}
\begin{pmatrix}
\paraa_0^n - \parac_0^n  & \cdots & * & \cdots & * \\
\vdots & \ddots & \vdots & \vdots & \vdots \\
* & \cdots & (\paraa_0^i - \parac_0^i)B_{n-i} & \cdots & * \\
\vdots & \vdots & \vdots & \ddots & \vdots \\
* & \cdots & * & \cdots & (\paraa_0 - \parac_0)B_{n-1}
\end{pmatrix},
\end{align}
with the number of blocks of form $(\paraa_0^i - \parac_0^i)B_{n-i}$ on the diagonal given by $\binom{n}{i}$. We would like to reduce \autoref{preaform} such that the entries above the diagonal blocks are zero. 

\begin{Example}
We work through the case $A_2$ as an illustration, where the ordering of left cells is given in \autoref{n2}. Following the first stage reduction, we obtain a matrix of size $u \times u$ where $u=\sum_{i=0}^{2-1}\binom{2}{i}\cdot 2^i=1+4=5$,
\begin{gather*}
\begin{blockarray}{cccccc}
&  \begin{tikzpicture}[anchorbase]
\draw[usual,dot] (0,0) to (0,0.6);
\draw[usual,dot] (0.5,0) to (0.5,0.6);
\end{tikzpicture}  & \begin{tikzpicture}[anchorbase]
\draw[usual, mob=0.5,dot] (0,0) to (0,0.6);
\draw[usual,dot] (0.5,0) to (0.5,0.6);
\end{tikzpicture} & \begin{tikzpicture}[anchorbase]
\draw[usual, mob=0.33, mob=0.66,dot] (0,0) to (0,0.6);
\draw[usual,dot] (0.5,0) to (0.5,0.6);
\end{tikzpicture} & \begin{tikzpicture}[anchorbase]
\draw[usual,dot] (0,0) to (0,0.6);
\draw[usual,mob=0.5, dot] (0.5,0) to (0.5,0.6);
\end{tikzpicture} & \begin{tikzpicture}[anchorbase]
\draw[usual,dot] (0,0) to (0,0.6);
\draw[usual,mob=0.33, mob=0.66, dot] (0.5,0) to (0.5,0.6);
\end{tikzpicture}\\
\begin{block}{c(ccccc)}
\begin{tikzpicture}[anchorbase]
\draw[usual,dot] (0,0) to (0,-0.6);
\draw[usual,dot] (0.5,0) to (0.5,-0.6);
\end{tikzpicture} & \paraa_0^2 - \parac_0^2 & (\paraa_0-\parac_0)\parab_0 & (\paraa_0-\parac_0)\parac_0 & (\paraa_0-\parac_0)\parab_0 & (\paraa_0-\parac_0)\parac_0 \\
\begin{tikzpicture}[anchorbase]
\draw[usual, mob=0.5,dot] (0,0) to (0,-0.6);
\draw[usual,dot] (0.5,0) to (0.5,-0.6);
\end{tikzpicture} & (\paraa_0-\parac_0)\parab_0 & (\paraa_0-\parac_0)\parac_0 & (\paraa_0-\parac_0)\parab_0 & 0 & 0 \\
\begin{tikzpicture}[anchorbase]
\draw[usual, mob=0.33, mob=0.66,dot] (0,0) to (0,-0.6);
\draw[usual,dot] (0.5,0) to (0.5,-0.6);
\end{tikzpicture} & (\paraa_0-\parac_0)\parac_0 & (\paraa_0-\parac_0)\parab_0 & (\paraa_0-\parac_0)\parac_0 & 0 & 0 \\
 \begin{tikzpicture}[anchorbase]
\draw[usual,dot] (0,0) to (0,-0.6);
\draw[usual,mob=0.5, dot] (0.5,0) to (0.5,-0.6);
\end{tikzpicture} & (\paraa_0-\parac_0)\parab_0  & 0 & 0 & (\paraa_0-\parac_0)\parac_0 & (\paraa_0-\parac_0)\parab_0 \\
 \begin{tikzpicture}[anchorbase]
\draw[usual,dot] (0,0) to (0,-0.6);
\draw[usual,mob=0.33, mob=0.66, dot] (0.5,0) to (0.5,-0.6);
\end{tikzpicture}& (\paraa_0-\parac_0)\parac_0 & 0 & 0 & (\paraa_0-\parac_0)\parab_0 & (\paraa_0-\parac_0)\parac_0 \\
\end{block}
\end{blockarray}
\,.
\end{gather*}
We can immediately see three diagonal blocks corresponding to the three groupings of left cells: $\binom{2}{2-2}=\binom{2}{0}=1$ block of $(\paraa_0^2 - \parac_0^2)B_0=(\paraa_0^2 - \parac_0^2)$ and $\binom{2}{2-1}=\binom{2}{1}=2$ blocks of 
\begin{align*}
    (\paraa_0 - \parac_0)B_{1} =     (\paraa_0 - \parac_0)\begin{pmatrix}
\parac_0 & \parab_0  \\
\parab_0  & \parac_0 
\end{pmatrix}.
\end{align*}

Starting at the bottom right above the diagonal blocks at position $(3,5)$, we see that the entry is zero and so are the entries corresponding to the rest of the final left cell grouping on row 3 (and also the corresponding right cell grouping extending to row 2). 

We are left with the first row: starting with the final left cell grouping corresponding to $(1, 4)$ and $(1, 5)$, we note that a right cell representative with matching parts without Möbius dots and two Möbius dots otherwise corresponds to the last row. Therefore, subtracting row 5 from row 1 sets the entries $(1, 4)$ and $(1, 5)$ to zero. The entries $(1, 2)$ and $(1, 3)$ corresponding to the next left cell grouping can be set to zero in the same manner by subtracting row 3 from row 1. 

We are then left with final block lower triangular form with the $(1, 1)$ entry augmented by:
\begin{align*}
\paraa_0^2 - \parac_0^2-2(\paraa_0-\parac_0)\parac_0 =\paraa_0^2 - \parac_0^2-\sum_{k=1}^{2-1}\binom{2}{k}(\paraa_0^{2-k}\parac_0^k - \parac_0^2),
\end{align*}
while the other diagonal blocks are scaled by $\paraa_0-\parac_0 =\paraa_0 - \parac_0-\sum_{k=1}^{1-1}\binom{1}{k}(\paraa_0^{1-k}\parac_0^k - \parac_0^1)$, where the final sum is the empty sum by definition.
\end{Example}

The more general procedure is as follows: beginning at the bottom right and working backwards along rows until encountering the block diagonal entries, we see that an entry is zero (nonzero) whenever the corresponding left and right cell representatives in the ordering have no (at least one) components in common that are without Möbius dots. 

Furthermore, if an entry is zero, all other entries in the same row falling under the same left cell grouping are zero also. Therefore, given a row $\ell$ whose right cell representative contains exactly $i$ parts without Möbius dots, there are $\sum_{k=1}^{i-1}\binom{i}{k}$ left cell groupings to the right of the block diagonal group associated with nonzero entries; for each of these groupings, we subtract from row $\ell$ a lower row whose right cell representative has matching parts without Möbius dots and two Möbius dots on all other components. (Note that this row either intersects or lies below the diagonal block for the corresponding grouping.) 

From here, we see that all entries above the diagonal blocks are zero, and the effect on the diagonal blocks is the following augmentation:
\begin{align*}
\Big((\paraa_0^i - \parac_0^i)-\sum_{k=1}^{i-1}\binom{i}{k}(\paraa_0^{i-k}\parac_0^k - \parac_0^i)  \Big)  B_{n-i},
\end{align*}
for $i=1, \ldots, n$, which completes the proof. 

\textit{(iii).} The proof is identical to the corresponding case in \cite[Theorem 4B.15]{Tu-sandwich-cellular}; namely, $\mrk(\gmatrix[{n-1}])=\binom{n}{n-1}\cdot3^{n-(n-1)}=3n$ as given in (ii) above.
\end{proof}

\begin{Example}
Assume the conditions set out in \autoref{P:DAlgebrasGram}.(ii). Fix $n=5$, $\lambda=2$, and further assume that not all $\paraa_0, \parab_0, \parac_0$ are zero. 

By \autoref{T:DAlgebrasDSimples}, the dimension of the simple $\mromon[5]{\paraa, \parab, \parac}$-modules with apex $2$, given any simple $\sand[2]$-module $K$, is 
\begin{gather*}
\dim\big(\lmod[{2,K}]\big)
=\mrk(\gmatrix[2]).
\end{gather*}

Now, by \autoref{P:DAlgebrasGram}.(ii), we have
\begin{gather*}
\mrk(\gmatrix[2]) = \binom{5}{2}\cdot 3^{5-2}
=270,
\end{gather*}
whenever
\begin{gather*}
(\paraa_0^3-3\paraa_0^2\parac_0 -3\paraa_0\parac_0^2+5\parac_0^3)(\paraa_0-\parac_0)^{24} 
(\parac_0^2-\parab_0^2)^{27}
\neq 0;
\end{gather*}
we can take \eg \, $\paraa_0=\parab_0=1$ and $\parac_0=0$. 

On the other hand, if we take instead $\paraa_0=\parab_0=\parac_0=1$, then by the arguments set out in the proof of \autoref{P:DAlgebrasGram}.(ii), $\gmatrix[2]$ consists of $\binom{5}{2}=10$ blocks of $\gmatrix[0]$ evaluated for the case $5-2=3$, which has rank one. Therefore, we have
\begin{gather*}
\dim\big(\lmod[{2,K}]\big)=\mrk(\gmatrix[2]) = 10,
\end{gather*}
and the dimensions are ten.
\end{Example}


\begin{thebibliography}{BHMV95}
	
	\bibitem[AST18]{AnStTu-cellular-tilting}
	H.H. Andersen, C.~Stroppel, and D.~Tubbenhauer.
	\newblock Cellular structures using {$\mathrm{U}_q$}-tilting modules.
	\newblock {\em Pacific J. Math.}, 292(1):21--59, 2018.
	\newblock URL: \url{https://arxiv.org/abs/1503.00224}, \href
	{https://doi.org/10.2140/pjm.2018.292.21}
	{\path{doi:10.2140/pjm.2018.292.21}}.
	
	\bibitem[BM23]{BaMc-graph-coloring}
	S.~Baldridge and B.~McCarty.
	\newblock A topological quantum field theory approach to graph coloring.
	\newblock 2023.
	\newblock URL: \url{https://arxiv.org/abs/2303.12010}.
	
	\bibitem[BN05]{BaNa-tangles-cobordisms}
	D.~Bar-Natan.
	\newblock Khovanov's homology for tangles and cobordisms.
	\newblock {\em Geom. Topol.}, 9:1443--1499, 2005.
	\newblock URL: \url{https://arxiv.org/abs/math/0410495}, \href
	{https://doi.org/10.2140/gt.2005.9.1443} {\path{doi:10.2140/gt.2005.9.1443}}.
	
	\bibitem[BW10]{BeWa-eqft}
	A.~Beliakova and E.~Wagner.
	\newblock On link homology theories from extended cobordisms.
	\newblock {\em Quantum Topol.}, 1(4):379--398, 2010.
	\newblock URL: \url{https://arxiv.org/abs/0910.5050}, \href
	{https://doi.org/10.4171/QT/9} {\path{doi:10.4171/QT/9}}.
	
	\bibitem[BT18]{BeTh-hwt-findim-algebras}
	G.~Bellamy and U.~Thiel.
	\newblock Highest weight theory for finite-dimensional graded algebras with
	triangular decomposition.
	\newblock {\em Adv. Math.}, 330:361--419, 2018.
	\newblock URL: \url{https://arxiv.org/abs/1705.08024}, \href
	{https://doi.org/10.1016/j.aim.2018.03.011}
	{\path{doi:10.1016/j.aim.2018.03.011}}.
	
	\bibitem[BHMV95]{BlHaMaVo-tqft-kauffman-bracket}
	C.~Blanchet, N.~Habegger, G.~Masbaum, and P.~Vogel.
	\newblock Topological quantum field theories derived from the {K}auffman
	bracket.
	\newblock {\em Topology}, 34(4):883--927, 1995.
	\newblock \href {https://doi.org/10.1016/0040-9383(94)00051-4}
	{\path{doi:10.1016/0040-9383(94)00051-4}}.
	
	\bibitem[Bra37]{Br-brauer-algebra-original}
	R.~Brauer.
	\newblock On algebras which are connected with the semisimple continuous
	groups.
	\newblock {\em Ann. of Math. (2)}, 38(4):857--872, 1937.
	\newblock \href {https://doi.org/10.2307/1968843} {\path{doi:10.2307/1968843}}.
	
	\bibitem[Bro55]{Br-gen-matrix-algebras}
	W.P.~Brown.
	\newblock Generalized matrix algebras.
	\newblock {\em Canadian J. Math.}, 7:188--190, 1955.
	\newblock \href {https://doi.org/10.4153/CJM-1955-023-2}
	{\path{doi:10.4153/CJM-1955-023-2}}.
	
	\bibitem[Com20]{Co-jelly}
	J.~Comes.
	\newblock Jellyfish partition categories.
	\newblock {\em Algebr. Represent. Theory}, 23(2):327--347, 2020.
	\newblock URL: \url{https://arxiv.org/abs/1612.05182}, \href
	{https://doi.org/10.1007/s10468-018-09851-7}
	{\path{doi:10.1007/s10468-018-09851-7}}.
	
	\bibitem[CO11]{CoOs-blocks-deligne-sym-group}
	J.~Comes and V.~Ostrik.
	\newblock On blocks of {D}eligne's category
	{$\underline{\mathrm{Re}}\mathrm{p}(S_{t})$}.
	\newblock {\em Adv. Math.}, 226(2):1331--1377, 2011.
	\newblock URL: \url{https://arxiv.org/abs/0910.5695}, \href
	{https://doi.org/10.1016/j.aim.2010.08.010}
	{\path{doi:10.1016/j.aim.2010.08.010}}.
	
	\bibitem[COT24]{CoOsTu-growth}
	K.~Coulembier, V.~Ostrik, and D.~Tubbenhauer.
	\newblock Growth rates of the number of indecomposable summands in tensor
	powers.
	\newblock {\em Algebr. Represent. Theory}, 27(2):1033--1062, 2024.
	\newblock URL: \url{https://arxiv.org/abs/2301.00885}, \href
	{https://doi.org/10.1007/s10468-023-10245-7}
	{\path{doi:10.1007/s10468-023-10245-7}}.
	
	\bibitem[Cze23]{Cz-mobius-tqft}
	A.~Czenky.
	\newblock Unoriented 2-dimensional {TQFT}s and the category
	{$Rep(S_t\wr\mathbb{Z}_2)$}.
	\newblock 2023.
	\newblock URL: \url{https://arxiv.org/abs/2306.08826}.
	
	\bibitem[DT24]{DeTu-dicyclic}
	P.~DeBello and D.~Tubbenhauer.
	\newblock Diagrammatics for dicyclic groups.
	\newblock 2024.
	\newblock URL: \url{https://arxiv.org/abs/2412.12376}.
	
	\bibitem[Del07]{De-cat-st}
	P.~Deligne.
	\newblock La cat{\'e}gorie des repr{\'e}sentations du groupe sym{\'e}trique
	{$S_t$}, lorsque {$t$} n'est pas un entier naturel.
	\newblock In {\em Algebraic groups and homogeneous spaces}, volume~19 of {\em
		Tata Inst. Fund. Res. Stud. Math.}, pages 209--273. Tata Inst. Fund. Res.,
	Mumbai, 2007.
	
	\bibitem[Eas06]{Ea-cellular-semigroups}
	J.~East.
	\newblock Cellular algebras and inverse semigroups.
	\newblock {\em J. Algebra}, 296(2):505--519, 2006.
	\newblock \href {https://doi.org/10.1016/j.jalgebra.2005.04.027}
	{\path{doi:10.1016/j.jalgebra.2005.04.027}}.
	
	\bibitem[EST17]{EhStTu-blanchet-khovanov}
	M.~Ehrig, C.~Stroppel, and D.~Tubbenhauer.
	\newblock The {B}lanchet--{K}hovanov algebras.
	\newblock In {\em Categorification and higher representation theory}, volume
	683 of {\em Contemp. Math.}, pages 183--226. Amer. Math. Soc., Providence,
	RI, 2017.
	\newblock URL: \url{https://arxiv.org/abs/1510.04884}, \href
	{https://doi.org/10.1090/conm/683} {\path{doi:10.1090/conm/683}}.
	
	\bibitem[EK10]{ElKh-diagrams-soergel}
	B.~Elias and M.~Khovanov.
	\newblock Diagrammatics for {S}oergel categories.
	\newblock {\em Int. J. Math. Math. Sci.}, pages Art. ID 978635, 58, 2010.
	\newblock URL: \url{https://arxiv.org/abs/0902.4700}, \href
	{https://doi.org/10.1155/2010/978635} {\path{doi:10.1155/2010/978635}}.
	
	\bibitem[EMM12]{EMM-twisted-duality}
	J.A.~Ellis-Monaghan and I.~Moffatt.
	\newblock Twisted duality for embedded graphs.
	\newblock {\em Trans. Amer. Math. Soc.}, 364(3):1529--1569, 2012.
	\newblock URL: \url{https://arxiv.org/abs/0906.5557}, \href
	{https://doi.org/10.1090/S0002-9947-2011-05529-7}
	{\path{doi:10.1090/S0002-9947-2011-05529-7}}.
	
	\bibitem[EGNO15]{EtGeNiOs-tensor-categories}
	P.~Etingof, S.~Gelaki, D.~Nikshych, and V.~Ostrik.
	\newblock {\em Tensor categories}, volume 205 of {\em Mathematical Surveys and
		Monographs}.
	\newblock American Mathematical Society, Providence, RI, 2015.
	\newblock \href {https://doi.org/10.1090/surv/205}
	{\path{doi:10.1090/surv/205}}.
	
	\bibitem[EO22]{EtOs-semisimple}
	P.~Etingof and V.~Ostrik.
	\newblock On semisimplification of tensor categories.
	\newblock In {\em Representation theory and algebraic geometry---a conference
		celebrating the birthdays of {S}asha {B}eilinson and {V}ictor {G}inzburg},
	Trends Math., pages 3--35. Birkh\"auser/Springer, Cham, [2022] \copyright
	2022.
	\newblock URL: \url{https://arxiv.org/abs/1801.04409}, \href
	{https://doi.org/10.1007/978-3-030-82007-7\_1}
	{\path{doi:10.1007/978-3-030-82007-7\_1}}.
	
	\bibitem[FG95]{FiGr-canonical-cases-brauer}
	S.~Fishel and I.~Grojnowski.
	\newblock Canonical bases for the {B}rauer centralizer algebra.
	\newblock {\em Math. Res. Lett.}, 2(1):15--26, 1995.
	\newblock \href {https://doi.org/10.4310/MRL.1995.v2.n1.a3}
	{\path{doi:10.4310/MRL.1995.v2.n1.a3}}.
	
	\bibitem[FST25]{FrStTu-twists}
	M.~Fresacher, W.~Stewart, and D.~Tubbenhauer.
	\newblock Generalized diagram categories and monoids, and their
	representations.
	\newblock 2025.
	\newblock URL: \url{https://arxiv.org/abs/2512.17177}.
	
	\bibitem[GMS09]{GaMaSt-irreps-semigroups}
	O.~Ganyushkin, V.~Mazorchuk, and B.~Steinberg.
	\newblock On the irreducible representations of a finite semigroup.
	\newblock {\em Proc. Amer. Math. Soc.}, 137(11):3585--3592, 2009.
	\newblock URL: \url{https://arxiv.org/abs/0712.2076}, \href
	{https://doi.org/10.1090/S0002-9939-09-09857-8}
	{\path{doi:10.1090/S0002-9939-09-09857-8}}.
	
	\bibitem[GL96]{GrLe-cellular}
	J.J.~Graham and G.~Lehrer.
	\newblock Cellular algebras.
	\newblock {\em Invent. Math.}, 123(1):1--34, 1996.
	\newblock \href {https://doi.org/10.1007/BF01232365}
	{\path{doi:10.1007/BF01232365}}.
	
	\bibitem[Gre51]{Gr-structure-semigroups}
	J.A.~Green.
	\newblock On the structure of semigroups.
	\newblock {\em Ann. of Math. (2)}, 54:163--172, 1951.
	\newblock \href {https://doi.org/10.2307/1969317} {\path{doi:10.2307/1969317}}.
	
	\bibitem[Gre98]{Gr-dots-tl}
	R.M.~Green.
	\newblock Generalized {T}emperley--{L}ieb algebras and decorated tangles.
	\newblock {\em J. Knot Theory Ramifications}, 7(2):155--171, 1998.
	\newblock URL: \url{https://arxiv.org/abs/q-alg/9712018}, \href
	{https://doi.org/10.1142/S0218216598000103}
	{\path{doi:10.1142/S0218216598000103}}.
	
	\bibitem[GT25]{GrTu-diagram-growth}
	J.~Gruber and D.~Tubbenhauer.
	\newblock Growth problems in diagram categories.
	\newblock {\em Bull. Lond. Math. Soc.}, 57(11):3454--3469, 2025.
	\newblock URL: \url{https://arxiv.org/abs/2503.00685}, \href
	{https://doi.org/10.1112/blms.70163} {\path{doi:10.1112/blms.70163}}.
	
	\bibitem[GW15]{GuWi-almost-cellular}
	N.~Guay and S.~Wilcox.
	\newblock Almost cellular algebras.
	\newblock {\em J. Pure Appl. Algebra}, 219(9):4105--4116, 2015.
	\newblock \href {https://doi.org/10.1016/j.jpaa.2015.02.010}
	{\path{doi:10.1016/j.jpaa.2015.02.010}}.
	
	\bibitem[HR05]{HaRa-partition-algebras}
	T.~Halverson and A.~Ram.
	\newblock Partition algebras.
	\newblock {\em European J. Combin.}, 26(6):869--921, 2005.
	\newblock URL: \url{https://arxiv.org/abs/math/0401314}, \href
	{https://doi.org/10.1016/j.ejc.2004.06.005}
	{\path{doi:10.1016/j.ejc.2004.06.005}}.
	
	\bibitem[HT25a]{HeTu-affine-monoid}
	D.~He and D.~Tubbenhauer.
	\newblock Affine diagram categories, algebras and monoids.
	\newblock 2025.
	\newblock URL: \url{https://arxiv.org/abs/2512.05510}.
	
	\bibitem[HT25b]{HeTu-monoid-growth}
	D.~He and D.~Tubbenhauer.
	\newblock Growth problems for representations of finite monoids.
	\newblock {\em North-West. Eur. J. Math.}, 11:103--117, i, 2025.
	\newblock URL: \url{https://arxiv.org/abs/2502.02849}.
	
	\bibitem[JK81]{JaKe-reptheory-symmetric-group}
	G.~James and A.~Kerber.
	\newblock {\em The representation theory of the symmetric group}, volume~16 of
	{\em Encyclopedia of Mathematics and its Applications}.
	\newblock Addison-Wesley Publishing Co., Reading, Mass., 1981.
	\newblock With a foreword by P. M. Cohn, With an introduction by Gilbert de B.
	Robinson.
	
	\bibitem[Jon94]{Jo-potts-sym}
	V.~F.~R. Jones.
	\newblock The {P}otts model and the symmetric group.
	\newblock In {\em Subfactors ({K}yuzeso, 1993)}, pages 259--267. World Sci.
	Publ., River Edge, NJ, 1994.
	
	\bibitem[KMY19]{KaMaYu-extension-tl}
	Z.~K{\'a}d{\'a}r, P.P.~Martin, and S.~Yu.
	\newblock On geometrically defined extensions of the {T}emperley--{L}ieb
	category in the {B}rauer category.
	\newblock {\em Math. Z.}, 293(3-4):1247--1276, 2019.
	\newblock URL: \url{https://arxiv.org/abs/1401.1774}, \href
	{https://doi.org/10.1007/s00209-019-02246-4}
	{\path{doi:10.1007/s00209-019-02246-4}}.
	
	\bibitem[Koc04]{Ko-tqfts}
	J.~Kock.
	\newblock {\em Frobenius algebras and 2{D} topological quantum field theories},
	volume~59 of {\em London Mathematical Society Student Texts}.
	\newblock Cambridge University Press, Cambridge, 2004.
	
	\bibitem[KOK22]{KhOsKo-cobordisms}
	M.~Khovanov, V.~Ostrik, and Y.~Kononov.
	\newblock Two-dimensional topological theories, rational functions and their
	tensor envelopes.
	\newblock {\em Selecta Math. (N.S.)}, 28(4):Paper No. 71, 68, 2022.
	\newblock URL: \url{https://arxiv.org/abs/2011.14758}, \href
	{https://doi.org/10.1007/s00029-022-00785-z}
	{\path{doi:10.1007/s00029-022-00785-z}}.
	
	\bibitem[KS24]{KhSa-cobordisms}
	M.~Khovanov and R.~Sazdanovic.
	\newblock Bilinear pairings on two-dimensional cobordisms and generalizations
	of the {D}eligne category.
	\newblock {\em Fund. Math.}, 264(1):1--20, 2024.
	\newblock URL: \url{https://arxiv.org/abs/2007.11640}, \href
	{https://doi.org/10.4064/fm283-8-2023} {\path{doi:10.4064/fm283-8-2023}}.
	
	\bibitem[KST24]{KhSiTu-monoidal-cryptography}
	M.~Khovanov, M.~Sitaraman, and D.~Tubbenhauer.
	\newblock Monoidal categories, representation gap and cryptography.
	\newblock {\em Trans. Amer. Math. Soc. Ser. B}, 11:329--395, 2024.
	\newblock URL: \url{https://arxiv.org/abs/2201.01805}, \href
	{https://doi.org/10.1090/btran/151} {\path{doi:10.1090/btran/151}}.
	
	\bibitem[KX12]{KoXi-affine-cellular}
	S.~K{\"o}nig and C.~Xi.
	\newblock Affine cellular algebras.
	\newblock {\em Adv. Math.}, 229(1):139--182, 2012.
	\newblock \href {https://doi.org/10.1016/j.aim.2011.08.010}
	{\path{doi:10.1016/j.aim.2011.08.010}}.
	
	\bibitem[Kup96]{Ku-spiders-rank-2}
	G.~Kuperberg.
	\newblock Spiders for rank {$2$} {L}ie algebras.
	\newblock {\em Comm. Math. Phys.}, 180(1):109--151, 1996.
	\newblock URL: \url{https://arxiv.org/abs/q-alg/9712003}, \href
	{https://doi.org/10.1007/bf02101184} {\path{doi:10.1007/bf02101184}}.
	
	\bibitem[LRMD23]{LaReMoDu-affine-tl-monoid}
	A.~Langlois-R{\'e}millard and A.~Morin-Duchesne.
	\newblock Uncoiled affine {T}emperley--{L}ieb algebras and their
	{W}enzl--{J}ones projectors.
	\newblock 2023.
	\newblock URL: \url{https://arxiv.org/abs/2302.12782}.
	
	\bibitem[Liu25]{Li-cyclic-monoid}
	J.~Liu.
	\newblock Representations of cyclic diagram monoids.
	\newblock 2025.
	\newblock URL: \url{https://arxiv.org/abs/2511.15945}.
	
	\bibitem[Mar94]{Ma-potts-sym}
	P.~Martin.
	\newblock Temperley--{L}ieb algebras for nonplanar statistical mechanics---the
	partition algebra construction.
	\newblock {\em J. Knot Theory Ramifications}, 3(1):51--82, 1994.
	\newblock \href {https://doi.org/10.1142/S0218216594000071}
	{\path{doi:10.1142/S0218216594000071}}.
	
	\bibitem[MS94]{MaSa-blob}
	P.~Martin and H.~Saleur.
	\newblock The blob algebra and the periodic {T}emperley--{L}ieb algebra.
	\newblock {\em Lett. Math. Phys.}, 30(3):189--206, 1994.
	\newblock URL: \url{https://arxiv.org/abs/hep-th/9302094}, \href
	{https://doi.org/10.1007/BF00805852} {\path{doi:10.1007/BF00805852}}.
	
	\bibitem[MQS15]{MaQuSt-characters-monoids}
	A.M.~Masuda, L.~Quoos, and B.~Steinberg.
	\newblock Character theory of monoids over an arbitrary field.
	\newblock {\em J. Algebra}, 431:107--126, 2015.
	\newblock URL: \url{https://arxiv.org/abs/1410.7107}, \href
	{https://doi.org/10.1016/j.jalgebra.2015.02.017}
	{\path{doi:10.1016/j.jalgebra.2015.02.017}}.
	
	\bibitem[MT24]{MaTu-klrw-algebras}
	A.~Mathas and D.~Tubbenhauer.
	\newblock Cellularity and subdivision of {KLR} and weighted {KLRW} algebras.
	\newblock {\em Math. Ann.}, 389(3):3043--3122, 2024.
	\newblock URL: \url{https://arxiv.org/abs/2111.12949}, \href
	{https://doi.org/10.1007/s00208-023-02660-4}
	{\path{doi:10.1007/s00208-023-02660-4}}.
	
	\bibitem[MT23a]{MaTu-klrw-algebras-bad}
	A.~Mathas and D.~Tubbenhauer.
	\newblock Cellularity for weighted {KLRW} algebras of types {$B$}, {$A^{(2)}$},
	{$D^{(2)}$}.
	\newblock {\em J. Lond. Math. Soc. (2)}, 107(3):1002--1044, 2023.
	\newblock URL: \url{https://arxiv.org/abs/2201.01998}, \href
	{https://doi.org/10.1112/jlms.12706} {\path{doi:10.1112/jlms.12706}}.
	
	\bibitem[MT23b]{MaTu-klrw-crystal}
	A.~Mathas and D.~Tubbenhauer.
	\newblock Cellularity of {KLR} and weighted {KLRW} algebras via crystals.
	\newblock 2023.
	\newblock URL: \url{https://arxiv.org/abs/2309.13867}.
	
	\bibitem[MPS17]{MoPeSn-categories-trivalent-vertex}
	S.~Morrison, E.~Peters, and N.~Snyder.
	\newblock Categories generated by a trivalent vertex.
	\newblock {\em Selecta Math. (N.S.)}, 23(2):817--868, 2017.
	\newblock URL: \url{https://arxiv.org/abs/1501.06869}, \href
	{https://doi.org/10.1007/s00029-016-0240-3}
	{\path{doi:10.1007/s00029-016-0240-3}}.
	
	\bibitem[RTW32]{RuTeWe-sl2}
	G.~Rumer, E.~Teller, and H.~Weyl.
	\newblock Eine f{\"u}r die {V}alenztheorie geeignete {B}asis der bin{\"a}ren
	{V}ektorinvarianten.
	\newblock {\em Nachrichten von der Ges. der Wiss. Zu G{\"o}ttingen. Math.-Phys.
		Klasse}, pages 498--504, 1932.
	\newblock In German.
	
	\bibitem[SS22]{SaSn-triangular}
	S.V.~Sam and A.~Snowden.
	\newblock The representation theory of {B}rauer categories {I}: {T}riangular
	categories.
	\newblock {\em Appl. Categ. Structures}, 30(6):1203--1256, 2022.
	\newblock URL: \url{https://arxiv.org/abs/2006.04328}, \href
	{https://doi.org/10.1007/s10485-022-09689-7}
	{\path{doi:10.1007/s10485-022-09689-7}}.
	
	\bibitem[Ste16]{St-rep-monoid}
	B.~Steinberg.
	\newblock {\em Representation theory of finite monoids}.
	\newblock Universitext. Springer, Cham, 2016.
	\newblock \href {https://doi.org/10.1007/978-3-319-43932-7}
	{\path{doi:10.1007/978-3-319-43932-7}}.
	
	\bibitem[Tag13]{Ta-uhqft}
	K.~Tagami.
	\newblock A {K}hovanov type invariant derived from an unoriented {HQFT} for
	links in thickened surfaces.
	\newblock {\em Internat. J. Math.}, 24(10):1350078, 28, 2013.
	\newblock URL: \url{https://arxiv.org/abs/1311.1909}, \href
	{https://doi.org/10.1142/S0129167X1350078X}
	{\path{doi:10.1142/S0129167X1350078X}}.
	
	\bibitem[Tub24a]{Tu-web-reps}
	D.~Tubbenhauer.
	\newblock On rank one 2-representations of web categories.
	\newblock {\em Algebr. Comb.}, 7(6):1813--1843, 2024.
	\newblock URL: \url{https://arxiv.org/abs/2307.00785}, \href
	{https://doi.org/10.5802/alco.389} {\path{doi:10.5802/alco.389}}.
	
	\bibitem[Tub25]{Tu-qt}
	D.~Tubbenhauer.
	\newblock Quantum topology without topology.
	\newblock 2025.
	\newblock URL: \url{https://arxiv.org/abs/2307.00785}.
	
	\bibitem[Tub24b]{Tu-sandwich-cellular}
	D.~Tubbenhauer.
	\newblock Sandwich cellularity and a version of cell theory.
	\newblock {\em Rocky Mountain J. Math.}, 54(6):1733--1773, 2024.
	\newblock URL: \url{https://arxiv.org/abs/2206.06678}, \href
	{https://doi.org/10.1216/rmj.2024.54.1733}
	{\path{doi:10.1216/rmj.2024.54.1733}}.
	
	\bibitem[Tub14]{Tu-virtual-khovanov}
	D.~Tubbenhauer.
	\newblock Virtual {K}hovanov homology using cobordisms.
	\newblock {\em J. Knot Theory Ramifications}, 23(9):1450046, 91, 2014.
	\newblock URL: \url{https://arxiv.org/abs/1111.0609}, \href
	{https://doi.org/10.1142/S0218216514500461}
	{\path{doi:10.1142/S0218216514500461}}.
	
	\bibitem[TV23]{TuVa-handlebody}
	D.~Tubbenhauer and P.~Vaz.
	\newblock Handlebody diagram algebras.
	\newblock {\em Rev. Mat. Iberoam.}, 39(3):845--896, 2023.
	\newblock URL: \url{https://arxiv.org/abs/2105.07049}, \href
	{https://doi.org/10.4171/rmi/1356} {\path{doi:10.4171/rmi/1356}}.
	
	\bibitem[TT06]{TuTu-utqft}
	V.~Turaev and P.~Turner.
	\newblock Unoriented topological quantum field theory and link homology.
	\newblock {\em Algebr. Geom. Topol.}, 6:1069--1093, 2006.
	\newblock URL: \url{https://arxiv.org/abs/math/0506229}, \href
	{https://doi.org/10.2140/agt.2006.6.1069}
	{\path{doi:10.2140/agt.2006.6.1069}}.
	
	\bibitem[TV17]{TuVi-monoidal-tqft}
	V.G.~Turaev and A.~Virelizier.
	\newblock {\em Monoidal categories and topological field theory}, volume 322 of
	{\em Progress in Mathematics}.
	\newblock Birkh{\"a}user/Springer, Cham, 2017.
	\newblock \href {https://doi.org/10.1007/978-3-319-49834-8}
	{\path{doi:10.1007/978-3-319-49834-8}}.
	
	\bibitem[Yam89]{Ya-invariant-graphs}
	S.~Yamada.
	\newblock An invariant of spatial graphs.
	\newblock {\em J. Graph Theory}, 13(5):537--551, 1989.
	\newblock \href {https://doi.org/10.1002/jgt.3190130503}
	{\path{doi:10.1002/jgt.3190130503}}.
	
\end{thebibliography}

\end{document}